\documentclass{amsart}

\usepackage{pst-plot}
\usepackage{fullpage}
\usepackage{graphicx} %Formatting for submission
\usepackage{hyperref} %Clickable links
\usepackage{amsmath, amssymb, amsthm}
\usepackage{xcolor} % using for comments
\usepackage{enumitem}
\usepackage{thmtools}
\usepackage{multirow}
\usepackage{booktabs}
\usepackage{subcaption}
\usepackage[style=alphabetic, maxbibnames=9]{biblatex}
\usepackage{tikz}
\usepackage{amsmath,amssymb}
\usepackage{xcolor}

\definecolor{Bittersweet}{rgb}{1.0, 0.44, 0.37}
\definecolor{KateColour}{rgb}{0.3, 0.6, 0.3}%56, 156, 64

\hypersetup{colorlinks=true, urlcolor=purple, citecolor=blue, linkcolor=red, pdfstartview={XYZ null null 0.90}}
\numberwithin{equation}{subsection}
\definecolor{colblue}{HTML}{9fc2f8}

\DeclareMathOperator{\SL}{SL}
\DeclareMathOperator{\GL}{GL}
\DeclareMathOperator{\PSL}{PSL}
\DeclareMathOperator{\PGL}{PGL}
\DeclareMathOperator{\Mat}{Mat}
\DeclareMathOperator{\EisE}{E}
\DeclareMathOperator{\Eis}{Eis}

\DeclareMathOperator{\Inv}{Inv}
\DeclareMathOperator{\Nm}{Nm}
\DeclareMathOperator{\pres}{pres}
\DeclareMathOperator{\Tr}{Tr}

\newcommand{\FF}{\mathbb{F}}
\newcommand{\QQ}{\mathbb{Q}}
\newcommand{\RR}{\mathbb{R}}
\newcommand{\ZZ}{\mathbb{Z}}
\newcommand{\CC}{\mathbb{C}}
\newcommand{\PP}{\mathbb{P}}
\newcommand{\uhp}{\ensuremath{\mathbb{H}}} %Upper half plane
\newcommand{\om}{\omega}
\newcommand{\Egp}{\mathcal{E}}
\newcommand{\ol}[1]{\overline{#1}}
\newcommand{\GammaE}{\Gamma_{\EisE}}
\newcommand{\kron}[2]{\left(\frac{#1}{#2}\right)}%Kronecker symbol
\newcommand{\sm}[4]{\ensuremath{\left(\begin{smallmatrix} #1 & #2\\#3 & #4\end{smallmatrix}\right)}}%Small 2x2 matrix
\newcommand{\lm}[4]{\ensuremath{\left(\begin{matrix} #1 & #2\\#3 & #4\end{matrix}\right)}}%Large 2x2 matrix
\newcommand{\genmtx}{\sm{a}{b}{c}{d}}%Generic matrix [a,b;c,d] as a small matrix

\newcommand{\cir}{\ensuremath{\mathcal{C}}}
\newcommand{\Dcir}{\ensuremath{\mathcal{D}}}
\newcommand{\spor}{\mathcal{S}}
\newcommand{\Epres}{\Egp_{\pres}}

\DeclareRobustCommand{\pint}{
\includegraphics[height=0.9em]{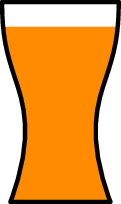}
}
\DeclareRobustCommand{\widepint}{
    \widehat{\pint}
}

\newcommand{\tablestretch}[1]{
\begin{center}
    \renewcommand{\baselinestretch}{1.3}
		{\upshape
		#1
		}
    \renewcommand{\baselinestretch}{1}
\end{center}}%Make the table centred, not affected by being in a theorem, and stretched a bit.

\newtheorem{theorem}{Theorem}[section]
\newtheorem{conjecture}[theorem]{Conjecture}
\newtheorem{corollary}[theorem]{Corollary}
\newtheorem{lemma}[theorem]{Lemma}
\newtheorem{proposition}[theorem]{Proposition}

\theoremstyle{definition}

\newtheorem{definition}[theorem]{Definition}

\newtheorem{remark}[theorem]{Remark}

\begin{document}

\title{Eisenstein circle packings and the Eisenpint Schmidt arrangement}
%\title{Wheretofore be thine cubique obstructions oh lovely flower of the Eisensteins?}

\author{James Rickards}
\address{Saint Mary's University, Halifax, Nova Scotia, Canada}
\email{james.rickards@smu.ca}
\urladdr{https://jamesrickards-canada.github.io/}

\author{Katherine E. Stange}
\address{University of Colorado Boulder, Boulder, Colorado, USA}
\email{kstange@math.colorado.edu}
\urladdr{https://math.katestange.net/}

\date{\today}

\begin{abstract}
The Schmidt arrangement of an imaginary quadratic number field $K$ is the orbit of the extended real line under $\text{PSL}(2, \mathcal{O}_K)$ as M\"obius transformations on the extended complex plane. If $K\neq\mathbb{Q}(\sqrt{-3})$, then the resulting set of circles can only intersect tangentially, leading to various classes of integral circle packings, including Apollonian circle packings.  When $K=\mathbb{Q}(\sqrt{-3})$, circles can intersect at angles of $\frac{\pi}{3}$ and $\frac{2\pi}{3}$, making it unclear how to extract circle packings from the arrangement. The goal of this paper is to study a modification of the $\mathbb{Q}(\sqrt{-3})-$Schmidt arrangement called the \emph{Eisenpint Schmidt arrangement} and associated integral \emph{Eisenstein circle packings}.  In analogy to the study of Apollonian circle packings, we study the number theory of such packings, including associated families of quadratic forms, show the Eisenpint Schmidt arrangement is formed of exactly all primitive Eisenstein circle packings, show strong approximation and classify congruence obstructions, prove a density-one local-global statement, and find quadratic -- but alas no cubic -- reciprocity obstructions.  Unexpected aspects of the Eisenstein case include the role of congruence subgroups, the bipartite nature of the packings and reciprocity obstructions, the coefficients of quadratic obstructions, an abundance of extra symmetry, and the need to use ``first-odd'' quadratic forms. 
\end{abstract}

\maketitle

\section{Introduction}
The $K-$\emph{Schmidt arrangement} of an imaginary quadratic number field $K$ is the orbit of the extended real line $\widehat{\RR}=\RR\cup\{\infty\}$ under the group $\PSL(2, \mathcal{O}_K)$, which creates an arrangement of circles in $\widehat{\CC}$ (Figure~\ref{fig:gaussianschmdit}). Call an individual circle in this orbit a $K-$\emph{Bianchi circle}.

\begin{figure}[ht]
  \centering
  \begin{subfigure}[t]{0.95\textwidth}
    \centering
    \includegraphics[width=\textwidth]{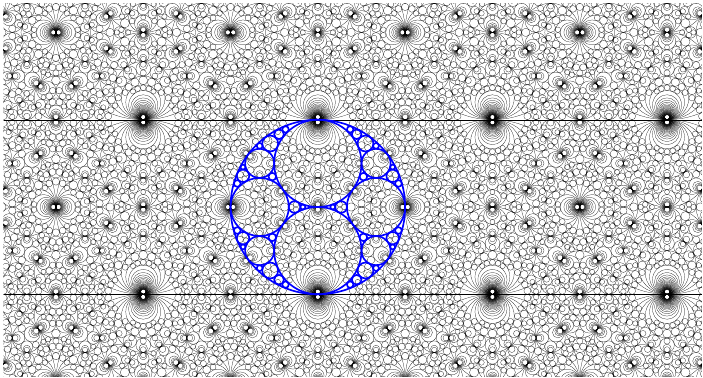}
    \caption{Part of the Gaussian Schmidt arrangement (curvatures at most $70$, i.e. reduced curvatures at most $35$). An immediate tangency packing, i.e. an Apollonian circle packing, is highlighted in blue (in the circle tangent to the real line at $0$).}\label{fig:gaussianschmdit}
  \end{subfigure}

   \vspace{0.5em}

  \begin{subfigure}[t]{0.95\textwidth}
    \centering
    \includegraphics[width=\textwidth]{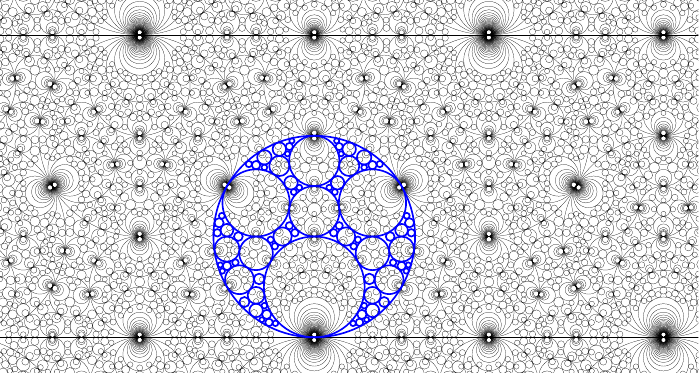}
    \caption{Part of the Eisenpint Schmidt arrangement (curvatures at most $40\sqrt{3}$, i.e. reduced curvatures at most $40$). An immediate tangency packing, i.e. an Eisenstein circle packing, is highlighted in blue (in the circle tangent to the real line at $0$).}\label{fig:pintschmidt}
  \end{subfigure}
      \caption{Gaussian and Eisenpint Schmidt arrangements and packings.}
\end{figure}

Remarkably, as long as $K\neq\QQ(\sqrt{-3})$, any two $K-$Bianchi circles can only intersect tangentially, giving this orbit a packing-like structure (\cite[Proposition 4.4]{StangeVisualizingArith}). Many other intriguing arithmetic properties hold:  the intersection points of any two circles lie in $K$ (\cite[Proposition 4.1]{StangeVisualizingArith}), and the orbit is connected if and only if $\mathcal{O}_K$ is Euclidean (\cite[Theorem 7.1]{StangeVisualizingArith}). More general arrangements were also studied in \cite{MartinGeneralSchmdit}, where an entire family of arrangements was associated to each $K$ (many of which include circles with non-tangent intersections).

From $K-$Schmidt arrangements one can define \emph{$K$-Apollonian packings} \cite{StangeApollonianStructureBianchi}, generalizing integral Apollonian circle packings (the case of $K = \QQ(i)$).  Apollonian circle packings (not necessarily integral) have been around long enough that their origin is murky, but their arithmetic properties have been of interest in recent years since an investigation of Graham, Lagarias, Mallows, Wilks and Yan \cite{GLMWY02}.  In particular, an integral Apollonian circle packing is one in which all the curvatures are integral.  It has been shown that the integers satisfy congruence conditions \cite{GLMWY02, FuchsStrongApproximation, HKRS23}, and conjectured (the \emph{local-to-global conjecture}) that all but finitely many integers obeying these conditions will appear \cite{GLMWY02, FS11}.  A density-one subset of such are known to appear \cite{BK14}, and recently a thin but infinite set of exceptions were discovered in many (but not all) packings, under the name \emph{reciprocity obstructions} \cite{HKRS23}.  
All integral Apollonian circle packings can be recovered from the $\QQ(i)-$Schmidt arrangement by growing a collection of circles by tangent neighbours.  

However, the original work of \cite{StangeApollonianStructureBianchi} avoids the case $K=\QQ(\sqrt{-3})$.  In this field, we immediately run into problems: circles may intersect at angles of $\frac{\pi}{3}$ and $\frac{2\pi}{3}$!  Extracting a sub-arrangement containing only tangency intersections turns out to be challenging.

A primary goal of this paper is to study a suitable adjustment of the theory for $\QQ(\sqrt{-3})$.  The basic circle packing definition via geometric configurations and circle swaps, an \emph{Eisenstein circle packing}, was originally discovered by Andrew Jensen, Cherry Ng, Evan Oliver and Tyler Schrock as part of an undergraduate/graduate research project held at the University of Colorado, Boulder in the summer of 2015, under the direction of the second author, but essentially independently rediscovered by the first author, which lends weight to the belief that this is the ``natural'' packing hidden in $\QQ(\sqrt{-3})$.  Oliver went on to write his undergraduate honors thesis describing them \cite{OliverEvanThesisEisenstein}.  In particular, the team described the same swap structure we use, but with reference to a cluster of six circles, instead of the four we use here.

The {Eisenstein circle packing} will be the suitable replacement of Apollonian circle packings in the Eisenstein setting. We will examine the geometry, reduction, and basic number theoretic properties of Eisenstein circle packings, including a formula for the number of primitive integral packings with a given outer curvature, and a somewhat twisted analogue to the crucial relationship between Apollonian circle packings and binary quadratic forms. 
Eisenstein circle packings are examples of superintegral crystallographic circle packings \cite[Definitions 2 and 10]{KontorovichNakamuraCrystallographic} (see Remark~\ref{rmk:crystl} and Figure~\ref{fig:stripbase}), although the Eisenstein strip packing does not appear on the list of examples yet known \cite{ChaitRothCuiStierTaxonomyCrystallographic, CrystallographicDatabase}.

\begin{figure}[htb]
    \includegraphics[scale=0.8]{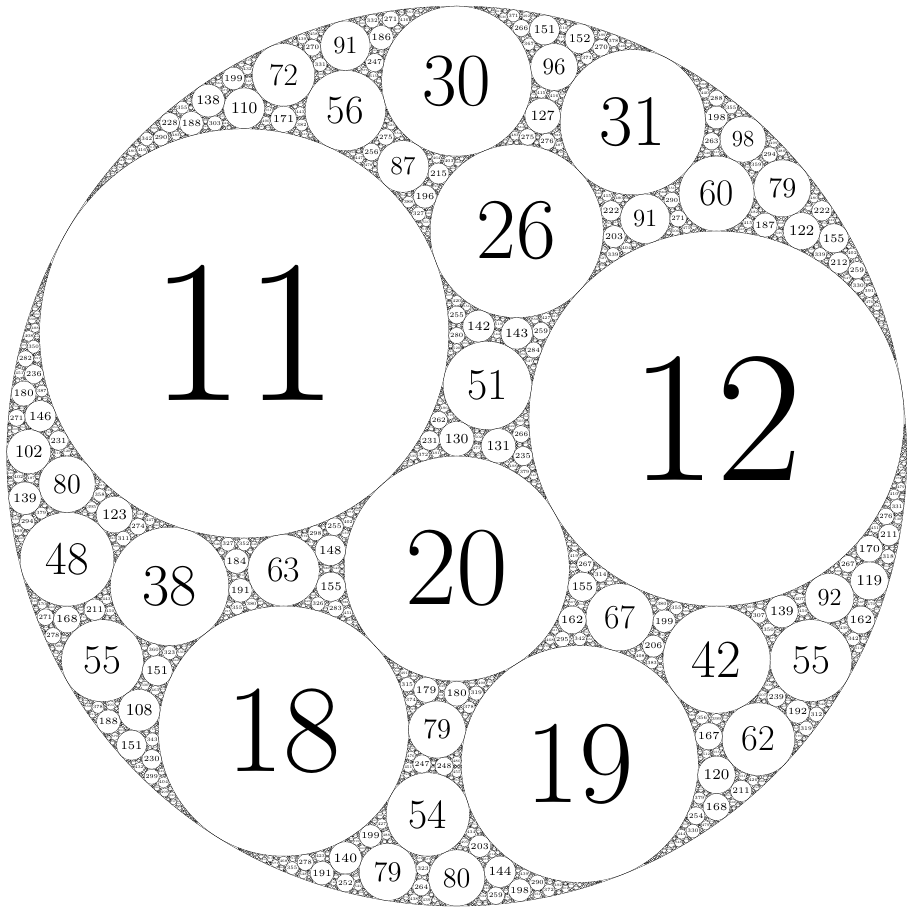}
    \caption{Circles in the Eisenstein circle packing with root quadruple $(-5, 11, 20, 12)$. The 5287 circles with curvature at most $4000$ are shown.}\label{fig:m5_11_20_12_to4000}
\end{figure}

Next, we connect this definition to the $\QQ(\sqrt{-3})-$Schmidt arrangement.  We modify the $\QQ(\sqrt{-3})-$Schmidt arrangement to produce the \emph{Eisenpint Schmidt arrangement}, which assembles all primitive Eisenstein packings into one picture with only tangential intersections.  The Eisenpint Schmidt arrangement was also discovered experimentally by the REU group of 2015, as an orbit of a congruence subgroup.  Interestingly, several interacting scaled, translated and rotated copies of the Eisenpint arrangement lie in the $\QQ(\sqrt{-3})-$Schmidt arrangement.  We study these arrangements, generalizing work of \cite{StangeApollonianStructureBianchi} describing how they capture all integral Eisenstein circle packings.

We end by considering the finer number theoretic properties of Eisenstein circle packings. This includes (again, in analogy to the story for Apollonian circle packings), determining the bad modulus coming from strong approximation, as well as considering the local-global conjecture for the curvatures. Similarly to the Apollonian, octahedral, cubic, square, and triangular circle packings \cite{HKRS23, WhiteheadEtAlReciprocity}, we find quadratic reciprocity obstructions.  Alas, cubic obstructions elude us. 

In the remainder of the introduction we provide the main definitions and statements, while proofs and further details follow in later sections.

\subsection{Eisenstein circle packings}

An \emph{oriented circle} in $\widehat{\CC} := \CC \cup \{ \infty \}$ is a circle or straight line with an orientation, where the region to the left is designated as the interior. Oriented circles can be represented as column vectors $(u, v, p, q)\in\RR^4$, where $u$ is the curvature, $v$ is the co-curvature, and $p+qi$ is the curvature-centre (see Section~\ref{sec:pedoedefinition}). The notion of \emph{inversive distance} between two oriented circles extends to a symmetric bilinear form $\langle,\rangle$ on $\RR^4$ (Definition~\ref{def:pedoedistance}), where a vector $\cir\in\RR^4$ is an oriented circle if and only if $\langle \cir, \cir\rangle = 1$. Crucially, if $\cir_1$ and $\cir_2$ are oriented circles whose exteriors contain $\infty$, then $\langle \cir_1, \cir_2\rangle \leq-1$ if and only if their interiors are disjoint. Equality with $-1$ is equivalent to them being externally tangent.

\begin{definition}
    A \emph{4-wheel} is 4-tuple $ W = (\cir_1, \cir_2, \cir_3, \cir_4)$ of oriented circles such that
    \[\Inv(\cir_1, \cir_2, \cir_3, \cir_4) := \left(\begin{matrix} \langle \cir_1, \cir_1\rangle & \langle \cir_1, \cir_2\rangle & \langle \cir_1, \cir_3\rangle& \langle \cir_1, \cir_4\rangle\\ \langle \cir_2, \cir_1\rangle & \langle \cir_2, \cir_2\rangle & \langle \cir_2, \cir_3\rangle & \langle \cir_2, \cir_4\rangle\\ \langle \cir_3, \cir_1\rangle & \langle \cir_3, \cir_2\rangle & \langle \cir_3, \cir_3\rangle & \langle \cir_3, \cir_4\rangle\\ \langle \cir_4, \cir_1\rangle & \langle \cir_4, \cir_2\rangle & \langle \cir_4, \cir_3\rangle & \langle \cir_4, \cir_4\rangle\end{matrix}\right) = \left(\begin{matrix}1 & -1 & -2 & -1\\-1 & 1 & -1 & -2\\-2 & -1 & 1 & -1 \\ -1 & -2 & -1 & 1\end{matrix}\right)=:R_{\Eis},\]
    and the sum of the four curvatures is positive.

    The matrix $\Inv(\cir_1, \cir_2, \cir_3, \cir_4)$ is called the \emph{inversive matrix} of our tuple of circles. We sometimes consider the tuple $W$ itself as a $4\times 4$ matrix, where the columns represent the circles as column vectors in $\RR^4$.
\end{definition}

These circles have disjoint interiors and form a tangency cycle in the order $\cir_1\rightarrow \cir_2\rightarrow \cir_3\rightarrow \cir_4\rightarrow \cir_1$. Call $\cir_1$ and $\cir_3$ \emph{opposite circles}, as well as $\cir_2$ and $\cir_4$.

Up to M\"obius symmetry and complex conjugation, there is only one 4-wheel.

\begin{proposition}\label{prop:4wheelmobius}
    Given two 4-wheels $W$ and $W'$, there exists a M\"obius map taking $W$ to either $W'$ or the image of $W'$ upon complex conjugation.
\end{proposition}

The bilinear form $\langle , \rangle$ gives rise to circle swaps, which build out a circle packing, analogously to Apollonian circle packings.

\begin{definition}\label{def:circleswap}
    Given a 4-wheel $W = (\cir_1, \cir_2, \cir_3, \cir_4)$, there are four circle swaps. To construct the $i$\textsuperscript{th} swap, consider the circle $\cir'$ that is orthogonal to the other three circles in the wheel (besides $\cir_i$). The image of $\cir_i$ upon inversion through $\cir'$ gives $\cir_i'$, the swap of $\cir_i$ with respect to $W$.  See Figure~\ref{fig:swapC3C4} for depictions of the swaps.
\end{definition}

Since M\"obius maps and inversions preserve the inversive distance, and the above inversion fixes the circles besides $\cir_i$, this implies that replacing $\cir_i$ with $\cir_i'$ in the 4-wheel $W$ produces a new 4-wheel. 

\begin{figure}[htb]
	\centering
	\begin{subfigure}{.5\textwidth}
		\centering
		\includegraphics[width=.9\linewidth]{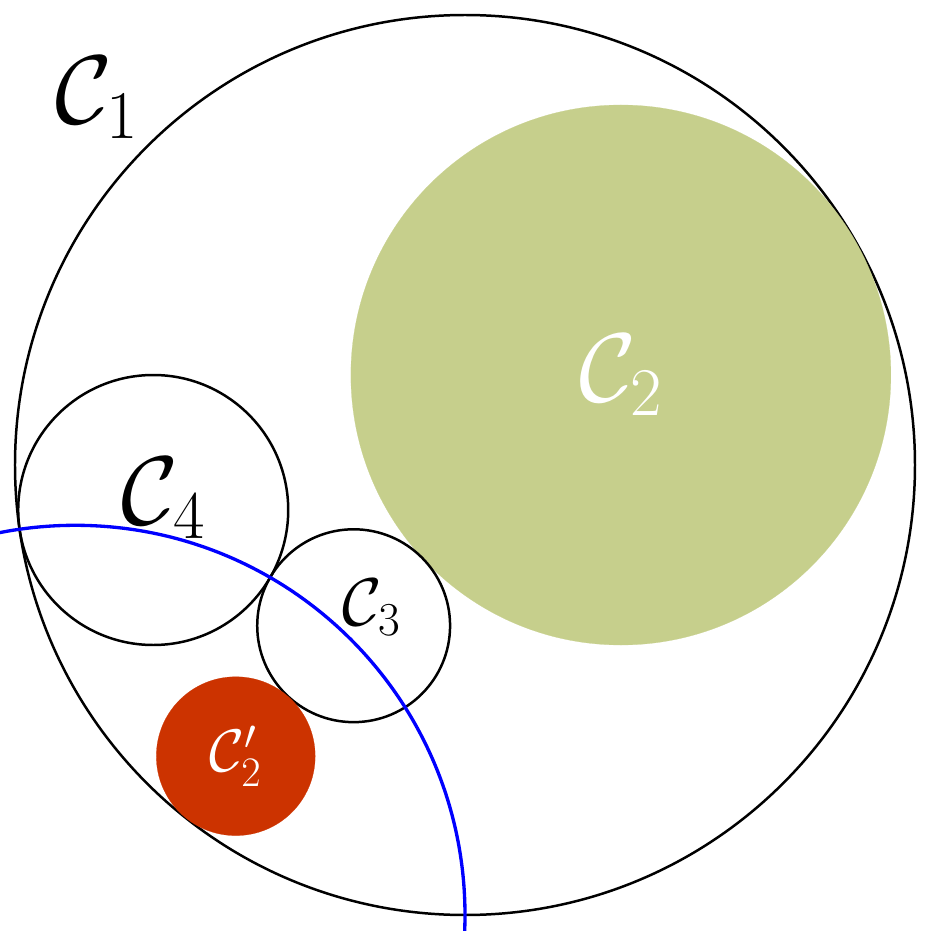}
		\caption{Swapping $\cir_2$ for $\cir_2'$ via inversion in the blue circle $\cir'$, which is orthogonal to $\cir_1, \cir_3, \cir_4$.}\label{fig:swapc2}
	\end{subfigure}%
	\begin{subfigure}{.5\textwidth}
		\centering
		\includegraphics[width=.9\linewidth]{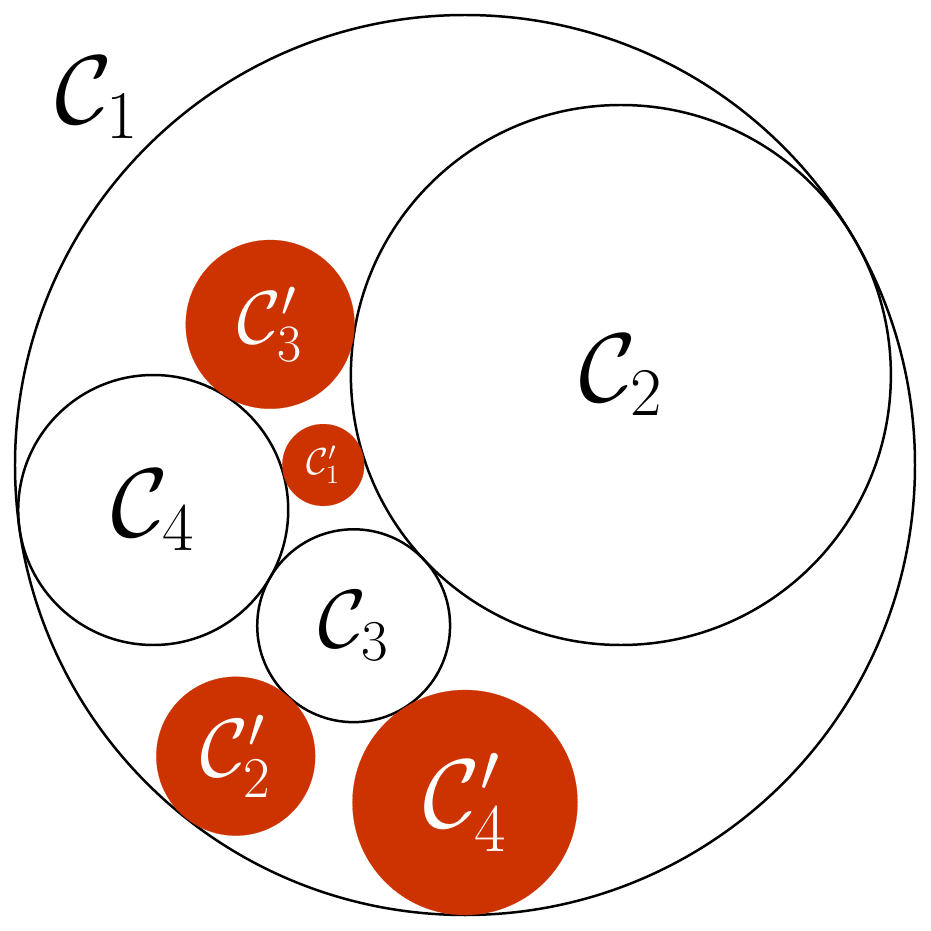}
		\caption{All four swaps. \\$\,$}\label{fig:allfourswaps}
	\end{subfigure}
	\caption{The swaps of the 4-wheel $(\cir_1, \cir_2, \cir_3, \cir_4)$.}\label{fig:swapC3C4}
\end{figure}
%The mysterious \\$\,$ line in the second caption is to make the pictures aligned horizontally.

In Figure~\ref{fig:allfourswaps}, ignoring $\cir_1'$ and $\cir_3'$, we see a picture that resembles a wheel, with $\cir_1$ and $\cir_3$ the \emph{hubs}, and $\cir_2, \cir_4, \cir_2', \cir_4'$ the \emph{spokes}. Similarly, if we ignore $\cir_2',\cir_4'$, we are treating $\cir_2, \cir_4$ as the hubs and $\cir_1, \cir_3, \cir_1', \cir_3'$ as the spokes. This explains the nomenclature of \emph{4-wheel}.

The swapping can also be captured by a linear map on the 4-wheel $W$.

\begin{proposition}\label{prop:circleswapmatrices}
    Swapping circle $\cir_i$ is equivalent to replacing $W$ with $WS_i$, where
    \[S_1 = \left(\begin{matrix} -1 & 0 & 0 & 0\\2 & 1 & 0 & 0\\0 & 0 & 1 & 0\\2 & 0 & 0 & 1\end{matrix}\right),\quad S_2 = \left(\begin{matrix} 1 & 2 & 0 & 0\\0 & -1 & 0 & 0\\0 & 2 & 1 & 0\\0 & 0 & 0 & 1\end{matrix}\right),\quad S_3 = \left(\begin{matrix} 1 & 0 & 0 & 0\\0 & 1 & 2 & 0\\0 & 0 & -1 & 0\\0 & 0 & 2 & 1\end{matrix}\right),\quad S_4 = \left(\begin{matrix} 1 & 0 & 0 & 2\\0 & 1 & 0 & 0\\0 & 0 & 1 & 2\\0 & 0 & 0 & -1\end{matrix}\right).\]
\end{proposition}

By iterating the circle swaps, we produce a tangency-connected collection of circles, where every pair of interiors is disjoint, and the measure of the complement is 0.

\begin{definition}\label{def:eisensteinpacking}
    An \emph{Eisenstein circle packing} is a collection $E$ of (oriented) circles in $\widehat{\CC}$ obtained by the following process:
    \begin{itemize}
        \item Start with a 4-wheel $W$;
        \item Apply all four circle swaps to $W$, creating new 4-wheels, and add these new circles to the collection;
        \item Repeat the process with any new $4$-wheels, ad infinitum.
    \end{itemize}

    Call the packing \emph{integral} if all curvatures are integral, and \emph{primitive} if they additionally share no common factor.
\end{definition}

See Figure~\ref{fig:m5_11_20_12_to4000} for an example of a primitive Eisenstein circle packing, where circles are labeled by curvature.

To describe the geometry of an Eisenstein circle packing, the data beyond the curvatures is extraneous.

\begin{definition}\label{def:eisensteinequation}
    Any quadruple $(a, b, c, d)\in\RR^4$ that satisfies the \emph{Eisenstein equation},
    \begin{equation}
    \label{eqn:Eisensteinform}
    a^2+b^2+c^2+d^2=2(a+c)(b+d),\quad a+b+c+d>0,
    \end{equation}
    is called an \emph{Eisenstein quadruple}. If $a,b,c,d\in\ZZ$ have no common divisor, we call it a \emph{primitive Eisenstein quadruple}.
\end{definition}

Eisenstein quadruples biject with 4-wheels up to symmetry.

\begin{proposition}\label{prop:eisquadruplefrom4wheel}
    Consider a 4-wheel $(\cir_1, \cir_2, \cir_3, \cir_4)$ where the circles have respective curvatures $a, b, c, d$. Then $(a, b, c, d)$ is an Eisenstein quadruple.

    Furthermore, given any Eisenstein quadruple, there exists a 4-wheel with those respective curvatures. This configuration is unique up to translation, rotation, and reflection.
\end{proposition}

We will restrict our study to Eisenstein quadruples. Considering them as row vectors, the swaps act in the same fashion, via right multiplication by $S_i$.

\begin{definition}
\label{def:Egp}
     The \emph{Eisenstein swap group} is $\Egp:=\langle S_1, S_2, S_3, S_4\rangle\leq \GL(4, \ZZ)$.
\end{definition}

The four generators have order 2, but do not generate a free group, since the pairs $(S_1, S_3)$ commute, as well as $(S_2, S_4)$. In fact, these are the only relations.

\begin{proposition}\label{prop:eisensteingrouppresentation}
    The Eisenstein swap group $\Egp$ is isomorphic to
    \[\langle a, b, c, d \; | \; aca^{-1}c^{-1}, bdb^{-1}d^{-1},a^2,b^2,c^2,d^2\rangle.\]
\end{proposition}

Going back to Eisenstein quadruples, the four circle swaps have the following effect on curvatures:
\begin{align*}
        S_1: (a, b, c, d) & \rightarrow (2b+2d-a, b, c, d),\\
        S_2: (a, b, c, d) & \rightarrow (a, 2a+2c-b, c, d),\\
        S_3: (a, b, c, d) & \rightarrow (a, b, 2b+2d-c, d),\\
        S_4: (a, b, c, d) & \rightarrow (a, b, c, 2a+2c-d).
\end{align*}

The multiset of quadruples formed by iterating the swaps from an initial quadruple $q$ (as row vector) is given by the orbit $q\Egp$, and it consists of all 4-wheels in the entire Eisenstein circle packing described in Definition~\ref{def:eisensteinpacking}.

\begin{definition}
\label{def:equivalentquadruple}
    Call any quadruple reachable from $q$ via a series of circle swaps \emph{equivalent} to $q$.
\end{definition}

In Section~\ref{sec:reductiontheory}, we develop the reduction theory for Eisenstein quadruples. The main definition and results are as follows.

\begin{definition}
    An Eisenstein quadruple $(a, b, c, d)$ is called \emph{reduced} if none of the four swaps decrease the sum $a+b+c+d$.
\end{definition}

\begin{proposition}\label{prop:atmostonereducedintegralone}
    Let $q$ be an Eisenstein quadruple. Then, there is at most one reduced quadruple equivalent to $q$. If $q$ is integral, there is exactly one.
\end{proposition}

Three of the smallest curvatures are described by a reduced quadruple.

\begin{proposition}\label{prop:threesmallestcurvatures}
    Let $(a, b, c, d)$ be a reduced Eisenstein quadruple, and assume that $a=\min\{a, b, c, d\}$. Then, $a\leq 0$, and the three smallest curvatures in the packing are $a, b, d$.
\end{proposition}

As evidenced by Figure~\ref{fig:m5_11_20_12_to4000}, curvature $c$ need not be one of the four smallest curvatures.

\subsection{Primitive Eisenstein circle packings}

\begin{definition}
    A primitive Eisenstein quadruple $(a, b, c, d)$ is in \emph{standard position} if $a\equiv b\pmod{2}$ and $c\equiv d\pmod{2}$.
\end{definition}

By possibly swapping $(b, d)$, every primitive quadruple can be put into standard position. This gives us a unique way of identifying primitive circle packings.

\begin{proposition}\label{prop:eisensteinreduced}
    Every primitive Eisenstein packing corresponds to a unique reduced quadruple $(a, b, c, d)$ in standard position with $a\leq 0$. This is called the \emph{root quadruple} of the packing.
\end{proposition}

Analogously to Apollonian circle packings, there exists a bijection between primitive Eisenstein circle packings and certain classes of quadratic forms (Section~\ref{sec:quadraticforms}). This enables efficient computation of all primitive Eisenstein packings containing a circle of curvature $n$, as well as a formula for the total count.

\begin{theorem}\label{thm:countcirclepackingsoutercurvn}
    Let $\omega(n)$ denote the number of distinct prime factors of $n$, and let $\left(\frac{-3}{p}\right)\in\{-1, 0, 1\}$ denote the Kronecker symbol, which satisfies $\left(\frac{-3}{p}\right)\equiv p\pmod{3}$. Then the number of primitive Eisenstein circle packings with outer curvature $-n$ is
    \[\frac{n}{2}\prod_{p\mid n}\left(1-\left(\frac{-3}{p}\right)\frac{1}{p}\right) + 2^{\omega(3n)-2}\]
    if $n$ is odd, and
    \[\frac{n}{3}\prod_{p\mid n}\left(1-\left(\frac{-3}{p}\right)\frac{1}{p}\right) + 2^{\omega{(3n/2)}-1}\]
    if $n$ is even.
\end{theorem}

\subsection{The Eisenpint Schmidt arrangement}
Before investigating the local-global conjecture, we go back to the $\QQ(\sqrt{-3})-$Schmidt arrangement, adjust it suitably, and verify that this gives rise to (scaled) primitive Eisenstein circle packings.

Recall that general $K-$Apollonian packings are constructed in a ``greedy'' manner: start with one circle in the $K-$Schmidt arrangement, adjoin the largest externally tangent circle at each tangency point (i.e. the \emph{immediately tangent} circle), and repeat (see \cite[Definition 4.7]{StangeApollonianStructureBianchi}).

When $K=\QQ(\sqrt{-3})$, we can still attempt this process, as it is not disrupted by the presence of non-tangent circles. However, this too fails: you can follow paths of immediate tangencies that loop back on themselves to intersect non-tangentially! This is demonstrated in Figure~\ref{fig:schmidt-full}.

\begin{figure}[htb]
    \includegraphics[scale=2.9]{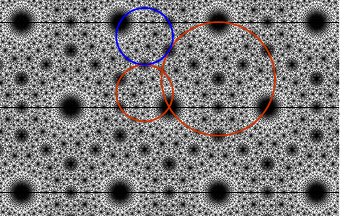}
    \caption{Part of the full Schmidt arrangement for $K=\QQ(\sqrt{-3})$ (curvatures at most 20). Both red circles are immediately tangent to the blue circle, yet the red circles intersect non-tangentially.}\label{fig:schmidt-full}
\end{figure}

The root cause is the presence of extra units. Let $\om=\frac{1+\sqrt{-3}}{2}$ denote the upper half plane root of the equation $x^2-x+1=0$, a primitive sixth root of unity. Then, the \emph{Eisenstein integers} are $\ZZ[\om]$, which is the maximal order of $\QQ(\sqrt{-3})$. The Schmidt arrangement naturally has a translation by $\ZZ[\om]$ symmetry, but these units also give rise to rotations by $\frac{2\pi}{3}$ and $\frac{4\pi}{3}$, which in turn produce non-tangential intersections. Thus, in order to extract a reasonable analogue of other $K-$Schmidt arrangements, one should attempt to remove these extra symmetries. To this end, take $\GammaE$ to be the matrices of $\PSL(2, \ZZ[\om])$ which, modulo 2, have unit diagonal and are lower triangular, i.e.
\[\GammaE:=\{M\in\PSL(2, \ZZ[\om]):M\equiv\sm{1}{0}{\ast}{1}\pmod{2}\}.\]
This group has index 15 in $\PSL(2, \ZZ[\om])$, and we call it the \emph{Eisenpint\footnote{A traditional beer stein holds 1 litre, or about 1.76 imperial pints.} group}.

\begin{definition}
    The \emph{Eisenpint Schmidt arrangement} for $\QQ(\sqrt{-3})$ is the orbit of $\widehat{\RR}$ under $\GammaE$, where the circles are taken to be unoriented. A circle in this orbit is called an \emph{Eisenpint circle}, and the collection of Eisenpint circles is denoted $\pint$. Considering each circle with both orientations gives the set $\widepint$.
\end{definition}

The Eisenpint Schmidt arrangement (Figure~\ref{fig:pintschmidt}) recovers the most familiar properties of Schmidt arrangements for other fields, whilst constituting a large part of the $\QQ(\sqrt{-3})-$Schmidt arrangement.

\begin{theorem}\label{thm:pintschmditiscanonical}
    Circles in the Eisenpint Schmidt arrangement only intersect tangentially. Additionally, given any pair of tangent circles in the $\QQ(\sqrt{-3})-$Schmidt arrangement, there exists a combination of a rotation about the origin (by $0$, $\frac{2\pi}{3}$ or $\frac{4\pi}{3}$), a translation (by $0$, $1$, $\om$, or $\om+1$), and a dilation (by $1$ or $\frac{1}{2}$) of the Eisenpint Schmidt arrangement such that the image is again a subset of the $\QQ(\sqrt{-3})-$Schmidt arrangement, and contains the intersecting circles in question.
\end{theorem}

Roughly speaking, Theorem~\ref{thm:pintschmditiscanonical} shows that we can recover the full $\QQ(\sqrt{-3})-$Schmidt arrangement through basic rigid motions of the Eisenpint Schmidt arrangement.

At last, we can demonstrate that the Eisenpint Schmidt arrangement describes all primitive Eisenstein circle packings. In particular, this prescribes a unique (up to symmetry) placement and orientation in $\widehat{\CC}$ for every primitive Eisenstein circle packing.

\begin{definition}\label{def:eisensteinschmidtpacking}
    A collection of circles $\mathcal{P}$ \emph{straddles} a circle $\cir$ if it intersects both the interior and exterior of $\cir$. We call $\mathcal{P}$ \emph{tangency-connected} if the graph of tangencies is connected. An \emph{Eisenpint tangency packing} is a maximal tangency-connected subset $\mathcal{P}$ of oriented circles from $\widepint$ such that $\mathcal{P}$ does not straddle any circle in $\widepint$, and the interiors of $\mathcal{P}$ are disjoint.
\end{definition}

It is clear from the definition that Eisenpint tangency packings are obtained by picking any circle in $\widepint$, greedily adding the immediate tangency circles, and repeating the process ad infinitum. In particular, every circle in $\pint$ is a part of exactly two packings: one per orientation.

Symmetries of the Eisenpint Schmidt arrangement imply some repetition of packings.

\begin{definition}
    Declare two Eisenpint tangency packings as \emph{equivalent} if they are related by a sequence of translations by $2\ZZ[\om]$, rotations by $\pi$ about the origin, and reflections across the imaginary axis.
\end{definition}

These equivalences give rise to the bijection with primitive Eisenstein circle packings, thereby providing a canonical (up to the above equivalence) way of placing a primitive Eisenstein circle packing in the plane.

\begin{theorem}\label{thm:Qrt3packingsareEisenstein}
    Every Eisenpint tangency packing is a primitive Eisenstein circle packing scaled by $\frac{1}{\sqrt{3}}$. This association bijects equivalence classes of Eisenpint tangency packings with primitive Eisenstein circle packings.
\end{theorem}

Figure~\ref{fig:pintschmidt} demonstrates an Eisenpint tangency packing inside the whole Eisenpint Schmidt arrangement. This packing corresponds to the Eisenstein circle packing with root quadruple $(-1, 3, 4, 2)$.

\subsection{Local-global aspects of primitive Eisenstein circle packings}

The tangency graph of an Eisenstein circle packing is bipartite (see Figure~\ref{fig:m4_12_19_7_chi2}).

\begin{definition}\label{def:moiety}
    Call each of the bipartite halves a \emph{moiety\footnote{``Moiety'' means \emph{one of two halves of a whole}, from the french \emph{moiti\'e.}}} of the packing.
\end{definition}

Considering the packing as generated by a 4-wheel (or Eisenstein quadruple) and circle swaps, one moiety corresponds to the circles which appear in either the first or third position, and the other is the circles appearing in the second or fourth positions.

We will study the sets of curvatures of the circles in each moiety of the packing separately. This also makes sense from a thin group point of view: the multiset of curvatures in the whole packing is a union of four linear functionals on $(a, b, c, d)\cdot\Egp$: one for each index of the vector. By studying a moiety, we are only unioning two of these at a time, which reveals certain interesting behaviours that get washed out upon combining them. It would be reasonable to go one step further and study each index independently, however, this is unnecessary. Indeed, in a primitive Eisenstein circle packing in standard position, $a-c\equiv b-d\equiv 1\pmod{2}$, so the $a-$orbit and $c-$orbit (and $b-$ orbit and $d-$orbit) are disjoint.

\begin{definition}
    Let $S$ be a subset of the circles in a primitive Eisenstein circle packing. Denote by $c(S)$ the multiset of curvatures found in $S$.
\end{definition}

As with Apollonian circle packings, the curvatures have modular restrictions. This time, they occur modulo 4.

\begin{proposition}\label{prop:curvaturesmod4}
    Let $E$ be either all or a moiety of the circles in a primitive Eisenstein circle packing. Then, the set $c(E)\pmod{4}$ is either $\{0, 1, 2\}$ or $\{0, 2, 3\}$. In particular, this set is identified by whether the odd curvatures are all $1\pmod{4}$ or all $3\pmod{4}$.
\end{proposition}

This modular behaviour is captured by the \emph{type} of the packing or moiety.

\begin{definition}
    The type of $E$ is defined by
    \begin{center}
		{\upshape
        \renewcommand{\arraystretch}{1.2}
		\begin{tabular}{|c|c|} 
			\hline
			Type      & $c(E)\pmod{4}$ \\ \hline
			$(3, 1)$  & $0, 1, 2$\\ \hline
			$(3, 3)$  & $0, 2, 3$\\ \hline
		\end{tabular}
		}
	\end{center}
    Any positive integer $n$ equivalent modulo $4$ to an element of $c(E)\pmod{4}$ is deemed \emph{admissible}.
\end{definition}

Analogously to the Apollonian setting (from \cite{HKRS23}), type $(a, b)$ denotes that there are $a$ residues modulo $4$, with $b$ being the smallest one coprime to $4$. The type of one moiety of a packing is equal to the type of the other: the type is well-defined on the whole packing.

We may ask if the analogue of the local-global conjecture is true for primitive Eisenstein circle packings: must every sufficiently large admissible integer occur in the packing?

To begin, we apply the results of \cite{FSZ19} to conclude that a density one local-to-global statement holds for Eisenstein packings. This requires an analysis of the strong approximation of a subgroup of $\SL(2,\ZZ[\omega])$ governing the packing. The group we define here differs slightly from the groups already used to describe the packing, because the machinery of \cite{FSZ19} has certain restrictions. In particular, it is a Kleinian group whose orbit on a collection of circles tangent to the real line gives the Eisenstein strip packing. This is in contrast to the abstract swap group, which acts on quadruples of circles or curvatures. 
To be precise, we define
\[
T: z \mapsto z + 2, \quad
S: z \mapsto z/(z+1), \quad
V: z \mapsto -z + 1 + \sqrt{-3},
\]
and
\[
\Egp' := \langle S, T, V \rangle, \quad
\Egp'' := \Egp' \cap \SL(2,\ZZ[w]]).
\]
The following is an introductory version of Proposition~\ref{prop:orbitsE}.
\begin{proposition}
\label{prop:orbitsEintro}
    The circles of the Eisenstein strip packing are made up of two orbits of $\Egp'$ or four orbits of $\Egp''$.
\end{proposition}

We prove that $\Egp''$ has strong approximation, and as a consequence derive a density one local-to-global statement.

\begin{theorem}
\label{thm:densityone}
    Any primitive Eisenstein circle packing or a moiety of such satisfies a finite congruence condition and has a density one set of curvatures amongst the curvatures allowed by its type.
\end{theorem}

The finite congruence condition given by the analysis is modulo 16. However, it turns out that the set of admissible curvatures always descends down to modulo 4, giving the type as above (see Remark~\ref{rmk:SA}).

Despite Theorem~\ref{thm:densityone}, the local-global conjecture does not hold.  To this end, we define the $\chi_2$ function on any circle in a primitive Eisenstein circle packing.

\begin{definition}\label{def:chi2fromtangentcircle}
    Let $\cir$ be a circle of curvature $n$ in a primitive Eisenstein circle packing of type $(3, t)$, and let it be tangent to a circle of coprime curvature $b$. Define
    \[\chi_2(\cir):=\begin{cases}
        \kron{b}{n} & \text{if $n$ and $b$ are odd;}\\
        \kron{2b}{n} & \text{if $n$ is odd and $b$ is even;}\\
        \kron{b}{n/2} & \text{if $n$ is even and $t=1$;}\\
        -\kron{-b}{n/2} & \text{if $n$ is even and $t=3$;}
    \end{cases}\]
\end{definition}

This gives a well-defined invariant on each moiety! 

\begin{proposition}\label{prop:chi2invariance}
    Let $E$ be a primitive Eisenstein circle packing, which divides into moieties $E_1\cup E_2$. Then, the function $\chi_2$ is well-defined on all circles, and constant on each moiety $E_1$ and $E_2$, giving a well defined notion of $\chi_2(E_1)$ and $\chi_2(E_2)$. If $E$ has type $(3, 1)$, then $\chi_2(E_1)=\chi_2(E_2)$, hence $\chi_2$ is well-defined on the entire packing. If $E$ has type $(3, 3)$, then $\chi_2(E_1)=-\chi_2(E_2)$, hence $\chi_2$ alternates across tangencies.
\end{proposition}

A packing of type $(3, 3)$ with alternating $\chi_2$ values is demonstrated in Figure~\ref{fig:m4_12_19_7_chi2}.

\begin{figure}[htb]
    \includegraphics[scale=0.8]{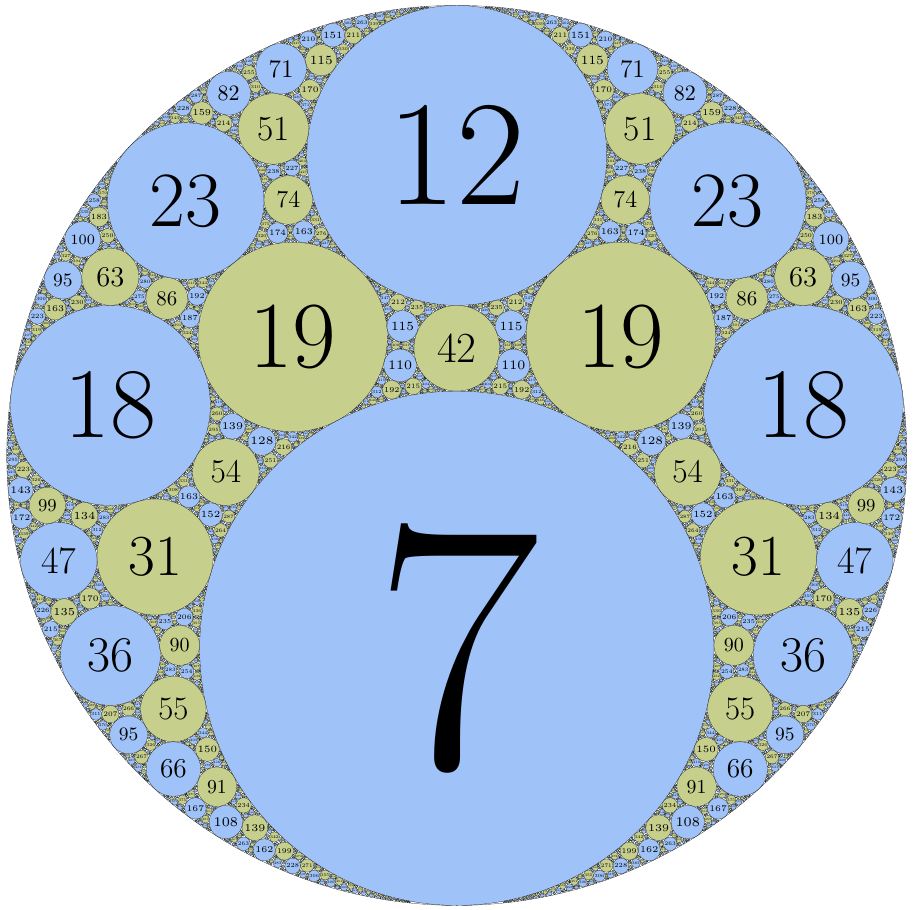}
    \caption{Circles in the Eisenstein circle packing with root quadruple $(-4, 12, 19, 7)$. The 4718 circles with curvature at most $3000$ are shown. Circles in green have $\chi_2=1$, and circles in blue have $\chi_2=-1$. These are the two moieties of this packing.}\label{fig:m4_12_19_7_chi2}
\end{figure}

By incorporating the value of $\chi_2$ into the type, we get an extended type of a moiety.

\begin{definition}
    Let $E$ be a moiety of a packing of type $(3, t)$. The \emph{extended type} of $E$ is the triple $(3, t, \chi_2(E))$. If $t=1$, then this extends to an extended type for the whole packing.
\end{definition}

We can now fully state the reciprocity obstructions for half packings.

\begin{theorem}\label{thm:reciprocity}
    Let $E$ be a moiety of a primitive Eisenstein circle packing. Then, as long as $E$ does not have extended type $(3, 1, 1)$, there exists (infinite) quadratic families of admissible curvatures which do not occur in $c(E)$. A precise description of the families is found in Table~\ref{table:reciprocity}.
\end{theorem}

\begin{table}[hbt]
\centering
    {
    \renewcommand{\arraystretch}{1.2}
    \begin{tabular}{|l|c|}
        \hline
        Extended Type & Quadratic Obstructions\\ \hline
        $(3, 1, 1)$ & \\ \hline
        $(3, 1, -1)$ & $(2n+1)^2, 2n^2, 6n^2$\\ \hline
        $(3, 3, 1)$ & $2n^2, 3(2n+1)^2$\\ \hline
        $(3, 3, -1)$ & $6n^2$\\ \hline
    \end{tabular}
    }
    \caption{Reciprocity obstructions in a moiety of a primitive Eisenstein circle packing.}\label{table:reciprocity}
\end{table}

Note that the obstructions in type $(3, 3, \pm 1)$ are complementary, and thus ``wash each other out'' when you pass to the entire primitive Eisenstein circle packing. On the other hand, the obstructions in type $(3, 1, -1)$ remain in the full packing.

This behaviour is analogous to the cases studied in \cite{WhiteheadEtAlReciprocity}. Only with certain congruence conditions do they find full reciprocity obstructions for the entire set of circles in octahedral, cubic, triangular, and square circle packings. However, for each of the four packing types, they still recover partial obstructions for a moiety of the circles: the patterns of $\chi_2$ behave in a predictable way. In future works studying the local-global question for circle packings, one should be aware of the likely need to partition the set of circles into two (or more) moieties to obtain finer reciprocity obstructions.

Another point of note is we have reciprocity obstructions of the form $6n^2$ and $3(2n+1)^2$, despite there being no congruence obstructions at $3$. This is in contrast to the previously known examples of reciprocity obstructions. In the Apollonian case \cite{HKRS23}, the reciprocity obstructions are of the form $cn^k$ where $c$ is supported on $2,3$, both of which contribute congruence obstructions as well. Similarly, in the octahedral, cubic, triangular, and square circle packings studied in \cite{WhiteheadEtAlReciprocity}, the leading constant in the quadratic obstructions is only supported at primes where there are congruence obstructions. See Remark~\ref{rem:mod6andstrongapprox} for more.

\subsection{Computations}
An accompanying PARI/GP \cite{PARI} package to explore Eisenstein circle packings and the Eisenstein Schmidt arrangement is found in \cite{GHEisenstein}. This includes a comprehensive suite of methods to create circle packings, find all of the circles or curvatures up to a fixed bound, make LaTeX images of the packings, and much more. Consult the user's guide in the project for a full listing of the capabilities.

Using this package, we can further investigate the local-global conjecture. First, we estimate the growth rate.

\begin{conjecture}\label{conj:growthratehdim}
    Let $E$ be a moiety of a bounded primitive Eisenstein circle packing. Then, the number of circles of curvature at most $N$ is asymptotic to $c_E N^{\delta}$, where $c_E$ is a constant depending on $E$, and $\delta\approx 1.4124$ is the Hausdorff dimension of the limit set of the packing.
\end{conjecture}

This predicted growth rate is analogous to Apollonian circle packings (where $\delta \approx 1.3056$), as well as more general thin groups. To estimate $\delta$, we took the packing with root quadruple $(-3, 5, 14, 10)$, computed all circles of curvature up to $10^7$, and separated this data into $10^4$ bins of length $10^3$. There were $661,725,121$ curvatures processed, and running an ordinary least squares regression on the data, the estimate $\delta\approx 1.4124$ was found, with an $R^2$ coefficient of $0.999999998$.

We take into account the quadratic families found in Theorem~\ref{thm:reciprocity} to make an updated local-global conjecture.

\begin{definition}
    Let $E$ be a moiety of a primitive Eisenstein circle packing. Call a positive integer \emph{sporadic} if it is admissible in $E$, does not appear in $E$, and does not lie in one of the quadratic families described by Theorem~\ref{thm:reciprocity}. Take the set of all sporadic curvatures to be $\spor_E$, and the set of sporadic curvatures up to $N$ is $\spor_E(N)$.
\end{definition}

We took all 48 primitive Eisenstein circle packings with root quadruple having all entries bounded by $60$, and computed $\spor_E(N)$ for a varying bound $N$. In each case, we chose an $N$ large enough so that the final sporadic curvature found was below $N/10$. An excerpt of this data is displayed in Table~\ref{table:conjecturedata}. In particular, this suggests that the sporadic sets are finite.

\begin{conjecture}[Modified local-global conjecture]\label{conj:modifiedlocalglobal}
    The set $\spor_E$ is finite.
\end{conjecture}

\subsection{Map of the paper}
Sections~\ref{sec:eisensteincirclepackings}-\ref{sec:quadraticforms} describe the basics of Eisenstein circle packings, going through their geometry, algebra, and number theory, and finally connecting them to quadratic forms. In Section~\ref{sec:schmidt} we begin an interlude where we study the full Eisenstein Schmidt arrangement, from which we extract the smaller Eisenpint Schmidt arrangement. Section~\ref{sec:schmidtsymmetries} connects these two parts, by demonstrating that the Eisenpint tangency packings are in fact primitive Eisenstein circle packings, i.e. Theorem~\ref{thm:Qrt3packingsareEisenstein}. Section~\ref{sec:strongapproximation} settles strong approximation and density one local-global, and the reciprocity obstructions are explained in Section~\ref{sec:reciprocity}. Section~\ref{sec:computation} concludes the paper with explicit computational evidence towards Conjecture~\ref{conj:modifiedlocalglobal}.

\subsection*{Acknowledgements}  Thank you to Iv\'an Rasskin for helpful discussions.  Rickards was supported by NSERC Discovery Grant RGPIN-2025-04068. Katherine E. Stange was supported by NSF DMS-2401580 and a Joan and Joseph Birman Fellowship from the American Mathematical Society.

\section{Eisenstein circle packings: Geometry}\label{sec:eisensteincirclepackings}

We begin by recalling a treatment of inversive geometry through linear algebra, which was pioneered by Pedoe \cite{PedoeCircles57, PedoeGeometryCourse}.

\subsection{Inversive geometry and oriented circles}\label{sec:pedoedefinition}

To any oriented circle $\cir$, we can associate $u, v \in \RR$ and $w \in \CC$ so that:
\[\cir = \{X/Y\in\widehat{\CC}:uX\overline{X}-wY\overline{X}-\overline{w}X\overline{Y}+vY\overline{Y}=0\},\]
where
\begin{itemize}
    \item $u\in\RR$ is the curvature of $\cir$;
    \item $v\in\RR$ is the co-curvature of $\cir$;
    \item $w\in\CC$ is the curvature-centre of $\cir$.
\end{itemize}
Observe that swapping the orientation multiplies $u,v,w$ all by $-1$ simultaneously.

Furthermore, having chosen such $u,v,w$, 
\begin{itemize}
    \item they satisfy the relation $uv=w\overline{w}-1$;
    \item $u\neq 0$ if and only if $\cir$ is a circle, which has radius $\frac{1}{|u|}$;
    \item $u<0$ if and only if $\infty$ is in the interior of $\cir$;
    \item $u>0$ if and only if $\infty$ is in the exterior of $\cir$;
    \item if $u\neq 0$, then $w$ is the product of $u$ and the centre of $\cir$;
    \item if $u=0$, then $w$ is a unit vector pointing from the exterior to the interior, orthogonal to $\cir$;
    \item $v$ is the curvature of $\cir$ after the M\"obius inversion $z\rightarrow\frac{1}{z}$.
\end{itemize}

\begin{definition}
    The \emph{quadruple associated to an oriented circle} is the column vector $(u, v, p, q)\in\RR^4$, where $u$ is the curvature, $v$ is the co-curvature, and $w=p+qi$ is the curvature-centre.
\end{definition}

An oriented circle necessarily satisfies the quadratic form $-uv+p^2+q^2=1$. Conversely, any $(u, v, p, q)\in\RR^4$ satisfying $-uv+p^2+q^2=1$ gives an oriented circle.

When considering collections of oriented circles, we typically do not want to consider a circle and its opposite orientation together. With this in mind, we make the following definition.

\begin{definition}
    Consider two oriented circles as \emph{distinguishable} if the corresponding unoriented circles are distinct.
\end{definition}

We recall the inversive distance between two oriented circles (e.g. from \cite[Section 40.2]{PedoeGeometryCourse}), which is M\"obius invariant. We use a slightly different convention than the classical inversive distance, which is often defined as the negative or cosine of our choice (for example, see \cite{CoxeterInversiveDistance}).

\begin{definition}\label{def:pedoedistance}
    Consider two oriented circles $\cir_1,\cir_2$, with associated quadruples $(u_i, v_i, p_i,q_i)$. The inversive distance between them is given by the symmetric bilinear form
    \[\langle \cir_1, \cir_2\rangle := -\frac{1}{2}u_1v_2-\frac{1}{2}u_2v_1+p_1p_2+q_1q_2=(u_1, v_1, p_1, q_1)\left(\begin{matrix}0 & -1/2 & 0 & 0\\-1/2 & 0 & 0 & 0\\0 & 0 & 1 & 0\\0 & 0 & 0 & 1\end{matrix}\right)\left(\begin{matrix}u_2\\v_2\\p_2\\q_2\end{matrix}\right).\]
    Call this matrix $M_{\Inv}$.
\end{definition}

Note that $\langle v, v\rangle = 1$ if and only if $v\in\RR^4$ corresponds to an oriented circle. Furthermore, if $\cir_1$ and $\cir_2$ are distinguishable circles, then
\begin{itemize}
    \item $\langle \cir_1, \cir_2\rangle = -1$ if and only if the circles are tangent, with the orientations meeting in opposite directions at the tangency point;
    \item $\langle \cir_1, \cir_2\rangle = 0$ if and only if the circles are mutually orthogonal;
    \item $\langle \cir_1, \cir_2\rangle = 1$ if and only if the circles are tangent, with the orientations meeting in the same direction at the tangency point.
\end{itemize}
If $\cir_1,\cir_2$ are tangent Euclidean circles with exteriors containing $\infty$, then $\langle \cir_1, \cir_2\rangle = -1$ if and only if they are externally tangent, i.e. have disjoint interiors.

If the curvatures are non-zero, we also have the formula
\[\langle \cir_1, \cir_2\rangle = \frac{1}{2}(u_1/u_2+u_2/u_1-d^2u_1u_2),\]
where $d$ is the distance between the centres.

\subsection{Tangency points of circles}

The inversive distance is an inner product on $\RR^4$, where oriented circles correspond to vectors $\cir$ with $\langle \cir, \cir\rangle =1$. We can express points in $\widehat{\CC}$ in this framework as well.

\begin{definition}
    Let $P\in\widehat{\CC}$ be a point. The quadruple associated to $P$ is
    \[\begin{cases}(0, 1, 0, 0) & \text{if $P=\infty$}\\(1, p^2+q^2, p, q) & \text{if $P=p+qi$}\end{cases}.\]
\end{definition}

If we also label this quadruple $P$, we note that $P\neq 0$ and $\langle P, P\rangle = 0$. In fact, every non-zero vector $v\in\RR^4$ for which $\langle v, v\rangle = 0$ can be rescaled by $\RR^{\times}$ in a unique way to produce a point.

With this parametrization, a simple computation demonstrates the following.

\begin{proposition}
    A point $P$ is on a circle $\cir$ if and only if $\langle \cir, P\rangle = 0$.
\end{proposition}

From this, we can describe the tangency point of two circles.

\begin{proposition}\label{prop:tangencypoint}
    Let $\cir_1, \cir_2$ be distinguishable circles with $\langle \cir_1, \cir_2\rangle = -1$. Then, the tangency point is represented by $\cir_1+\cir_2$ (that is, the tangency point is represented by the sum of the vectors in $\RR^4$ representing the circles).
\end{proposition}
\begin{proof}
    Note that
    \[\langle \cir_1+\cir_2, \cir_1+\cir_2\rangle = \langle \cir_1, \cir_1\rangle + 2\langle \cir_1, \cir_2\rangle + \langle \cir_2, \cir_2\rangle = 1 - 2 + 1 = 0,\]
    so as $\cir_1+\cir_2\neq 0$, this represents a point. Finally, we have
    \[\langle \cir_1+\cir_2, \cir_1\rangle = \langle \cir_1, \cir_1\rangle + \langle \cir_1, \cir_2\rangle = 1 - 1 = 0,\]
    so $\cir_1+\cir_2$ is on $\cir_1$. Similarly, it is on $\cir_2$, hence it is the unique intersection point.
\end{proof}

\subsection{Proof of Proposition~\ref{prop:4wheelmobius}}
Let $W=(\cir_1, \cir_2, \cir_3, \cir_4)$ and $W'=(\cir_1', \cir_2', \cir_3', \cir_4')$ be two 4-wheels. It suffices to show that one can map $W$ to $W'$ via a sequence of translations, dilations, reflections, and circle inversions.

To begin, let $P_i$ denote the tangency point of $\cir_i$ and $\cir_{i+1}$, and define $P_i'$ analogously. There exists a M\"obius map which takes $(\cir_1, \cir_2)$ to $(\cir_1', \cir_2')$, whence $P_1$ to $P_1'$ (to see this, one can map any pair of externally tangent circles to the real axis pointed left and a unit circle centred at $i$ oriented counterclockwise). In fact, we can also assume that $P_2$ is sent to $P_2'$. If it were not, say it is sent to $K$. Consider a circle $\cir$ through $P_1'$ that is orthogonal to $\cir_1'$, and thus is orthogonal to $\cir_2'$ as well. An inversion through $\cir$ will fix both $\cir_1'$ and $\cir_2'$. We can choose $\cir$ to map $K$ to $P_2'$, which preserves the other assumptions made. See Figure~\ref{fig:4wheelmobius} for a depiction of this setup.

\begin{figure}[htb]
    \includegraphics[scale=0.7]{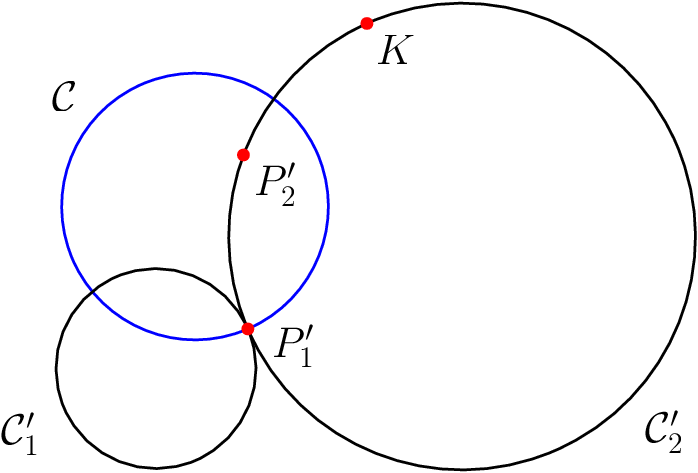}
    \caption{A sample configuration where circle $\cir$ is orthogonal to $\cir_1'$ and $\cir_2'$ at $P_1'$, and swaps $P_2'$ and $K$.}\label{fig:4wheelmobius}
\end{figure}

Taking stock of the situation, we have mapped the 4-wheel $W$ to the 4-wheel $(\cir_1', \cir_2', \Dcir, \mathcal{F})$, where $\cir_2'$ and $\Dcir$ are tangent at $P_2'$. Circle $\Dcir$ must satisfy
\[\langle \cir_1', \Dcir\rangle = -2, \quad \langle \cir_2', \Dcir\rangle = -1,\quad \langle P_2', \Dcir\rangle = 0.\]
As $(\cir_1', \cir_2', \cir_3', \cir_4')$ are linearly independent (follows from Proposition~\ref{prop:4wheelinvertible}), write $\Dcir=a\cir_1'+b\cir_2'+c\cir_3'+d\cir_4'$. The above equations give the linear system
\begin{align*}
    -2 & = a - b -2c - d\\
    -1 & = -a + b - c - 2d\\
    0  & = -3a - 3d,
\end{align*}
which has general solution $(a, b, c, d) = (a, 0, a+1, -a)$. Finally, using $\langle \Dcir, \Dcir\rangle = 1$, this reduces to $3a^2+1 = 1$, hence $a=0$ is the only solution. This implies that $(a, b, c, d) = (0, 0, 1, 0)$, i.e. $D=\cir_3'$.

In a similar fashion, circle $\mathcal{F}$ is almost determined by the inversive distances with $\cir_1',\cir_2',\cir_3'$; this again gives a linear set of solutions (in one variable). The quadratic condition $\langle \mathcal{F}, \mathcal{F}\rangle = 1$ gives at most two possibilities for $\mathcal{F}$: necessarily $\cir_4'$, and the circle swap of $\cir_4'$. Since these 4-wheels are related by an inversion (Definition~\ref{def:circleswap}), the proposition follows.

\section{Eisenstein circle packings: Algebra}\label{sec:eispackingalgebra}

We defined a 4-wheel $W=(\cir_1, \cir_2, \cir_3, \cir_4)$ to be a tuple that satisfies $\Inv(\cir_1, \cir_2, \cir_3, \cir_4) = R_{\Eis}$, which also has the sum of curvatures being non-negative.  Recalling that we can write $W$ as a $4\times 4$ matrix whose columns are the vectors for the $\cir_i$, the equation is equivalent to

\begin{equation}\label{eqn:4wheel}
    W^TM_{\Inv}W = W^T \left(\begin{matrix}0 & -1/2 & 0 & 0\\-1/2 & 0 & 0 & 0\\0 & 0 & 1 & 0\\0 & 0 & 0 & 1\end{matrix}\right) W = \left(\begin{matrix}1 & -1 & -2 & -1\\-1 & 1 & -1 & -2\\-2 & -1 & 1 & -1 \\ -1 & -2 & -1 & 1\end{matrix}\right) = R_{\Eis}.
\end{equation}

\subsection{Circle swaps}

Since the matrices $M_{\Inv}$ and $R_{\Eis}$ are invertible, Equation~\eqref{eqn:4wheel} implies the following result.

\begin{proposition}\label{prop:4wheelinvertible}
    As a matrix, every 4-wheel is invertible. In particular, the four circles in a 4-wheel are linearly independent as vectors in $\RR^4$.
\end{proposition}

Our first order of business is to settle the algebra of circle swaps, i.e. Proposition~\ref{prop:circleswapmatrices}. 

\begin{lemma}\label{lem:C4swap}
    Swapping circle $\cir_4$ is equivalent to replacing $W$ with $WS_4$, where
    \[S_4 = \left(\begin{matrix} 1 & 0 & 0 & 2\\0 & 1 & 0 & 0\\0 & 0 & 1 & 2\\0 & 0 & 0 & -1\end{matrix}\right).\]
\end{lemma}
\begin{proof}
    By Proposition~\ref{prop:4wheelinvertible}, we can write $\cir_4'=a\cir_1+b\cir_2+c\cir_3+d\cir_4$ for some $a,b,c,d\in\RR$. Since $(\cir_1, \cir_2, \cir_3, \cir_4')$ is a 4-wheel, it follows that
    \begin{align*}
        -1 = \langle \cir_1, \cir_4'\rangle = & \quad a - b - 2c - d\\
        -2 = \langle \cir_2, \cir_4'\rangle = & -a + b -c - 2d\\
        -1 = \langle \cir_3, \cir_4'\rangle = & -2a - b + c - d
    \end{align*}
    This system of equations has general solution $(a, b, c, d) = (a, 0, a, 1-a)$. Since $\cir_4'$ is an oriented circle, we have the final equation of $\langle \cir_4', \cir_4'\rangle = 1$, which gives
    \[1 = a^2 + a^2 + (1-a)^2 - 4a^2 - 2a(1-a) - 2a(1-a) = 3a^2-6a+1,\]
    hence $a=0, 2$. If $a=0$ we recover $\cir_4$, so $\cir_4'$ is given by $2\cir_1+2\cir_3-\cir_4$, giving the lemma.
\end{proof}

By symmetry, Proposition~\ref{prop:circleswapmatrices} follows from Lemma~\ref{lem:C4swap}.

\subsection{Eisenstein equation}\label{sec:eiseqn}

Inverting and rearranging $W^TM_{\Inv}W = R_{\Eis}$ gives \[WR_{\Eis}^{-1}W^T = W\left( \frac{1}{3}\left(\begin{matrix}1 & -1 & 0 & -1\\-1 & 1 & -1 & 0\\0 & -1 & 1 & -1\\-1 & 0 & -1 & 1\end{matrix}\right)\right) W^T = \left(\begin{matrix}0 & -2 & 0 & 0\\-2 & 0 & 0 & 0\\0 & 0 & 1 & 0\\0 & 0 & 0 & 1\end{matrix}\right)= M_{\Inv}^{-1}.\]

Recall that,
\[W^T = \left(\begin{matrix}u_1 & v_1 & p_1 & q_1\\u_2 & v_2 & p_2 & q_2\\u_3 & v_3 & p_3 & q_3\\u_4 & v_4 & p_4 & q_4\end{matrix}\right),\]
hence this gives quadratic relations in the tuples of curvatures, co-curvatures, curvature-centre real, and curvature-centre imaginary values. The top left corner equation is equivalent to
\[u_1^2+u_2^2+u_3^2+u_4^2=2(u_1+u_3)(u_2+u_4),\]
the Eisenstein equation for the curvatures. Combining this with the additional condition of $u_1+u_2+u_3+u_4>0$, it follows that $(u_1, u_2, u_3, u_4)$ is an Eisenstein quadruple, proving the first claim of Proposition~\ref{prop:eisquadruplefrom4wheel}.

\begin{remark}\label{rem:quadruplesymmetries}
    The Eisenstein equation has a $D_8-$symmetry: given a solution $(a, b, c, d)$, we can swap $a$ with $c$, swap $b$ with $d$, or swap $(a, c)$ and $(b, d)$. This is equivalent to the action of $D_8$, and gives $8$ quadruples. Similarly, if we only allow swapping $a$ and $c$ or $b$ and $d$, we generate the $V_4-$orbit of the quadruple.
\end{remark}

\begin{proposition}\label{prop:4wheelgivencurvatures}
    Given any Eisenstein quadruple, there exists a 4-wheel with those respective curvatures.
\end{proposition}
\begin{proof}
    Let $\langle,\rangle^{R_{\Eis}^{-1}}$ denote the bilinear form corresponding to $R_{\Eis}^{-1}$, and let $c_1\in\RR^4$ denote a non-zero column vector satisfying $\langle c_1, c_1\rangle^{R_{\Eis}^{-1}} = 0$. We claim that there exists $c_2,c_3,c_4\in\RR^4$ so that the matrix $C=\left(\begin{matrix}c_1 & c_2 & c_3 & c_4\end{matrix}\right)$ satisfies $C^TR_{\Eis}^{-1}C=M_{\Inv}^{-1}$. Indeed, if this were true, then take $c_1$ to be the Eisenstein quadruple. The assumption implies that it is non-zero and satisfies the bilinear product being $0$, and let $W=C^T$. Reversing the above computations shows that $W^TM_{\Inv}W = R_{\Eis}$, i.e. $W$ is a 4-wheel with the prescribed curvatures, as desired.

    For the claim, it is equivalent to construct $c_2, c_3, c_4$ such that $\langle c_i, c_j\rangle^{R_{\Eis}^{-1}} = 0$ for all $1\leq i, j\leq 4$, with the exceptions of $\langle c_1, c_2\rangle^{R_{\Eis}^{-1}} = -2$ and $\langle c_3, c_3\rangle^{R_{\Eis}^{-1}} = \langle c_4, c_4\rangle^{R_{\Eis}^{-1}} = 1$. This is accomplished in a similar way to the Gram–Schmidt process.

    First, since $\langle,\rangle^{R_{\Eis}^{-1}}$ is non-degenerate and $c_1\neq 0$, there exists $d\in\RR^4$ for which $\langle c_1,d\rangle^{R_{\Eis}^{-1}}\neq 0$. Take
    \[\lambda = -\frac{\langle d,d\rangle^{R_{\Eis}^{-1}}}{2\langle c_1,d\rangle^{R_{\Eis}^{-1}}},\]
    and let $c_2'=d+\lambda c_1$, so that $\langle c_2', c_2'\rangle^{R_{\Eis}^{-1}}=0$ and $\langle c_1,c_2'\rangle^{R_{\Eis}^{-1}}=\langle c_1,d\rangle^{R_{\Eis}^{-1}}\neq 0$. By appropriately scaling $c_2'$, we produce $c_2$ with $\langle c_1, c_2\rangle^{R_{\Eis}^{-1}}=-2$.

    Next, we can pick non-zero $c_3'$ orthogonal to $c_1, c_2$ (with respect to $\langle,\rangle^{R_{\Eis}^{-1}}$), and finally pick non-zero $c_4'$ orthogonal to $c_1, c_2, c_3'$. By scaling $c_3',c_4'$ appropriately to $c_3,c_4$, we can ensure that $\langle c_i,c_i\rangle^{R_{\Eis}^{-1}}=-1, 0,1$ for $i=3,4$. In particular, we have
    \[C^TR_{\Eis}^{-1}C=\left(\begin{matrix}0 & -2 & 0 & 0\\-2 & 0 & 0 & 0\\0 & 0 & \delta_1 & 0\\0 & 0 & 0 & \delta_2\end{matrix}\right),\]
    with $\delta_i=-1, 0, 1$. Recall that the signature of $R_{\Eis}^{-1}$ is $(3, 1)$, and this is preserved under similarity of quadratic forms. The eigenvalues of the above matrix are $-2, 2, \delta_1, \delta_2$, proving that $\delta_1=\delta_2=1$, as desired.  
\end{proof}

We can now demonstrate that up to symmetry, 4-wheels are completely described by their corresponding Eisenstein quadruples.

\begin{proof}[Proof of Proposition~\ref{prop:eisquadruplefrom4wheel}]
Given a 4-wheel, the work before Remark~\ref{rem:quadruplesymmetries} proves that it produces an Eisenstein quadruple. Proposition~\ref{prop:4wheelgivencurvatures} implies that Eisenstein quadruples all give rise to 4-wheels.

To finish, consider 4-wheels $W=(\cir_1, \cir_2, \cir_3, \cir_4)$ and $W'=(\cir_1', \cir_2', \cir_3', \cir_4')$ which correspond to the same Eisenstein quadruple, hence the corresponding curvatures are equal. As the inversive distance only depends on the relative positions and curvatures of two oriented circles, it follows that there exists a translation and rotation of $W$ that takes $\cir_1$ to $\cir_1'$ and $\cir_3$ to $\cir_3'$. Assume that $\cir_2$ is sent to $\Dcir$ and $\cir_4$ to $\mathcal{F}$. The curvature of $\Dcir$ is fixed, and $\Dcir$ is tangent to disjoint circles $\cir_1'$ and $\cir_3'$. In particular, there are at most two possible locations for $\Dcir$. These locations are related by a reflection across the unique Euclidean line orthogonal to $\cir_1'$ and $\cir_3'$.

Since $\cir_2'$ satisfies the above conditions, after possibly a reflection, we can assume that $\Dcir=\cir_2'$. Finally, we know that $\mathcal{F}$ is tangent to $\cir_1'$ and $\cir_3'$, and has $\langle \cir_2', \mathcal{F}\rangle = -2$. This has at most one solution, and we know that $\cir_4'$ is a solution, whence $\mathcal{F}=\cir_4'$, completing the proof.
\end{proof}

\subsection{Eisenstein swap group}
Recall that the Eisenstein swap group, $\Egp=\langle S_1, S_2, S_3, S_4\rangle$, acts on Eisenstein quadruples $q$ on the right. The equivalence class of $q$ is the orbit $q\Egp$. The four generators of $\Egp$ have order 2, but do not generate a free group, since the pairs $(S_1, S_3)$ commute, as well as $(S_2, S_4)$.

For the moment, take
\[\Epres=\langle S_1, S_2, S_3, S_4 \; | \; S_1S_3S_1^{-1}S_3^{-1},S_2S_4S_2^{-1}S_4^{-1},S_1^2,S_2^2,S_3^2,S_4^2\rangle\]
as an abstract group. There is an obvious homomorphism $\Epres\rightarrow\Egp$ taking $S_i$ to $S_i$, and the claim in Proposition~\ref{prop:eisensteingrouppresentation} is that the kernel is trivial, i.e. $\Egp$ satisfies no more relations. This will be proven at the end of Section~\ref{sec:stationary}, by developing a theory of reduced words for both $\Epres$ and its action on Eisenstein quadruples.

We visualize the Cayley graph of $\Epres$ in Figure~\ref{fig:cayley}. Edges are unoriented as our four generators all have order 2. From this, it is clear how to define a reduced word that links any two elements of $\Epres$.

\begin{figure}[htb]
    \includegraphics[scale=0.65]{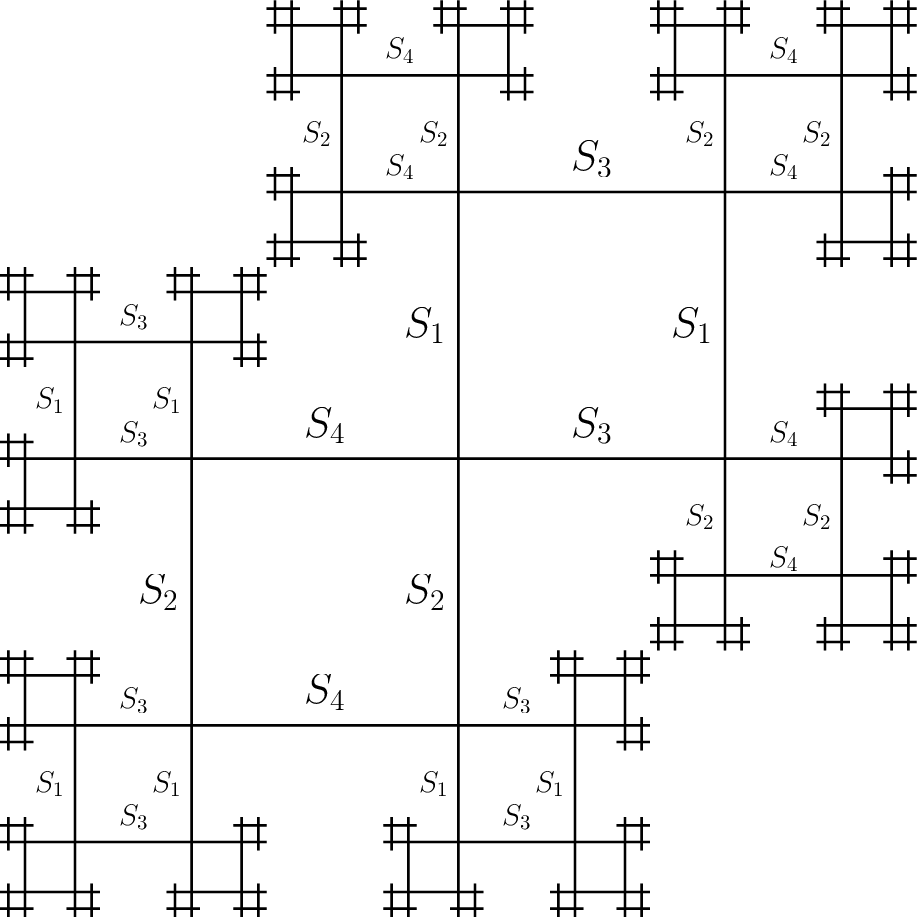}
    \caption{The start of the Cayley graph of $\Epres$.}\label{fig:cayley}
\end{figure}

\begin{definition}
    A \emph{reduced word} for $\Epres$ is a word $W=S_{i_1}S_{i_2}\cdots S_{i_k}$ such that for all $1\leq j\leq k-1$,
    \begin{enumerate}[label={\upshape(\roman*)}]
        \item $i_j\neq i_{j+1}$;
        \item if $i_j=3$, then $i_{j+1}\neq 1$;
        \item if $i_j=4$, then $i_{j+1}\neq 2$.
    \end{enumerate}
\end{definition}

The Cayley graph is comprised of squares, attached at their corners. As we travel along it, if we pass between adjacent corners of a square, there is a unique shortest way to do so. If we travel to an opposite corner, there are two possible choices, but items (ii) and (iii) indicate that we always follow $S_1$ then $S_3$, or $S_2$ then $S_4$. In particular, given two elements $M, N\in\Epres$, there exists a unique reduced word $W$ for which $MW=N$.

\subsection{Reduction theory}\label{sec:reductiontheory}
The action of $\Epres$ on Eisenstein quadruples is not always free. This corresponds to symmetries in the corresponding Eisenstein packing, and must be accounted for.

\begin{definition}
    Let $q=(a, b, c, d)$ be an Eisenstein quadruple. Call the action of a swap on $q$
    \begin{itemize}
        \item \emph{decreasing} if $a+b+c+d$ decreases;
        \item \emph{stationary} if $a+b+c+d$ is unchanged;
        \item \emph{increasing} if $a+b+c+d$ increases.
    \end{itemize}
\end{definition}

In particular, a reduced quadruple is equivalent to none of the four swaps being decreasing.

The existence of a stationary swap indicates a symmetry in the packing: the tuple of curvatures is unchanged, but one curvature-centre moves.

Consider a word $W=S_{i_1}S_{i_2}\cdots S_{i_k}\in\Epres$ acting on $q$. Denote $qS_{i_1}S_{i_2}\cdots S_{i_j}$ by $q_j$.

\begin{definition}\label{def:reducedword}
    The word $W$ is said to be a \emph{reduced word for $q$} if $W$ is reduced in $\Epres$, and for all $1\leq j\leq k-1$, $S_{i_{j+1}}$ is not stationary on $q_j$.
\end{definition}

Given any $q$ and equivalent $q'$ (Definition~\ref{def:equivalentquadruple}), there exists a reduced word $W$ for $q$ for which $qW=q'$. Indeed, start with a word $W'\in\Epres$ with $qW'=q'$, reduce $W'$ in $\Epres$, delete any of the stationary swaps, and repeat. This will eventually terminate, producing the desired reduced word $W$ for $q$.

The following proposition generalizes this observation, and sets out the rules and structure of quadruple reduction. Part (iii) is the statement of Proposition~\ref{prop:atmostonereducedintegralone}.

\begin{proposition}\label{prop:uniquereduction}
    Let $q$ be an Eisenstein quadruple. Then,
    \begin{enumerate}[label={\upshape(\roman*)}]
        \item If $q'$ is an Eisenstein quadruple equivalent to $q$, there is a unique reduced word $W$ for $q$ for which $qW=q'$;
        \item Furthermore, if $q$ is reduced, then $W$ contains only increasing swaps;
        \item There is at most one reduced quadruple equivalent to $q$. If $q$ is integral, there is exactly one.
        \item If $q$ is reduced, then at least one entry of $q$ is non-positive.
    \end{enumerate}
\end{proposition}

In particular, this demonstrates that the stationary swaps account for the only repetition of Eisenstein quadruples in $q\Epres$. By finding an equivalence class with no stationary swaps (which we will do in the next section), Proposition~\ref{prop:eisensteingrouppresentation} (that $\Egp = \Epres$) will follow.

Key to Proposition~\ref{prop:uniquereduction} is the good behaviour of decreasing/increasing swaps on \emph{Eisenstein quadruples} with relation to each other. For example, consider $q=(1, 1, -1, -1)$, which is \emph{not} an Eisenstein quadruple. We find that $q(S_1S_3S_2S_4)^2=q$, despite $(S_1S_3S_2S_4)^2$ being a non-trivial reduced word for $q$, contradicting part (i).

We start the proof with a series of lemmas.

\begin{lemma}\label{lem:pairsnonneg}
    Assume $(a, b, c, d)$ is an Eisenstein quadruple. Then
    \[a+c,b+d>0,\qquad a+b,a+d,b+c,c+d\geq 0.\]
\end{lemma}
\begin{proof}
    Recall the Eisenstein equation,
    \[a^2+b^2+c^2+d^2=2(a+c)(b+d).\]
    It follows that $(a+c)(b+d)\geq 0$. If equality holds, then $a^2+b^2+c^2+d^2=0$, whence $a=b=c=d=0$, contradiction. Therefore $(a+c)(b+d)>0$, so the two factors have the same sign. As $a+b+c+d>0$, this sign is positive.

    To finish the claim, by symmetry, it suffices to show that $a+b<0$ is impossible. If both $a< 0$ and $b< 0$, then $0<a+c<c$ and $0<b+d< d$. Thus $a^2+b^2+c^2+d^2<2cd$, which rearranges to
    \[a^2+b^2+(c-d)^2< 0,\]
    contradiction. Otherwise, we may assume that $a<0$ and $b\geq 0$, and the Eisenstein equation rearranges to
    \[(a+c-d)^2+b^2=2ab+2ac+2bc=2c(a+b)+2ab\leq 0.\]
    This implies equality with zero, whence $b=c=0$ and $a=d$, so $a+b+c+d=2a>0$, contradiction.
\end{proof}

Lemma~\ref{lem:pairsnonneg} implies that there do not exist multiple negative curvatures in a packing, consistent with there being at most one ``outer circle.''

\begin{lemma}\label{lem:tworeductive}
    If there exists two swaps that are not increasing, they must be either $S_1$ and $S_3$, or $S_2$ and $S_4$, with one exception: the $V_4-$orbit of $(a, a, 0, 0)$ for $a> 0$. In this case, both are stationary.  
\end{lemma}

Note that at least two swaps are increasing from any quadruple.

\begin{proof}
    Assume otherwise. By applying a $V_4-$permutation, we can assume that $S_1$ and $S_2$ are both decreasing or stationary. This is equivalent to $2b+2d-a\leq a$ and $2a+2c-b\leq b$, hence $a\geq b+d$ and $b\geq a+c$. By the Eisenstein equation,
    \[a^2+b^2+c^2+d^2=2(a+c)(b+d)\leq 2ab,\]
    so $(a-b)^2+c^2+d^2\leq 0$. This is equivalent to $a=b$ and $c=d=0$, the sole exception.
\end{proof}

This lemma is key to the reduction process, as it essentially demonstrates that there is a unique way to reduce (save for the commuting swaps). In particular, if $S=S_1$ or $S_3$ is increasing for $q$ (hence $S$ is decreasing for $qS$), then $S_2$ and $S_4$ are increasing for $qS$. Similarly, if $S=S_2$ or $S_4$ is increasing for $q$, then $S_1$ and $S_3$ are increasing for $qS$.

\begin{lemma}\label{lem:reducedwordnontrivial}
    There does not exist a non-trivial reduced word $W$ for $q$ for which $qW=q$.
\end{lemma}

\begin{proof}
    Assume otherwise, and across all possible counterexamples $(q, W)$, pick one that minimizes $a+b+c+d$ (where $q=(a, b, c, d)$).  (Recall that all Eisenstein quadruples have $a+b+c+d > 0$, so a minimal counterexample exists.) Note that every quadruple visited in the sequence from $q$ to $qW$ will be a counterexample (as this is a closed path), so the quadruple sum cannot be smaller than that of $q$.

    Call a $S_1$ or $S_3$ swap a \emph{blue swap}, and a $S_2$ or $S_4$ swap a \emph{red swap}. Consider $W$ as broken up into blocks of blue and red swaps. Each block will have 1 or 2 swaps. The first swap $S$ must be increasing, as stationary swaps do not occur in reduced words, and a decreasing swap would produce a smaller quadruple sum counterexample. If the first block has length 2, say $S_iS_j$, then $S_j$ must also be increasing. Indeed, since $S_iS_j=S_jS_i$, the form $qS_j$ is also a counterexample to the lemma.

    Thus, the first block is only increasing swaps. By Lemma~\ref{lem:tworeductive} (and the subsequent comment), all swaps in the next block are increasing. This continues inductively, and we see that every swap is increasing, so it was impossible to end back at $q=qW$, contradiction.
\end{proof}

We can now proceed to the proof of Proposition~\ref{prop:uniquereduction}.

\begin{proof}[Proof of Proposition~\ref{prop:uniquereduction}]
    Let $q'$ be equivalent to $q$. We have already noted that there exists a reduced word $W$ for $q$ for which $qW=q'$, so assume there exists a second distinct one, $W'$. By placing $W$ and $W'$ in the Cayley graph, we see that there exists a non-trivial reduced word $W''$ in $\Epres$ for which $q'W''=q'$. As this was built from $W$ and $W'$, none of the swaps are stationary, so it is still reduced for $q'$, contradicting Lemma~\ref{lem:reducedwordnontrivial}. This completes (i).

    Part (ii) follows from the proof of Lemma~\ref{lem:reducedwordnontrivial}. The first swap(s) must be increasing (as $q$ is reduced), and all subsequent swaps are therefore forced to be increasing as well.

    For (iii), the existence of at most one reduced quadruple follows from (i) and (ii). If $q$ is integral, applying a reductive swap decreases $a+b+c+d$. Repeat the process; since this term is integral and non-negative, we eventually terminate at a reduced quadruple.

    For (iv), assume that $(a, b, c, d)$ is reduced with $a,b,c,d>0$. Without loss of generality, it can be assumed that $a\geq c$ and $b\geq d$. Since $S_1$ and $S_2$ are not reductive, we obtain $b+d\geq a$ and $a+c\geq b$. The Eisenstein equation rearranges as
    \[c(c-a) + d(d-b) = (a+c-b)(b+d) + (a+c)(b + d - a).\]
    Our assumptions imply the the left hand side is $\leq 0$, and the assumptions plus Lemma~\ref{lem:pairsnonneg} implies the right hand side is $\geq 0$. Equality thus holds, which further gives $a=c$, $b=d$, $a+c=b$, $b+d=a$, hence $a=b=c=d=0$, contradiction. Part (iv) follows.
\end{proof}

We can also settle Proposition~\ref{prop:threesmallestcurvatures}.

\begin{proof}[Proof of Proposition~\ref{prop:threesmallestcurvatures}]
    Assume that $q=(a, b, c, d)$ is reduced with $a=\min\{a, b, c, d\}$. By Proposition~\ref{prop:uniquereduction}(iv), it follows that $a\leq 0$, and from Lemma~\ref{lem:pairsnonneg}, $b,c,d\geq 0$. Since $S_2$ and $S_4$ are not decreasing on $q$, $c\geq a+c\geq b\geq a$, and similarly $c\geq a+c\geq d$, whence $c=\max\{a, b, c, d\}$, so $a, b, d$ are the smallest three curvatures of $q$.  Now, we show that no other curvature in the packing are smaller.
    
    Let $M=\max(b, d)$, and we claim that when any swap $S_i$ is applied to this initial quadruple, the new curvature will be $\geq M$.
    
    For $S_1$, since it is not decreasing, $2b+2d-a\geq a$, which rearranges to
    \[2b+2d-a\geq b+d\geq M.\]
    For $S_2$, we know that $a+c\geq b, d$ (as above), hence $2a+2c-b\geq b$, and (by adding them) $2a+2c-b\geq d$, whence $2a+2c-b\geq M$. For $S_3$ the claim is immediate as $c\geq M$, and $S_4$ follows analogously to $S_2$.

    In general, any equivalent quadruple $q'=(a',b',c',d')$ is obtained from $q$ by a sequence of increasing swaps. By induction, we claim that every swap produces a curvature $\geq M$. The base case was given above. Next, if we are executing $S_i$, and have previously used $S_i$ to get to $q'$, then the new curvature is $\geq M$ by induction and the fact that $S_i$ is increasing. Otherwise, it is our first time swapping $S_i$. If $i=1$, then we have
    \[2(b'+d')-a'=2(b'+d')-a\geq 2(b+d)-a,\]
    the first entry of $q\cdot S_1$, which is $\geq M$ by the base case. The other indices $i$ are analogous.
\end{proof}

\subsection{Stationary swaps}\label{sec:stationary}
Lemma~\ref{lem:tworeductive} severely limits the possible stationary swaps.

\begin{lemma}\label{lem:stationaryreduced}
    Assume that $S_1$ is a stationary swap for $q$. Then either $q$ or $qS_3$ is reduced. Analogous results hold for $S_2, S_3, S_4$ being stationary.
\end{lemma}
\begin{proof}
    By Lemma~\ref{lem:tworeductive}, neither $S_2$ nor $S_4$ can be decreasing. If $S_3$ is not decreasing then $q$ is reduced, and otherwise, $qS_3$ is reduced (since, taking $qS_3$ in place of $q$, we still have that $S_1$ is stationary and neither $S_2$ nor $S_4$ can be decreasing). 
\end{proof}

Thus, it suffices to check the reduced quadruple for stationary swaps.

\begin{corollary}\label{cor:stationary}
    Assume $q=(a, b, c, d)$ is reduced, with $a,c\neq b+d$ and $b,d\neq a+c$. Then no swap applied to any Eisenstein quadruple equivalent to $q$ is stationary.
\end{corollary}
\begin{proof}
Stationary swaps correspond to $a,c=b+d$ or $b,d=a+c$. By Lemma~\ref{lem:stationaryreduced}, any such occurrences must also occur in the reduced quadruple (the unique one guaranteed by Proposition~\ref{prop:uniquereduction}(iii)).
\end{proof}

The quadruple $q=(-3, 5, 14, 10)$ is an Eisenstein quadruple falling under the purview of Corollary~\ref{cor:stationary}. Therefore the reduced words for $q$ coincide with the reduced words of $\Epres$. By Proposition~\ref{prop:uniquereduction}, the orbit of $q$ under reduced words in $\Epres$ produces distinct quadruples, giving a free group action. This action descends to $\Egp$ acting on $q$, whence $\Epres\cong\Egp$. In particular, Proposition~\ref{prop:eisensteingrouppresentation} follows.

\section{Eisenstein circle packings: Number Theory}

While the congruence obstructions for curvatures will be settled later (Section~\ref{sec:strongapproximation}), we will first describe the behaviour of Eisenstein quadruples modulo $3$, $4$, and $6$.

\subsection{Standard position}

\begin{lemma}\label{lem:twoeventwoodd}
    Assume $(a, b, c, d)$ is a primitive Eisenstein quadruple. Then exactly one of $(a, c)$ is even, and exactly one of $(b, d)$ is even.
\end{lemma}
\begin{proof}
From the Eisenstein equation $a^2+b^2+c^2+d^2=2(a+c)(b+d)$, it follows that zero, two or four of $a, b, c, d$ must be even. All four would violate primitivity, and if none were even, then $4\equiv a^2+b^2+c^2+d^2\equiv 2(a+c)(b+d)\equiv 0\pmod{8}$, contradiction. Therefore two are even and two are odd. If $a\equiv c\pmod{2}$, then $2\equiv a^2+b^2+c^2+d^2\equiv 2(a+c)(b+d)\equiv 0\pmod{4}$, contradiction. Thus, it must be that exactly one of $(a, c)$ is even, and exactly one of $(b, d)$ is even.
\end{proof}

A consequence of Lemma~\ref{lem:twoeventwoodd} is that every primitive Eisenstein quadruple is either in standard position, or can be put there by swapping $(b, d)$.

\begin{remark}
    Standard position is special to primitive Eisenstein packings, and should not be seen as ``canonical''. It relies on arithmetic data of our curvatures, and since all 4-wheels are M\"obius (and conjugation) equivalent (Proposition~\ref{prop:4wheelmobius}), it is not preserved.
    
    For non-integral packings, if one wanted to choose between $(a, b, c, d)$ and $(a, d, c, b)$, you can make the choice to take $b\leq d$. A disadvantage of this with primitive packings will be seen next section, where to ensure the quadratic forms we construct are primitive and integral, we need to know which of our curvatures are even and odd. Hence, we stick to standard position.
\end{remark}

Another easy result to settle now is Proposition~\ref{prop:eisensteinreduced}, which gives a unique representative for each primitive Eisenstein circle packing.

\begin{proof}[Proof of Proposition~\ref{prop:eisensteinreduced}]
    Take any Eisenstein quadruple $q$ corresponding to a 4-wheel in a fixed primitive packing, and Proposition~\ref{prop:uniquereduction}(iii) gives a unique reduced quadruple $(a, b, c, d)$ in the packing. By part (iv), at least one entry must be nonpositive. By (possibly) swapping the pair $(a, c)$ with $(b, d)$, and then (possibly) swapping $(a, c)$, we can ensure that $a\leq 0$.

    If $a$ is the only nonpositive curvature, the only ambiguity left is in the swapping of $(b, d)$. This is resolved by Lemma~\ref{lem:twoeventwoodd}, i.e. choosing standard position.
    
    Otherwise, if this quadruple has a second nonpositive curvature, Lemma~\ref{lem:pairsnonneg} implies that $a=0$, as well as $b=0$ or $d=0$. Similarly to Lemma~\ref{lem:tworeductive}, the quadruple is either $(0, 0, b, b)$ or $(0, b, b, 0)$, and so there is still no ambiguity other than swapping $(b, d)$, which was already resolved.
\end{proof}

\subsection{Modulo 4}

Proposition~\ref{prop:curvaturesmod4} claims that the odd curvatures in a primitive packing are all equivalent modulo 4. We now settle this claim.

\begin{proof}[Proof of Proposition~\ref{prop:curvaturesmod4}]
    Let $(a, b, c, d)$ be an Eisenstein quadruple. Using Lemma~\ref{lem:twoeventwoodd}, we can put it in standard position, and further assume that $a,b$ are odd and $c,d$ are even. If we swap $c$ with $S_3$, we find that
    \[2b+2d-c\equiv 2(b+d)-c\equiv 2-c\equiv 2+c\pmod{4},\]
    since $b+d$ is odd and $c$ is even. The analogous result holds for $d$, hence $c,d$ both represent the two classes $0,2\pmod{4}$. On the other hand, we swap $a$ with $S_1$ to get
    \[2b+2d-a\equiv 2(b+d)-a\equiv 2-a\equiv a\pmod{4},\]
    since $b+d$ and $a$ are odd. Thus swapping $a$ does not change it modulo 4, and the same holds for $b$.

    It remains to show that $a\equiv b\pmod{4}$, hence the modular behaviour on each moiety is identical. We can rewrite the Eisenstein equation as
    \[(a-b)^2=2ad+2bc-(c-d)^2.\]
    From above, we can swap $c$ and/or $d$ to a position where $c\equiv d\equiv 0\pmod{4}$, whence $2ad+2bc-(c-d)^2\equiv 0\pmod{8}$. Thus $8\mid (a-b)^2$, so $a\equiv b\pmod{4}$.
\end{proof}

\begin{remark}
\label{rmk:SA}
    In Theorem~\ref{thm:SA}, we confirm that the bad modulus of $\Egp''$ (defined before Proposition~\ref{prop:orbitsEintro}) is 16, hence all congruence obstructions in an orbit must occur modulo 16. In the current section we demonstrated the stronger result that they in fact occur modulo 4. Indeed,  strong approximation is a property of the entire group $\Egp''$, as opposed to an orbit. Orbits of \emph{Eisenstein quadruples} have restrictions modulo 2, and these happen to lead to conditions compactly described modulo 4.
    By contrast, consider $(0, 1, 1, 1)\cdot\Egp$, where $(0, 1, 1, 1)$ is \emph{not} an Eisenstein quadruple modulo 16. By computing the whole orbit, we find that the third entry is restricted to being $1,3,9\pmod{16}$, which cannot be described modulo a lower power of $2$.
\end{remark}

\subsection{Modulo 3}
As the bad modulus for strong approximation is 16, we may expect that nothing interesting happens modulo 3. Intriguingly, there are modulo 3 solutions that do not lift to modulo 9, and therefore do not lift to the integers.

\begin{proposition}\label{prop:mod3}
    Let $(a, b, c, d)$ primitively satisfy the Eisenstein equation modulo $3$. Then, up to $D_8 $ permutations, we have
    \[(a, b, c, d)\in\{(0, 0, 1, 1), (0, 0, 2, 2), (0, 1, 1, 2), (0, 1, 2, 2), (1, 2, 1, 2)\}.\]
\end{proposition}
\begin{proof}
    Recall that $x^2\equiv 0, 1\pmod{3}$. From the Eisenstein equation $a^2+b^2+c^2+d^2=2(a+c)(b+d)$, we analyze how many of $a,b,c,d$ are $0\pmod{3}$. It cannot be all four, since that would violate primitivity. It also cannot be three, as the left hand side would be $1\pmod{3}$ and the right hand side $0\pmod{3}$.

    If it is $2$, then we have $(a+c)(b+d)\equiv 1\pmod{3}$. This leads to the $D_8$ symmetries of the first two solutions.

    If it is $1$, then $(a+c)(b+d)\equiv 0\pmod{3}$, leading to the next two solutions.

    Finally, if it is none, then $(a+c)(b+d)\equiv 2\pmod{3}$, leading to the final solution.
\end{proof}

By applying circle swaps, we see that (up to symmetry), the first four quadruples in Proposition~\ref{prop:mod3} all lead to each other. On the other hand, the fifth quadruple is preserved under swaps. Tragically, this quadruple cannot actually lift to the integers, or even modulo $9$.

\begin{proposition}\label{prop:mod9}
    Let $(a, b, c, d)$ primitively satisfy the Eisenstein equation modulo $9$. Then at least one of $a,b,c,d$ is $0\pmod{3}$. In particular, the solution of $(1, 2, 1, 2)\pmod{3}$ does not lift to modulo $9$.
\end{proposition}
\begin{proof}
    If there were a solution modulo $9$ with all curvatures coprime to 3, then from Proposition~\ref{prop:mod3} we could write
    \[a=1+3t, \quad b=2+3x, \quad c=1+3y, \quad d=2+3z.\]
    Plugging this into the Eisenstein equation modulo $9$ yields
    \begin{align*}
        (1+3t)^2 + (2+3x)^2 + (1+3y)^2 + (2+3z)^2 & \equiv 2(1+3t+1+3y)(2+3x+2+3z)\pmod{9}\\
        1+6t+3x+6y+3z & \equiv 7+6t+3x+6y+3z\pmod{9}\\
        0 & \equiv 6\pmod{9},
    \end{align*}
    a contradiction.
\end{proof}

In particular, there are no obstructions to $c(E)\pmod{3}$.

\begin{corollary}
    For all primitive Eisenstein circle packings $E$, we have $c(E)\pmod{3}=\{0, 1, 2\}$.
\end{corollary}
\begin{proof}
    Based on Proposition~\ref{prop:mod9}, the curvatures of an Eisenstein quadruple must lie in one of the first four classes (up to symmetry) in Proposition~\ref{prop:mod3}. After applying circle swaps, these solutions are all in the same graph, and we always see $0,1,2\pmod{3}$.
\end{proof}

\subsection{Modulo 6}
In the Apollonian circle packing setting, the modular behaviour of a quadruple could be studied prime power by prime power, with the results combined via a Chinese remainder theorem as in \cite[Theorem 1.4]{FuchsStrongApproximation}. Intriguingly, this is \emph{not} the case for Eisenstein! In particular, the modulo 2 and 3 reductions of integral solutions to the Eisenstein equation are not independent. The culprit is analogous to \cite[Proposition 3.1]{HKRS23}, which gave restrictions on the sums of tangent Apollonian curvatures modulo 8. As before, these restrictions come from quadratic reciprocity in the integers, and not via a modular argument.

For a prime $p$ and non-zero integer $x$, let $v_p(x)$ denote the exponent of the largest power of $p$ that divides $x$. We recall a standard lemma.

\begin{lemma}\label{lem:xsq3ysqprimedivisors}
    Let $p\equiv 2\pmod{3}$ be prime. Then for any $x, y\in\ZZ$ not both $0$, $v_p(x^2+3y^2)$ is even. Furthermore, if $p$ is odd and $v_p(x^2+3y^2)>0$, then $p\mid x, y$.
\end{lemma}
\begin{proof}
    If $p=2$, write $x=2^ax_o$ and $y=2^by_o$ for $x_o,y_o$ odd. If $a\neq b$, then $v_2(x^2+3y^2)=2\min(a, b)$ is even. If $a=b$, then $x^2+3y^2=2^{2a}(x_0^2+3y_0^2)$, and $x_0^2+3y_0^2\equiv 1+3\equiv 4\pmod{8}$. Thus, $v_2(x^2+3y^2)=2a+2$ is again even.

    Now, if $p$ is odd, it suffices to prove that if $p\mid x^2+3y^2$, then $p\mid y$, and thus $p\mid x$, whence we can divide through by $p^2$ and repeat. Assume otherwise, hence $-3\equiv (x/y)^2\pmod{p}$, so by quadratic reciprocity,
    \[1=\kron{-3}{p}=\kron{p}{3}=-1,\]
    contradiction.
\end{proof}

As in \cite{HKRS23}, we re-write the Eisenstein equation as a divisibility by a sum of tangent curvatures:
\begin{equation}\label{eqn:eisensteinmod6}
    (a+c-d)^2+3a^2=(a+b)(3a-b+2c).
\end{equation}
To proceed, we require a result about the 2-adic behaviour of tangent circles.

\begin{proposition}\label{prop:tangentcurvature2adicvaluation}
    Let $a,b$ be tangent curvatures in a primitive Eisenstein circle packing. Then, $v_2(a+b)$ is either $0$ or odd.
\end{proposition}
\begin{proof}
    Complete $(a, b)$ into a primitive Eisenstein quadruple $(a, b, c, d)$ (not necessarily in standard position). If $a\not\equiv b\pmod{2}$, we are done, hence assume that $a,b$ have the same parity. If they are both odd, then by Proposition~\ref{prop:curvaturesmod4}, $a+b\equiv 2a\equiv 2\pmod{4}$, so $v_2(a+b)=1$ is odd, as desired.

    Thus, assume that $a,b$ are both even. If one is $0\pmod{4}$ and the other is $2\pmod{4}$, the result is immediate. So, assume that $a\equiv b\pmod{4}$. Then, $c$ is odd, hence 
    \[3a-b+2c\equiv 3a-a+2\equiv 2a+2\equiv 2\pmod{4},\]
    so $v_2(3a-b+2c)=1$. By Equation~\eqref{eqn:eisensteinmod6} and Lemma~\ref{lem:xsq3ysqprimedivisors}, this implies that $v_2(a+b)$ is odd, as desired.
\end{proof}

We can now prove the key result.

\begin{proposition}\label{prop:tangentcurvaturemod6sum}
    Let $a,b$ be tangent curvatures in a primitive Eisenstein circle packing. Then, $a+b\not\equiv 4, 5\pmod{6}$.
\end{proposition}
\begin{proof}
    Complete $(a, b)$ into a primitive Eisenstein quadruple $(a, b, c, d)$ (not necessarily in standard position). If $a+b\equiv 5\pmod{6}$, then there exists an odd prime $p\equiv 2\pmod{3}$ such that $v_p(a+b)$ is odd. By Lemma~\ref{lem:xsq3ysqprimedivisors}, $v_p((a+c-d)^2+3a^2)$ is even, so by Equation~\ref{eqn:eisensteinmod6}, $v_p(3a-b+2c)$ is odd. In particular, $p\mid a+b, 3a-b+2c, a+c-d, a$, whence it immediately follows that $p\mid a, b, c, d$. Therefore the quadruple was not primitive, contradiction.

     If $a+b\equiv 4\pmod{6}$, then $a+b\equiv 1\pmod{3}$ and is even. By Proposition~\ref{prop:tangentcurvature2adicvaluation}, it must have odd $2-$adic valuation, hence $a+b=2^rx$, where $r,x$ are odd positive integers. Thus, $x\equiv 2^{-r}\equiv 2\pmod{3}$ and is odd, so there exists an odd prime $p\equiv 2\pmod{3}$ such that $v_p(a+b)$ is odd. As in the first case, this gives a contradiction.
\end{proof}

Note that if $a+b\equiv 2\pmod{3}$ is \emph{even}, then it need not be divisible by an odd prime that is $2\pmod{3}$: it could be divisible by 2 to an odd power and a product of $1\pmod{3}$ primes. Furthermore, the final result in Lemma~\ref{lem:xsq3ysqprimedivisors} is false for $p=2$. This explains why we need to know information about the parity of $a+b$ for Proposition~\ref{prop:tangentcurvaturemod6sum}, and why the modulo 2 and 3 restrictions are not independent.

\begin{remark}\label{rem:mod6andstrongapprox}
    This may appear to contradict strong approximation (Section~\ref{sec:strongapproximation}), which is supported only on the prime $2$ for the group $\Egp''$.  However, the orbits of that group consist of \emph{non-tangent circles}.  It is only in taking the union of four orbits that we obtain tangent circles (Lemma~\ref{lem:Egpevenvsodd}), hence can discuss $a+b$.
\end{remark}

\section{Connection to quadratic forms}\label{sec:quadraticforms}

As with Apollonian circle packings, there is a direct connection between Eisenstein circle packings and binary quadratic forms. On the other hand, we require a slightly restricted set of quadratic forms, with a modified equivalence relation. We start by independently introducing this theory, before applying it to Eisenstein packings.

\subsection{Quadratic form theory}

\begin{definition}
    A \emph{first-odd} binary quadratic form is a function $f(x)=Ax^2+Bxy+Cy^2$ where $A, B, C\in\ZZ$, $\gcd(A, B, C)=1$, and $A$ is odd. Its discriminant is $D=B^2-4AC$. The form is \emph{positive definite} (respectively \emph{positive semidefinite}) if $D<0$ (respectively $D\leq 0$) and $A\geq 0$. We sometimes write $f=[A, B, C]$.
        
    The \emph{principal root} of a positive (semi)definite $f$ is
    \[r_f : = \begin{cases}\frac{-B+\sqrt{D}}{2A} & \text{if $A\neq 0$;}\\\infty & \text{if $A=0$.}\end{cases}\]
    In the case of positive definite forms, this is the unique upper half-plane root of $f(x, 1)$. 
\end{definition}

\begin{definition}
    Define
    \[\Gamma_0^{\PGL}(2):=\left\{\lm{a}{b}{c}{d}\in\PGL(2, \ZZ):2\mid c\right\}.\]
    This acts on first-odd binary quadratic forms via
    \begin{equation}\label{eqn:matrixactqf}
        \lm{a}{b}{c}{d}\circ f(x, y) := f(ax+by, cx+dy).
    \end{equation}
    The action preserves the discriminant.
\end{definition}

Denote the upper half-plane by $\uhp$. The action of $\gamma=\genmtx\in\PGL(2, \ZZ)$ on $z\in\uhp$ is 
\[\gamma(z):=\begin{cases}
\frac{az+b}{cz+d} & \text{if $\det(\gamma)=1$;}\\
\frac{a\overline{z}+b}{c\overline{z}+d} & \text{if $\det(\gamma)=-1$.}
\end{cases}\]

\begin{remark}
    Alternatively, if we allow for negative definite forms, we can now let $\PGL(2, \ZZ)$ act via M\"obius maps on $\widehat{\CC}$, considered as a union of the lower and upper half planes. The quadratic form action would also negate for matrices of determinant $-1$. Our choice is akin to ``folding'' the lower half plane onto the upper, and identifying the corresponding points. Considering the unfolded action may be preferable in certain applications.
\end{remark}

The action of $\PGL(2, \ZZ)$ on forms $f$ and their principal roots is anti-equivariant, i.e.
\[r_{\gamma\circ f} = \gamma^{-1}(r_f).\]

The action of $\Gamma_0^{\PGL}(2)$ on first-odd binary quadratic forms of discriminant $D$ separates them into finitely many equivalence classes, producing a modified class number (alas, the group structure is corrupted). In order to find the equivalence classes, we track the principal root, and reduce it to a fundamental domain for $\Gamma_0^{\PGL}(2)$.

\begin{proposition}\label{prop:gamma0pgl2generators}
    We have
    \[\Gamma_0^{\PGL}(2) = \left\langle\lm{1}{0}{0}{-1},\lm{1}{1}{0}{1}, \lm{1}{0}{2}{1}\right\rangle.\]
    It has index $3$ in $\PGL(2, \ZZ)$, and a fundamental domain for its action on $\uhp$ is the region given by:
    \[-\frac{1}{2}\leq x\leq 0,\quad \left(x+\frac{1}{2}\right)^2+y^2\geq \frac{1}{4},\]
    which is displayed in Figure~\ref{fig:fdomdomain}.
\end{proposition}

\begin{figure}[htb]
    \includegraphics[scale=1]{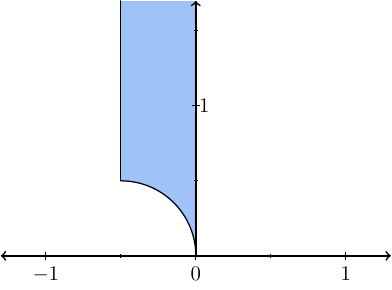}
    \caption{Fundamental domain for $\Gamma_0^{\PGL}(2)$.}\label{fig:fdomdomain}
\end{figure}

\begin{proof}
    Let $\Gamma=\left\langle\sm{1}{0}{0}{-1},\sm{1}{1}{0}{1}, \sm{1}{0}{2}{1}\right\rangle$, let $R$ denote the region claimed as a fundamental domain, and take $z\in\uhp$. First, we claim that we can $\Gamma-$reduce $z$ into $R$.

    Note that
    \[\lm{1}{0}{0}{-1}(z) = -\overline{z},\quad \lm{1}{1}{0}{1}(z) = z + 1,\quad \lm{1}{0}{2}{1}(z) = \frac{z}{2z+1}.\]
    Execute the following algorithm:
    \begin{enumerate}
        \item Apply $\sm{1}{1}{0}{1}^k$ to bring the real part of $z$ into the range $[-1/2, 1/2)$.
        \item If the real part is positive, apply $\sm{1}{0}{0}{-1}$ to bring the real part into the range $[-1/2, 0]$.
        \item If $\Im\left(\frac{z}{2z+1}\right)>\Im(z)$, apply $\sm{1}{0}{2}{1}$, and go back to step (1). Otherwise terminate.
    \end{enumerate}
    Let $z=x+yi$, and recall that
    \[\Im\left(\lm{a}{b}{c}{d}(z)\right) = \frac{y}{(cx+d)^2+(cy)^2}.\]
    Thus, for $\gamma\in\Gamma$, there are finitely many possible values of $\Im(\gamma(z))$ that are larger than any fixed positive real number. Steps (1) and (2) preserve imaginary part, while step (3) (if triggered), increases it. Thus, this can happen finitely many times, and the algorithm must terminate.

    At this point, we have reduced our point to $z$ where $-1/2\leq x\leq 0$ and $\Im\left(\frac{z}{2z+1}\right)\leq\Im(z)$. Since
    \[\Im\left(\frac{z}{2z+1}\right)=\frac{y}{(2x+1)^2+(2y)^2},\]
    this implies that $(x+1/2)^2+y^2\geq 1/4$, which gives the domain $R$. The hyperbolic area of $R$ is computed to be
    \[\int_{-1/2}^0\int_{\sqrt{1/4-(x+1/2)^2}}^{\infty}\frac{1}{y^2}dydx = \int_{-1/2}^0\frac{2}{\sqrt{1-(2x+1)^2}}dx=\int_0^1\frac{1}{\sqrt{1-t^2}}dt=\frac{\pi}{2}.\]
    Since the hyperbolic area of the fundamental domain of $\PGL(2, \ZZ)$ acting on $\uhp$ is $\frac{\pi}{6}=\frac{1}{3}\frac{\pi}{2}$, it follows that $[\PGL(2, \ZZ):\Gamma]\leq 3$. However, we know $\Gamma\leq \Gamma_0^{\PGL}(2)$, and by reducing modulo $2$, it follows that
    \[[\PGL(2, \ZZ):\Gamma]\geq [\PGL(2, \ZZ):\Gamma_0^{\PGL}(2)]\geq 3.\]
    Thus, equality holds, $\Gamma_0^{\PGL}(2)=\Gamma$ has index 3 in $\PGL(2, \ZZ)$, and the fundamental domain is as claimed.
\end{proof}

By tracking the principal root, we can reduce a positive definite first-odd binary quadratic form so that the principal root lies in the above fundamental domain. This corresponds to the equations
\[-\frac{1}{2}\leq\frac{-B}{2A}\leq 0,\quad \left(\frac{A-B}{2A}\right)^2 + \left(\frac{\sqrt{4AC-B^2}}{2A}\right)^2\geq\frac{1}{4}.\]
In turn, these are equivalent to $0\leq B\leq A, 2C$.

\begin{definition}
    Call a positive semidefinite first-odd binary quadratic form $[A, B, C]$ \emph{reduced} if $0\leq B\leq A, 2C$.
\end{definition}

As we have seen, every positive definite first-odd binary quadratic form is equivalent to a unique reduced form. A direct computation with $D=0$ shows that every such form can be reduced to $[1, 0, 0]$, the only reduced form of this discriminant, which extends the result to semidefinite forms.

\begin{definition}
    For $D\leq 0$ a discriminant, let $\mathcal{F}(D)$ denote the set of equivalence classes of positive semidefinite first-odd binary quadratic forms of discriminant $D$.
\end{definition}

\begin{proposition}
    The set $\mathcal{F}(D)$ is finite.
\end{proposition}
\begin{proof}
    It suffices to count the reduced forms. We have
    \[B^2=D+4AC\geq D+2B^2,\]
    hence $0\leq B\leq \sqrt{-D}$. This gives finitely many possible $B$ values. For each such $B$, we have $AC=\frac{B^2-D}{4}$. Since $A,C$ are integral, there are only finitely many solutions to this equation, giving finitely many reduced forms.
\end{proof}

In fact, we can do better by finding a formula for $|\mathcal{F}(D)|$ in terms of the class number corresponding to $D$, the genus, and the factorization of $D$. For Eisenstein circle packings, we care about the sets $\mathcal{F}(-3n^2)$, where we can make the formula explicit from the factorization of $n$.

\subsection{Relating $\mathcal{F}(D)$ to classical quadratic forms}

First, we recall some important aspects of classical quadratic form theory, which can be found in (for example) \cite[Chapters 2 and 3]{Coxprimesx2ny2}. Consider the set $S(D)$ of all primitive binary quadratic forms of discriminant $D<0$. Take this up to the action of $\PSL(2, \ZZ)$ (which acts identically to Equation~\eqref{eqn:matrixactqf}). This can be given the structure of an abelian group (called the class group). Important to us will be the following:
\begin{itemize}
    \item A form $[A, B, C]$ is called reduced whenever (a) $|B|\leq A\leq C$, and (b) if $|B|=A$ or $A=C$, then $B\geq 0$.
    \item Each quadratic form is equivalent to a unique reduced form.
    \item The number of equivalence classes of discriminant $D$ is called the class number, and denoted by $h(D)$. Denote the number of forms of order at most 2 by $h_2(D)$.
    \item A form has order at most 2 in the class group if and only if $B=0$, $A=B$, or $A=C$ \cite[Lemma 3.10]{Coxprimesx2ny2}.
    \item There is a formula for $h_2(D)$. Let $t=\omega(D)$ denote the number of distinct prime divisors of the discriminant, and if $D$ is odd, define $\mu=t$. If $D$ is even, let $D=-4n$ for an integer $n>0$, and define
    \begin{equation}\label{eqn:h2formula}
    \mu = \begin{cases} t & \text{if $n\equiv 2,3\pmod{4}$ or $n\equiv 4\pmod{8}$;}\\t+1 & \text{if $n\equiv 1\pmod{4}$ or $n\equiv 0\pmod{8}$.}\end{cases}
    \end{equation}
    Then, the number of elements of order at most 2 is exactly $2^{\mu-1}$ \cite[Proposition 3.11]{Coxprimesx2ny2}. 
    \item There is a formula relating $h(n^2D)$ to $h(D)$ \cite[Theorem 7.24]{Coxprimesx2ny2}. Specializing to the case of $D=-3$ yields for $n>1$,
    \begin{equation}\label{eqn:hformula}
        h(-3n^2)=\frac{n}{3}\prod_{p\mid n}\left(1-\left(\frac{-3}{p}\right)\frac{1}{p}\right)= \frac{n}{3}\prod_{p\mid n}\begin{cases} 1 & \text{if $p=3$;}\\ \frac{p-1}{p} & \text{if $p\equiv 1\pmod{3}$;}\\ \frac{p+1}{p} & \text{if $p\equiv 2\pmod{3}$.}\end{cases}
    \end{equation}
\end{itemize}

For the equivalence class $\mathcal{F}(D)$, we are considering a restricted set of forms (first coefficient must be odd), but also considering a different matrix group, $\Gamma_0^{\PGL}(2)$. To relate these distinct situations, we first go from the classical $\PSL$ action to $\PGL$.

\begin{lemma}
    The number of equivalence classes of $\PGL(2, \ZZ)\backslash S(D)$ is
    \[\frac{h(D)+h_2(D)}{2}.\]
\end{lemma}
\begin{proof}
    The group $\PGL(2, \ZZ)$ is generated by $\PSL(2, \ZZ)$ and $\sm{1}{0}{0}{-1}$. Note that
    \[\lm{1}{0}{0}{-1}\circ [A, B, C] = [A, -B, C].\]
    In particular, by considering only the reduced forms, this pairs up distinct equivalence classes, except when $B=0$, $A=B$, or $A=C$. In this case, these classes were already equivalent, and this corresponds to exactly the order at most 2 forms. In total, we get
    \[\frac{h(D)-h_2(D)}{2}+h_2(D)=\frac{h(D)+h_2(D)}{2}\]
    classes.
\end{proof}

While the $\PGL(2, \ZZ)\backslash S(D)$ forms no longer carry a natural group law, we can still refer to the ``order at most two forms'', and there remains $h_2(D)$ of them. Additionally, we can refer to a ``$\PGL(2, \ZZ)$-reduced form'' as satisfying $0\leq B\leq A\leq C$: each class has a unique representative of this form.

Since $\Gamma_0^{\PGL}(2)\leq\PGL(2, \ZZ)$, we can consider the map $\theta:\mathcal{F}(D)\rightarrow \PGL(2, \ZZ)\backslash S(D)$ defined by
\[\theta(\Gamma_0^{\PGL}(2)\circ f) = \PGL(2, \ZZ)\circ f.\]
It suffices to analyze the range and fibres of this map. First, by Proposition~\ref{prop:gamma0pgl2generators}, $\Gamma_0^{\PGL}(2)$ has index 3 in $\PGL(2, \ZZ)$. It can be checked that three left coset representatives are
\[I = \lm{1}{0}{0}{1},\quad U = \lm{0}{1}{1}{0},\quad V = \lm{1}{1}{-1}{0}.\]
Therefore, for every $f\in S(D)$, the class $\PGL(2, \ZZ)\circ f$ has three possible preimages under $\theta$: $I\Gamma_0^{\PGL}(2)\circ f$, $U\Gamma_0^{\PGL}(2)\circ f$, and $V\Gamma_0^{\PGL}(2)\circ f$. It suffices to compute which of these represent first-odd binary quadratic forms, and of those that do, how many are distinct. Assume that $f=[A, B, C]$ is given by its unique $\PGL(2, \ZZ)-$reduced form, hence $0\leq B\leq A\leq C$. We compute that
\[I\circ f = [A, B, C],\qquad U\circ f = [C, B, A], \qquad V\circ f = [A-B+C, 2A-B, A].\]
Call this new form $[A', B', C']$, and for each of them we see that $0\leq B'\leq A', 2C'$, whence it is $\Gamma_0^{\PGL}(2)-$ reduced. Furthermore,
\begin{itemize}
    \item $I$ and $U$ represent the same class if and only if $A=C$;
    \item $I$ and $V$ represent the same class if and only if $A=B=C$;
    \item $U$ and $V$ represent the same class if and only if $A=B$.
\end{itemize}
Since the form is primitive, $A=B=C$ is only possible if $A=B=C=1$, giving $D=-3$. In particular, for $D\neq -3$, these represent either 2 or 3 distinct classes, though we are currently ignoring the first-odd requirement. We add this back in over a series of lemmas.

\begin{lemma}\label{lem:D5mod8}
    If $D\equiv 5\pmod{8}$ and $D\neq -3$, then
    \[|\mathcal{F}(D)|=\frac{3h(D)+h_2(D)}{2}.\]
\end{lemma}
\begin{proof}
    Let $f=[A, B, C]$, and since $D$ is odd, $B$ must be odd. As $AC=\frac{B^2-D}{4}$, it follows that $B^2-D\equiv 4\pmod{8}$, so $AC$ is odd, hence $A$ and $C$ are odd. In particular, all forms of this discriminant are already first-odd!

    Therefore all three possible preimages represent classes in $\mathcal{F}(D)$, so it suffices to determine which are equivalent. Since $D\neq -3$, the above work implies that there are two distinct classes if $A=B$ or $A=C$, and three classes otherwise. However, the order at most two elements correspond to $A=B$, $A=C$, and $B=0$. Since $B$ is necessarily odd, there are no such elements with $B=0$, whence we get 2 preimages for the order at most two elements, and 3 otherwise. This implies that
    \[|\mathcal{F}(D)|=2h_2(D)+3\left(\frac{h(D)+h_2(D)}{2}-h_2(D)\right) = \frac{3h(D)+h_2(D)}{2},\]
    as desired.
\end{proof}

If $D$ is even, not all preimages will be first-odd.

\begin{lemma}\label{lem:D0mod4}
    If $D\equiv 0\pmod{4}$ and $D<-4$, then
    \[|\mathcal{F}(D)|=h(D)+2^{\omega(-D/4)-1}.\]
\end{lemma}
\begin{proof}
    Let $f=[A, B, C]$, and since $D$ is even, $B$ is even. Since the form is primitive, it follows that
    \[f\equiv\quad [0,0,1],\quad [1,0,0],\quad [1,0,1]\pmod{2}.\]
    In each case, we compute that exactly two of $I\circ f$, $U\circ f$, and $V\circ f$ represent first-odd forms, hence we generically have two preimages. If our form did not have order at most 2, then these must be distinct. If it had order at most 2, then they are equal, except those forms with $B=0$ and $A\neq C$, which still have 2 distinct preimages. The only primitive form with $B=0$ and $A=C$ is $[1, 0, 1]$, having discriminant $-4$, which is excluded. Therefore, we just need to count $\PGL(2, \ZZ)-$reduced forms $[A, 0, C]$ of discriminant $D$.

    This corresponds to solving $-4AC=D$, so $AC=-D/4$ in coprime integers $(A, C)$. There are $2^{\omega(-D/4)}$ such solutions, and exactly half of them will have $A\leq C$, giving $2^{\omega(-D/4)-1}$ solutions.

    Putting this together gives
    \[|\mathcal{F}(D)| = 2\left(\frac{h(D)+h_2(D)}{2}-h_2(D)+2^{\omega(-D/4)-1}\right) + 1(h_2(D)-2^{\omega(-D/4)-1}) = h(D) + 2^{\omega(-D/4)-1},\]
    as desired.
\end{proof}

For completeness, we include the $D\equiv 1\pmod{8}$ case, even through it won't be relevant for Eisenstein circle packings, as $-3n^2\not\equiv 1\pmod{8}$.

\begin{lemma}\label{lem:D1mod8}
    If $D\equiv 1\pmod{8}$, then
    \[|\mathcal{F}(D)|=\frac{h(D)+h_2(D)}{2}.\]
\end{lemma}
\begin{proof}
    Let $f=[A, B, C]$, and since $D$ is odd, $B$ must be odd. As $AC=\frac{B^2-D}{4}$ and $B^2-D\equiv 0\pmod{8}$, $AC$ is even. Thus,
    \[f\equiv\quad [0,1,0],\quad [0,1,1],\quad [1,1,0]\pmod{2}.\]
    In each case, we compute that exactly one of $I\circ f$, $U\circ f$, and $V\circ f$ represents first-odd forms, so the class number remains unchanged.
\end{proof}

Combining these results together gives the final explicit formula for $|\mathcal{F}(-3n^2)|$.

\begin{corollary}\label{cor:exactquadraticformcount}
    If $n$ is odd, then
    \[|\mathcal{F}(-3n^2)|=\frac{n}{2}\prod_{p\mid n}\left(1-\left(\frac{-3}{p}\right)\frac{1}{p}\right) + 2^{\omega(3n)-2}.\]
    If $n$ is even, then
    \[|\mathcal{F}(-3n^2)|=\frac{n}{3}\prod_{p\mid n}\left(1-\left(\frac{-3}{p}\right)\frac{1}{p}\right) + 2^{\omega{(3n/2)}-1}.\]
\end{corollary}
\begin{proof}
    If $n$ is odd, then $-3n^2\equiv 5\pmod{8}$, and if $n$ is even, then $-3n^2\equiv 0\pmod{4}$. Combining Lemmas~\ref{lem:D5mod8} and~\ref{lem:D0mod4} with Equations~\eqref{eqn:h2formula} and \eqref{eqn:hformula} yields the result for $n\neq 1$. For $n=1$, the formula is verified directly.
\end{proof}

\subsection{Bijection with primitive Eisenstein packings}

\begin{proposition}\label{prop:eisensteinpackingbijectionqf}
    There is a bijection between primitive Eisenstein quadruples with first entry $a$ in standard position and primitive positive semidefinite first-odd binary quadratic forms of discriminant $-3a^2$. The bijection is realized by
    \[\phi(a, b, c, d) := \left[a + d, b + d - c, \frac{a+b}{2}\right],\]
    \[\theta[A, B, C] := (a, 2C-a, A-B+2C-2a, A-a).\]
\end{proposition}
\begin{proof}
    Let $q=(a, b, c, d)$ and $Q=[A, B, C]$. It is easy to check that $\phi$ and $\theta$ are inverses to each other (as linear maps), and we compute the discriminant of $\phi(q)$ to be
    \[(b+d-c)^2-2(a+d)(a+b)=-3a^2+a^2+b^2+c^2+d^2-2(a+c)(b+d)=-3a^2.\]
    Similarly, the Eisenstein equation for $\theta(Q)$ is satisfied.

    As $q$ is in standard position, $a+d$ is odd and $\frac{a+b}{2}\in\ZZ$, so $\phi(q)$ is first-odd and integral. By Lemma~\ref{lem:pairsnonneg}, $a+d>0$, hence it is positive semidefinite. For primitivity, assume $p$ is a prime such that $p\mid \phi(q)$. Thus $p\mid -3a^2$, so $p=3$ or $p\mid a$. In the second case, it immediately follows that $p\mid (a+d)-a=d$, $p\mid (a+b)-a=b$, and $p\mid b+d-(b+d-c)=c$, so $q$ was not primitive, contradiction. Otherwise, $p=3$, and it follows that $(a, b, c, d)\equiv (a, -a, a, -a)\pmod{3}$. However, this contradicts Proposition~\ref{prop:mod9}, as this behaviour cannot lift to modulo $9$, and therefore the integers. Thus, $\phi$ does produce a primitive positive semidefinite first-odd binary quadratic form.

    For $\theta$, it is clear that $\theta(Q)$ is in standard position, and for primitivity, assume that $p\mid\theta(Q)$. It follows that $p\mid a, A, 2C, B$. As $A$ is odd, $p\neq 2$, so $p\mid A, B, C$, proving that $Q$ was not primitive.

    Putting this all together shows that we have well-defined bijections between the two sets, as claimed.
\end{proof}

Under these bijections, reduced quadratic forms correspond to reduced quadruples!

\begin{proposition}\label{prop:quadruplequadraticbijection}
    Let $q=(a, b, c, d)$ be a primitive Eisenstein quadruple. Then $S_2, S_3, S_4$ are all simultaneously non-decreasing for $q$ if and only if $\phi(q)$ is reduced. In particular, if $a\leq 0$, then the set of primitive Eisenstein circle packings of outer curvature $a$ bijects with reduced positive semidefinite first-odd binary quadratic forms of discriminant $-3a^2$.
\end{proposition}
\begin{proof}
    Let $\phi(q)=[A, B, C]$, where
    \[A=a+d,\quad B=b+d-c,\quad C=\frac{a+b}{2}.\]
    We see that $S_2$ is non-decreasing if and only if $2a+2c-b\geq b$, which is equivalent to $B=b+d-c\leq a+d=A$. Next, $S_3$ is non-decreasing if and only if $2b+2d-c\geq c$, which is equivalent to $B=b+d-c\geq 0$. Finally, $S_4$ is non-decreasing if and only if $2a+2c-d\geq d$, which is equivalent to $B=b+d-c\leq a+b=2C$. These are the three inequalities defining a reduced quadratic form, giving the first claim. The second claim follows immediately, since if $a\leq 0$, then $S_1$ is also non-decreasing.
\end{proof}

Combining Corollary~\ref{cor:exactquadraticformcount} with Proposition~\ref{prop:quadruplequadraticbijection}  proves Theorem~\ref{thm:countcirclepackingsoutercurvn}.  

As in the Apollonian case, the power of the Proposition~\ref{prop:quadruplequadraticbijection} bijection goes further: we can describe the multiset of curvatures of circles tangent to a fixed circle. Indeed, take an Eisenstein circle packing, fix a circle of curvature $a$, and consider the multiset of Eisenstein quadruples containing that circle. Up to reordering, this is equivalent to fixing an Eisenstein quadruple $q=(a, b, c, d)$ in standard position, and considering the orbit under the swaps $S_2, S_3, S_4$.

\begin{definition}
    The \emph{reduced Eisenstein swap group} is $\Egp_1:=\langle S_2, S_3, S_4\rangle\leq \GL(4, \ZZ)$.
\end{definition}

The orbit of $\Egp_1$ on $q$ is equivalent to taking a $\Gamma_0^{\PGL}(2)$ orbit of $\phi(q)$.

\begin{proposition}\label{prop:quadraticformtangentcircles}
    The multiset $\phi(q\Egp_1)$ is equal to the multiset $\Gamma_0^{\PGL}(2)(\phi(q))$.
\end{proposition}
\begin{proof}
    Let $q=(a, b, c, d)$, and recall that
    \[Q=\phi(q)=\left[a+d, b+d-c, \frac{a+b}{2}\right]=[A, B, C].\]
    The action of $S_2$ gives
    \[\phi(S_2(q)) = \left[a+d, 2a-b+c+d, \frac{3a-b+2c}{2}\right] = [A, 2A-B, A-B+C] = \lm{1}{1}{0}{-1}\circ Q.\]
    The action of $S_3$ gives
    \[\phi(S_3(q)) = \left[a+d, -b+c-d, \frac{a+b}{2}\right] = [A, -B, C] = \lm{1}{0}{0}{-1}\circ Q.\]
    The action of $S_4$ gives
    \[\phi(S_4(q)) = \left[3a+2c-d, 2a+b+c-d, \frac{a+b}{2}\right] = [A-2B+4C, -B+4C, C] = \lm{1}{0}{-2}{-1}\circ Q.\]
    Since
    \[\left\langle\lm{1}{0}{0}{-1}, \lm{1}{1}{0}{-1}, \lm{1}{0}{-2}{-1}\right\rangle = \left\langle\lm{1}{0}{0}{-1},\lm{1}{1}{0}{1}, \lm{1}{0}{2}{1}\right\rangle = \Gamma_0^{\PGL}(2),\]
    the result follows.
\end{proof}

Let $[A, B, C]=\sm{e}{f}{g}{h}\circ\phi(q)$, where $\sm{e}{f}{g}{h}\in\Gamma_0^{\PGL}(2)$. It follows that $\theta([A, B, C])=(a, 2C-a, A-B+2C-2a, A-a)$ is an Eisenstein quadruple in the same packing. In particular, $2C-a$ represents a curvature of a circle tangent to $a$ with the same parity, and $A-a$ represents a curvature of a circle tangent to $a$ with opposite parity.

\begin{proposition}\label{prop:curvaturesoftangentcircles}
    The multiset of curvatures of circles tangent to our fixed circle of curvature $a$ is a union of two multisets, split by parity. Those curvatures which have the same parity as $a$ biject with the multiset
    \[\{2\phi(q)(x, y)-a:\gcd(x, y)=1,\text{$y$ is odd}\}.\]
    Those curvatures which have the opposite parity as $a$ biject with the multiset
    \[\{\phi(q)(x, y)-a:\gcd(x, y)=1,\text{$y$ is even}\}.\]
\end{proposition}
\begin{proof}
As seen above, we need to determine all possible values of $A,C$ for $\Gamma_0^{\PGL}(2)$-translates of $\phi(q)$. Note that
\[A=[A, B, C](1, 0) = \left(\lm{e}{f}{g}{h}\circ\phi(q)(x, y)\right)(1, 0) = (\phi(q)(ex+fy, gx+hy))(1, 0)= \phi(q)(e, g).\]
Similarly, $C=\phi(q)(f, h)$. Thus, all possible $A$ values come from a choice of coprime $e,g\in\ZZ$. We can find a corresponding $f,h\in\ZZ$ with $\sm{e}{f}{g}{h}\in\Gamma_0^{\PSL}(2)$, if and only if $g$ is even. The set of all possible such pairs $(f, h)$ corresponds to the column replacement $(f, h)\rightarrow (f+ke, h+kg)$, which in turn corresponds to the actions of $S_2$ and $S_3$, i.e. the stabilizer of the circle of curvature $A-a$ inside $q\Egp_1$. This gives the second claim.

For the first claim, note that since $eh-fg=1$ and $2\mid g$, $h$ must be odd. Given any coprime $(f, h)$ with odd $h$, we can complete it to $\sm{e}{f}{g}{h}\in\Gamma_0^{\PGL}(2)$. The set of all possible pairs $(e, g)$ corresponds to the column replacement $(e, g)\rightarrow (e + 2kf, g + 2kh)$, which in turn corresponds to the actions of $S_3$ and $S_4$, i.e. the stabilizer of the circle of curvature $2C-a$ inside $q\Egp_1$. This gives the first claim.
\end{proof}

\section{Schmidt arrangement}\label{sec:schmidt}

We pause our study of Eisenstein circle packings for an interlude to demonstrate that they arise as immediate tangency packings in the Eisenpint Schmidt arrangement.

\subsection{Eisenstein circles and tangencies}

Recall that an oriented circle is parametrized by the triple $(u, v, w)$, where $u,v\in\RR$ and $w\in\CC$ satisfy $uv=|w|^2-1$. There is a nice formula for this triple when $\widehat{\RR}$ is acted on by a M\"obius map.

\begin{proposition}[\upshape{Proposition 3.7 of \cite{StangeVisualizingArith}}]\label{prop:mobiuscurv}
    Let\footnote{Note that $\beta$ and $\gamma$ are transposed from their usual locations, to remain consistent with the second author's somewhat unhinged convention in \cite{StangeVisualizingArith}.} $M=\sm{\alpha}{\gamma}{\beta}{\delta}\in\SL(2, \CC)$. Then the oriented circle $M(\widehat{\RR})$ satisfies
    \[u = i(\beta\ol{\delta}-\ol{\beta}\delta), \qquad v=i(\alpha\ol{\gamma}-\ol{\alpha}\gamma),\qquad w=i(\alpha\ol{\delta}-\gamma\ol{\beta}).\]
\end{proposition}

Specializing Proposition~\ref{prop:mobiuscurv} to the case of $\PSL(2, \ZZ[\om])$ produces the following result.

\begin{proposition}\label{prop:eisensteincircleformulae}
    Let 
    \[M=\lm{a+b\om}{c+d\om}{e+f\om}{g+h\om}\in\PSL(2, \ZZ[\om]),\]
    and assume that $M(\widehat{\RR})$ produces the tuple $(u, v, w)$, where $u$ is the curvature, $v$ is the co-curvature, and $w$ is the curvature-centre. Then
    \[u = s\sqrt{3}, \qquad v = t\sqrt{3},\qquad w=\frac{x\sqrt{3}+yi}{2},\quad \text{where}\quad s,t,x,y\in\ZZ,\quad x\equiv y\pmod{2},\quad y\equiv 2\pmod{3}.\]
    The terms $s,t,x,y$ can be computed with the formulae
    \[s = eh-fg,\qquad t = ad-bc, \qquad x = ah-bg+de-cf, \qquad y = 3ag - 3ce - 1,\]
    and they satisfy $3x^2+y^2-12st=4$.
\end{proposition}

The modulo 2 condition on $y$ can be removed by replacing $y$ with $z=\frac{x-y}{2}$, where the modulo 3 condition becomes $x+z\equiv 2\pmod{3}$.

\begin{definition}\label{def:schmidtreducedquadruple}
    The circle $M(\widehat{\RR})$ is called an \emph{Eisenstein circle}. We call $s$ the \emph{reduced curvature}, $t$ the \emph{reduced co-curvature}, and $(s, t, x, z)$ is the \emph{reduced coordinates} of $M(\widehat{\RR})$. This satisfies 
    \begin{equation}
        \label{eqn:sheet}
        x^2-xz+z^2-1 = 3st,\qquad x+z\equiv 2\pmod{3}.
    \end{equation}
    The inversive distance in terms of reduced coordinates is
    \[\langle (s_1, t_1, x_1, z_1), (s_2, t_2, x_2, z_2)\rangle = \frac{-3s_1t_2-3s_2t_1-x_1z_2-x_2z_1}{2}+x_1x_2+z_1z_2.\]
\end{definition}

A converse to Proposition~\ref{prop:eisensteincircleformulae} holds, and gives a convenient description of these circles. 
Before we give this in Proposition~\ref{prop:eisensteincircleconverse}, we do a computation to demonstrate how the coordinates $(s, t, x, z)$ transform under translations and inversion in $\PSL(2, \ZZ[\om])$.

\begin{proposition}\label{prop:stxzbasictransformations}
    Let the circle $\cir$ correspond to the reduced coordinates $(s, t, x, z)$. The image of $\cir$ under translation by $e+f\om\in\RR[\om]$ is
    \[(s, t+(e^2+ef+f^2)s+(e+f)x-fz, x+(2e+f)s, z+(e-f)s).\]
    The image of $\cir$ under the inversion $\sm{0}{-1}{1}{0}$ is
    \[(t, s, -x, z-x).\]
    In particular, if $s, t, x, z, e, f\in\ZZ$, then the image is again an integral 4-tuple.
\end{proposition}
\begin{proof}
    Note that the algebraic formulae in Proposition~\ref{prop:eisensteincircleformulae} and below still hold even if our coefficients lie in $\RR$ instead of $\ZZ$. To this end, we can write 
    \[\cir = M(\widehat{\RR}) = \lm{A+B\om}{C+D\om}{E+F\om}{G+H\om}(\widehat{\RR}),\quad A,B,C,D\in\RR,\]
    multiply $M$ on the left by $\sm{1}{e+f\om}{0}{1}$ (for translation) or $M=\sm{0}{-1}{1}{0}$ (for inversion), and verify the claimed formulae.

    As the outputs are integral polynomials in $s, t, x, z, e, f$, if the inputs are integral, the outputs remain integral.
\end{proof}

A priori, given an integral solution $(s, t, x, z)$ to Equation~\eqref{eqn:sheet}, we don't know if it corresponds to reduced coordinates from $\PSL(2, \ZZ[\om])$. However, by  Proposition~\ref{prop:stxzbasictransformations}, we can perform a reduction process, where we translate, invert, and repeat, until we end back with $\widehat{\RR}$. As we are only applying $\PSL(2, \ZZ[\om])$ transformations, the result will follow.

\begin{lemma}\label{lem:reducedquadruplereduction}
Let $(s, t, x, z)$ be any integral solution to $x^2-xz+z^2-1=3st$ which corresponds to the circle $\cir$ and satisfies $s>0$. Then, there exists an $M\in\PSL(2, \ZZ[\om])$ for which $M(\cir)$ has reduced coordinates $(s', t', x', z')$ where $s',t',x',z'\in\ZZ$ and $0\leq s'<s$.
\end{lemma}

\begin{proof}
As $s>0$, $\cir$ is a circle with finite radius. We can translate the circle by $\ZZ[\om]$ so that the the centre lies in the parallelogram bounded by $-1, 0, \om, -1+\om$. If the centre of the circle is $(c_x, c_y)$, note that $c_x^2+c_y^2\leq 1$, as this fundamental domain lies entirely inside the unit circle. This is depicted in Figure~\ref{fig:circlereduction}.

\begin{figure}[htb]
    \includegraphics[scale=1.0]{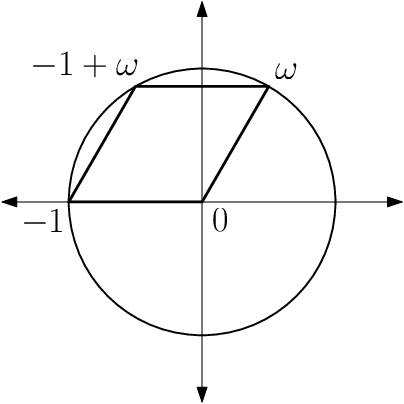}
    \caption{Each circle with non-zero curvature can be translated by $\ZZ[\om]$ so that the centre lies in this parallelogram, which lies inside the unit circle.}\label{fig:circlereduction}
\end{figure}

Assume that the translated circle has reduced coordinates $(S, T, X, Z)$, where by Proposition~\ref{prop:stxzbasictransformations}, $S, T, X, Z\in\ZZ$ and $S=s$. Using the transformation of $Y=X-2Z$, we have that $3X^2+Y^2=4+12sT$. Furthermore, by the formulae of Proposition~\ref{prop:eisensteincircleformulae}, $c_x=\frac{X}{2s}$ and $c_y=\frac{Y}{2\sqrt{3}s}$. Thus,
\[1\geq c_x^2+c_y^2=\frac{X^2}{4s^2}+\frac{Y^2}{12s^2},\]
hence
\[12s^2\geq 3X^2+Y^2=12sT+4>12sT,\]
whence $s>T$. As $12sT=3X^2+Y^2-4\geq -4$ is also a multiple of 12, $12sT\geq 0$, whence $s>T\geq 0$. Applying an inversion by $\sm{0}{-1}{1}{0}$ gives
where $0\leq s'=T<s$. This was obtained by the action of an $M$ constructed as a translation by $\ZZ[\om]$ and then inversion by $\sm{0}{-1}{1}{0}$, whence $M\in\PSL(2, \ZZ[\om])$, as desired.
\end{proof}

By iterating this result, we can reduce any integral solution back to $(0, 0, 0, -1)$, i.e. $\widehat{\RR}$. The following result is analogous to \cite[Definition 3.1 and Theorem 3.11]{MartinGeneralSchmdit}.

\begin{proposition}\label{prop:eisensteincircleconverse}
    Let $(s, t, x, z)$ be any integral solution to $x^2-xz+z^2-1=3st$. Then there exists a choice of $\pm\in\{-, +\}$ and an $M\in\PSL(2, \ZZ[\om])$ for which $\pm(s, t, x, z)$ is the reduced coordinates of $M(\widehat{\RR})$. In particular, the Eisenstein circles of curvature $s\sqrt{3}$ are exactly those centred at 
    \[\frac{x}{2s}+\frac{x-2z}{2s\sqrt{3}}i\]
    for all $(x, z)\in\ZZ^2$ for which $3s\mid x^2-xz+z^2-1$.
\end{proposition}
\begin{proof}
    If $s<0$, replace $(s, t, x, z)$ with $-(s, t, x, z)$. By Lemma~\ref{lem:reducedquadruplereduction} we can apply a $\PSL(2, \ZZ[\om])$ transformation to reduce the value of $s$. Iterating, we eventually find an $M_1\in\PSL(2, \ZZ[\om])$ which produces coordinates $(s_1, t_1, x_1, z_1)$ where $s_1=0$. Writing $y_1=x_1-2z_1$, we have $3x_1^2+y_1^2=4+12s_1t_1=4$. By possibly negating the coordinates again, we can assume that $y_1\equiv -1\pmod{3}$, giving the solutions $(x_1, y_1)=(\pm 1, -1), (0, 2)$. Translating back to $z_1$, we have the coordinates
    \[(s_1, t_1, x_1, z_1) =\quad (0, t_1, 1, 1), \quad (0, t_1, -1, 0), \quad (0, t_1, 0, -1).\]
    These are the reduced coordinates for  $M(\widehat{\RR})$, where the matrices $M\in\PSL(2, \ZZ[\om])$ are
    \[\lm{1-\om}{t_1\om}{0}{\om}, \quad \lm{\om}{-t_1}{0}{1-\om},\quad \lm{1}{t_1\om}{0}{1},\]
    respectively. Thus, $M_1^{-1}M(\widehat{\RR})$ has reduced coordinates $\pm(s, t, x, z)$, proving it corresponds to an Eisenstein circle.
\end{proof}

Note that the circle corresponding to $-(s, t, x, z)$ is the opposite orientation of the circle corresponding to $(s, t, x, z)$. Furthermore, it is easy to tell if integral reduced coordinates corresponds to a circle in $\PSL(2, \ZZ[\om])(\widehat{\RR})$, or the opposite orientation of such a circle. Indeed, we are in $\PSL(2, \ZZ[\om])(\widehat{\RR})$ if $x+z\equiv 2\pmod{3}$, and are the opposite orientation if $x+z\equiv 1\pmod{3}$.

\subsection{Reducing circles to $\widehat{\FF_2[\om]}$}

In order to understand the intersection and tangency structure of the full Schmidt arrangement, we reduce the picture modulo 2. It turns out that the reduced tangency structure lifts to the global circles.

In general, if $F$ is a number field whose ring of integers $\mathcal{O}_F$ is a unique factorization domain, then we can identify the projectivization $\PP^1(F)$ with $\widehat{F} = F\cup \{\infty\}$ via $(a:b)\rightarrow \frac{a}{b}$. Thus, given a prime $\mathfrak{p}$ of $\mathcal{O}_F$, there exists a well-defined notion of reduction modulo $\mathfrak{p}$ in $\widehat{F}$: scale the projectified point to be integral in both entries and coprime to $\mathfrak{p}$ in at least one, and reduce these entries modulo $\mathfrak{p}$.

First, we show that the $\widehat{\QQ(\om)}-$points on the circle $M(\widehat{\RR})$ are the images of $\widehat{\QQ}$.

\begin{lemma}\label{lem:Qomegarationalpoints}
    Let $M\in\PSL(2, \ZZ[\om])$. Then $\widehat{\QQ(\om)}\cap M(\widehat{\RR}) = M(\widehat{\QQ})$.
\end{lemma}
\begin{proof}
    For $x\in\widehat{\CC}$, it is clear that $M(x)\in\widehat{\QQ(\om)}$ if and only if $x\in\widehat{\QQ(\om)}$. Since $\widehat{\QQ(\om)}\cap\widehat{\RR}=\widehat{\QQ}$, the lemma follows.
\end{proof}

The prime $2$ is inert in $\ZZ[\om]$, hence reduction modulo 2 gives
\[\rho_2:\ZZ[\om]\rightarrow \FF_2[\om]\cong\FF_4.\]
As $\ZZ[\om]$ is a unique factorization domain, this reduction map extends to $\widehat{\QQ(\om)}$:
\[\rho_2:\widehat{\QQ(\om)}\rightarrow\widehat{\FF_2[\om]}.\]
This reduced space has 5 points,
\[\widehat{\FF_2[\om]} = \{ 0, 1, \om, 1 + \om, \infty \}.\]
The group $\PSL(2, \ZZ[\om])$ also reduces nicely to an order 60 group
\[\rho_2:\PSL(2, \ZZ[\om])\rightarrow \PSL(2, \FF_2[\om]) = \SL(2, \FF_2[\om]),\]
which again acts on $\widehat{\FF_2[\om]}$ via M\"obius transformations. 

\begin{definition}
    A $\FF_2[\om]-$circle is the image $\rho_2\left(\widehat{\QQ(\om)}\cap M(\widehat{\RR})\right)$ for some $M\in\PSL(2, \ZZ[\om])$. In other words, we take a circle in the Eisenstein Schmidt arrangement, consider the $\widehat{\QQ(\om)}-$points on the circle, and reduce them modulo 2 to $\widehat{\FF_2[\om]}$.
\end{definition}

Circles in $\PP^1(\CC)$ are defined by 3 points, and the same is true for $\FF_2[\om]-$circles.

\begin{lemma}
    The $\FF_2[\om]-$circles are exactly those subsets of $\widehat{\FF_2[\om]}$ of size $3$.
\end{lemma}

\begin{proof}
    By Lemma~\ref{lem:Qomegarationalpoints}, it suffices to find the reduction of $\widehat{\QQ}$, and take all M\"obius images under $\SL(2, \FF_2[\om])$. This reduction is computed to be
    \[\rho_2\left(\widehat{\QQ}\right) = \{\infty, 0, 1\}\]
    a set of three points. As $\SL(2, \FF_2[\om])$ acts triply transitively on points, the result follows.
\end{proof}

Take $\cir_0$ to be the $\FF_2[\om]-$circle $\{\infty, 0, 1\}$; there are ${5 \choose 3}=10$ $\FF_2[\om]-$circles.

\begin{lemma}
    Let $M\in\PSL(2, \ZZ[\om])$. A circle $M(\widehat{\RR})$ reduces to $\cir_0$ if and only if $M\in\PSL(2, \ZZ[2\om])$. 
\end{lemma}
\begin{proof}
    The circle $M(\widehat{\RR})$ reduces to $\cir_0$ if and only if $\rho_2(M)$ maps $\{\infty, 0, 1\}$ to itself, i.e. $\rho_2(M)\in\SL(2, \FF_2)$. This is equivalent to $M\in\PSL(2, \ZZ[2\om])$.
\end{proof}

In particular, the 10 sets of $\FF_2[\om]-$circles correspond exactly to the reductions of the 10 cosets of $\PSL(2, \ZZ[2\om])$ in $\PSL(2, \ZZ[\om])$. Representatives $g_i$ for these cosets are enumerated in Table~\ref{table:cosets}.

\begin{table}[hbt]
\centering
    \begin{tabular}{lllll}
     $g_0 = \lm{1}{0}{0}{1}$, & $g_1 = \lm{1}{\om}{0}{1}$, & $g_2 = \lm{\om}{0}{\om}{1-\om}$, & $g_3 = \lm{1-\om}{1}{0}{\om}$, & $g_4 = \lm{1}{0}{\om}{1}$,\\ \addlinespace[6pt]
     $g_5 = \lm{1}{\om}{1}{1+\om}$, & $g_6 = \lm{1-\om}{0}{1}{\om}$, & $g_7 = \lm{\om}{0}{0}{1-\om}$, & $g_8 = \lm{1-\om}{0}{0}{\om}$, & $g_9 = \lm{\om}{1}{0}{1-\om}$.
    \end{tabular}
    \caption{Matrices in $\PSL(2, \ZZ[\om])$ which reduce to $\SL(2, \FF_2[\om])$ to produce the 10 distinct circles in $\widehat{\FF_2[\om]}$. These are also representatives for the left cosets of $\PSL(2, \ZZ[2\om])$ in $\PSL(2, \ZZ[\om])$.}\label{table:cosets}
\end{table}

\begin{definition}
    Given an Eisenstein circle $\cir$, we say that it is in \emph{coset $i$} (for $0\leq i\leq 9$) if it is the image of $\widehat{\RR}$ by some matrix in the coset $g_i\PSL(2, \ZZ[2\om])$.
\end{definition}

For each of the cosets $g_i\PSL(2, \ZZ[2\om])$, the reductions of the reduced coordinates $(s, t, x, z)$ modulo 2 are well defined and distinct between distinct cosets: they parametrize the 10 solutions to $x^2-xz+z^2-1\equiv 3st\pmod{2}$ (Equation~\eqref{eqn:sheet} modulo 2). This data is represented in Table~\ref{table:F4circles}, and gives the following result.

\begin{proposition}
    Let $(s, t, x, z)\in\ZZ^4$ be the reduced coordinates of an Eisenstein circle. Then $(s, t, x, z)\pmod{2}$ determines which coset the circle is in, with reference to Table~\ref{table:F4circles}.
\end{proposition}

\begin{table}[hbt]
\centering
    \begin{tabular}{cccccc}
        \upshape{Coset} & $0$ & $1$ & $2$ & $3$ & $4$ \\
        \hline
        \upshape{Points} 
        & $\infty, 0, 1$ 
        & $\infty, \omega, 1+\omega$
        & $1, 0, \omega$
        & $\infty, 1+\omega, 1$
        & $1+\omega, 0, \omega$\\
        \hline
        \upshape{Red. curv. parity} & even & even & odd & even & odd \\
        \hline
                \upshape{Red. coord. $\pmod{2}$} & $0001$ & $0101$ & $1010$ & $0111$ & $1001$ \\
        \\
        \upshape{Coset} & $5$ & $6$ & $7$ & $8$ & $9$ \\
        \hline
        \upshape{Points} 
        & $1, 1+\omega, \omega$
        & $1+ \omega, 0, 1$
        & $\infty, 0, 1+\omega$
        & $\infty, 0, \omega$
        & $\infty, \omega, 1$\\
                \hline
        \upshape{Red. curv. parity} & odd & odd & even & even & even \\
        \hline
                \upshape{Red. coord. $\pmod{2}$} & $1100$ & $1011$ & $0010$ & $0011$ & $0110$ \\
    \end{tabular}
\caption{The $\widehat{\FF_2[\om]}-$circles, where coset $i$ corresponds to $g_i\PSL(2, \ZZ[2\om])$, and the reduced coordinates $(s, t, x, z)\pmod{2}$ are concatenated to a length 4 binary word.}\label{table:F4circles}
\end{table}

Before moving on to tangencies, we explain the appearance of the Eisenpint group,
\[\GammaE=\{M\in\PSL(2, \ZZ[\om]):M\equiv\sm{1}{0}{\ast}{1}\!\!\!\pmod{2}\}.\]
Despite $\GammaE$ not being a subgroup or supergroup of $\PSL(2, \ZZ[2\om])$, it produces circles in cosets 0 and 4.

\begin{proposition}\label{prop:GammaEcoset04}
    An Eisenstein circle $\cir$ lies in coset 0 or 4 if and only if $\cir=M(\widehat{\RR})$ for some $M\in\GammaE$.
\end{proposition}
\begin{proof}
    Assume $\cir=M(\widehat{\RR})$ with $M=\sm{a+b\om}{c+d\om}{e+f\om}{g+h\om}\in\GammaE$, whence $b\equiv c\equiv d\equiv h\equiv 0\pmod{2}$. From the formulae of Proposition~\ref{prop:eisensteincircleformulae},
    \[t=ad-bc\equiv 0\pmod{2},\quad x=ah-bg+de-cf\equiv 0\pmod{2}.\]
    Checking Table~\ref{table:F4circles}, this implies that the circle is in coset 0 or 4, as required.

    Now, assume $\cir$ is in coset 0 or 4. It follows that $\cir=M(\widehat{\RR})$, where $M=g_iM_1$ with $M_1\in\PSL(2, \ZZ[2\om])$ and $i=0, 4$. We will take advantage of the fact that $\PSL(2, \ZZ)$ preserves $\widehat{\RR}$ to modify the matrix $g_iM_1$ to make it land inside $\GammaE$, while preserving $\cir$.
    
    Since $\PSL(2, \ZZ[2\om])$ reduces modulo 2 to $\SL(2, \FF_2)$, we can write $M_1=M_2M_3$ with $M_3\in\PSL(2, \ZZ)$ and $M_2\equiv \sm{1}{0}{0}{1}\pmod{2}$. Since $M_3(\widehat{\RR})=\widehat{\RR}$, this implies that
    \[\cir=g_iM_2(\widehat{\RR}),\]
    where
    \[g_iM_2\equiv\lm{1}{0}{\ast}{1}\lm{1}{0}{0}{1}=\lm{1}{0}{\ast}{1}\pmod{2},\]
    so $g_iM_2\in \GammaE$.
\end{proof}

\subsection{Tangencies of $\widehat{\FF_2[\om]}-$circles}

We claim that the tangency structure of the $\widehat{\FF_2[\om]}-$circles will lift to the global circles.

\begin{definition}
    Two $\widehat{\FF_2[\om]}-$circles are \emph{tangent} if they intersect in exactly one point. 
\end{definition}

We can consider the tangency graph of the ten circles, depicted in Figure~\ref{fig:cosettangencies}. The M\"obius transformations of $\SL(2,\FF_2[\om])$ preserve circles and tangencies, so they represent automorphisms of this graph.

\begin{figure}[htb]
    \includegraphics[scale=0.7]{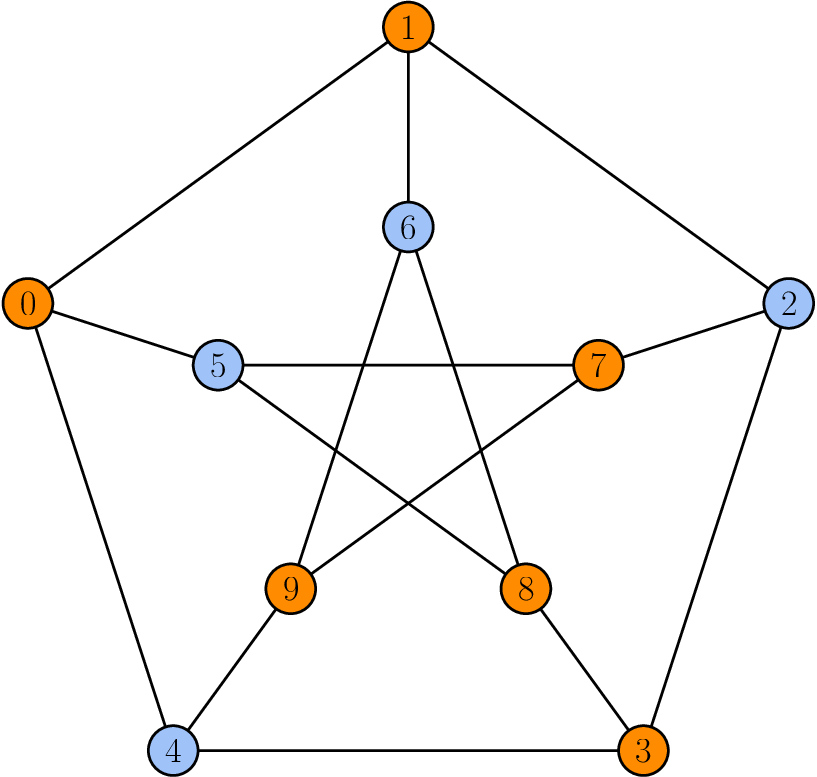}
    \caption{The circles of $\widehat{\FF_2[\om]}$, labeled by the coset they belong to, with edges representing tangencies.  This graph is known as the Petersen graph. The orange nodes (0, 1, 3, 7, 8, 9) represent even reduced curvatures, and the blue nodes (2, 4, 5, 6) represent odd reduced curvatures.}\label{fig:cosettangencies}
\end{figure}

We claim that this graph can detect Eisenstein circles that intersect in two places.

\begin{lemma}\label{lem:twointersect}
    Let $A$ and $B$ be two Eisenstein circles. If $A,B$ intersect in exactly two points, then $\rho_2(A),\rho_2(B)$ intersect in exactly two points. If $A,B$ are tangent, then $\rho_2(A),\rho_2(B)$ intersect in either one or three points.
\end{lemma}

\begin{proof}
    By applying a M\"obius transformation, we can assume that $A=\widehat{\RR}$. From \cite[Proposition 4.1]{StangeVisualizingArith}, the intersection point(s) must lie in $\widehat{\RR}\cap \widehat{\QQ(\om)} = \widehat{\QQ}$. Let $B=M(\widehat{\RR})$, and by pre and post composing $B$ with appropriate matrices in $\PSL(2, \ZZ)$, we can assume that one intersection point is $\infty$ and $M(\infty)=\infty$. Therefore we have
    \[M = \lm{u}{b}{0}{1/u} = \lm{1}{bu}{0}{1}\lm{u}{0}{0}{1/u},\]
    for some $u,d\in\ZZ[\om]$, where $u$ is necessarily a unit. This is a composition of a rotation (multiplication by $u^2$) and a translation (by $bu$).
    
    Since we are working up to $\pm 1$, we can take $u=1,\om,1-\om$. If $u=1$, then $M(\widehat{\RR})$ is either equal to $\widehat{\RR}$ or parallel to it, which corresponds to a tangency. The images of $0,1,\infty$ are then
    \[M(0)=b,\quad M(1)=1+b,\quad M(\infty)=\infty.\]
    Since $b\equiv 0, 1, \om, 1+\om\pmod{2}$, we compute that
    \[(M(0), M(1))\equiv (0, 1), (1, 0), (\om, 1+\om), (1+\om, \om) \pmod{2}.\]
    In all cases, our reduced circles intersect in exactly one or three points.

    Otherwise, $u=\om,1-\om$, so we rotate by either $\frac{2\pi}{3}$ or $\frac{4\pi}{3}$ respectively. In both cases we end up with an image which intersects $\widehat{\RR}$ exactly twice.

    If $u=1-\om$, we can instead consider $M^{-1}(\widehat{\RR})$ intersecting with $\widehat{\RR}$, which just reduces us to the case of $u=\om$. The images of $0, 1, \infty$ are then
    \[M(0) = b\om,\quad M(1) = \om^2+b\om = \om + 1 + b\om,\quad M(\infty) = \infty.\]
    Since $b\equiv 0, 1, \om, 1+\om\pmod{2}$, we compute that
    \[(M(0), M(1))\equiv (0, 1+\om), (\om, 1), (1+\om, 0), (1, \om) \pmod{2}.\]
    In all cases, we see that $\rho_2(A)$ and $\rho_2(B)$ intersect in exactly two points, as claimed.
\end{proof}

\begin{corollary}\label{cor:tangentonlypicture}
    Let $X$ be a subset of the Eisenstein circles containing only circles from cosets $i$ and $j$, where $i-j$ is an edge in Figure~\ref{fig:cosettangencies}. Then, circles in $X$ can only intersect tangentially.

    Furthermore, if a circle in coset $i$ and a circle in coset $j$ are tangent, either $i=j$ or $i-j$ is an edge in Figure~\ref{fig:cosettangencies}.
\end{corollary}
\begin{proof}
    For the first part, assume otherwise. Then there exists circles $A, B\in X$ that intersect exactly twice. By Lemma~\ref{lem:twointersect}, $\rho_2(A)$ and $\rho_2(B)$ intersect exactly twice. However, if they are in the same coset, they intersect three times, and if they are in opposite cosets, they intersect exactly once, as $i-j$ was an edge.

    The second part follows similarly from Lemma~\ref{lem:twointersect}.
\end{proof}

Proposition~\ref{prop:GammaEcoset04} and Corollary~\ref{cor:tangentonlypicture} demonstrate that the Eisenpint Schmidt arrangement only contains tangencies, as it consists of cosets 0 and 4, which are connected by an edge in Figure~\ref{fig:cosettangencies}.

\begin{corollary}
    Circles in $\pint$ can only intersect tangentially.
\end{corollary}

\subsection{The other tangency edges}

The Eisenpint Schmidt arrangement combines cosets 0 and 4, but we could have equally combined any other pair of tangent cosets. This choice is not canonical: we chose cosets 0 and 4 as they correspond to a very natural choice of subgroup $\GammaE$ of $\PSL(2, \ZZ[\om])$. It turns out that all of these lifts are related, as the 15 edges correspond to the 15 cosets of $\GammaE$. Furthermore, these can be seen as combinations of simple translations, dilations, and rotations of $\pint$, demonstrating that any choice is equivalent up to symmetry.

Consider the three elements
\[T_1 = \lm{1}{1}{0}{1},\quad T_{\om} = \lm{1}{\om}{0}{1},\quad R=\lm{\om}{0}{0}{1-\om},\]
which respectively act on $\widehat{\CC}$ as translation by $1$, translation by $\om$, and rotation counterclockwise about the origin by $2\pi/3$. Each move permutes the 10 cosets of circles, and the exact permutation can be computed using the explicit cosets in Table~\ref{table:cosets}. The translations are depicted in Figure~\ref{fig:cosettranslations}, and the rotations in Figure~\ref{fig:cosetrotations}.

\begin{figure}[htb]
    \includegraphics[scale=0.85]{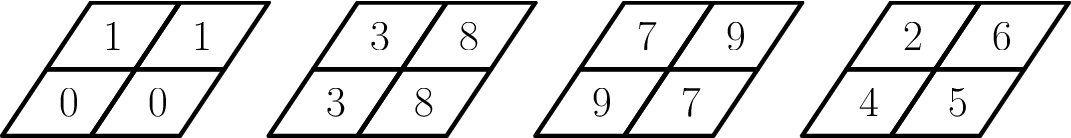}
    \caption{How the cosets change with translations. To read the diagram, locate the coset you wish to follow under translation, e.g. lower left 4 in the last diagram. Then translation by $\omega$ is a NNE move, so coset $4$ becomes coset $2$.  The diagrams wrap like tori, so that the same translation takes coset 6 to coset 5. The first three parallelograms consist of even curvatures, and the last one is the odd curvatures.}\label{fig:cosettranslations}
\end{figure}

\begin{figure}[htb]
    \includegraphics[scale=0.85]{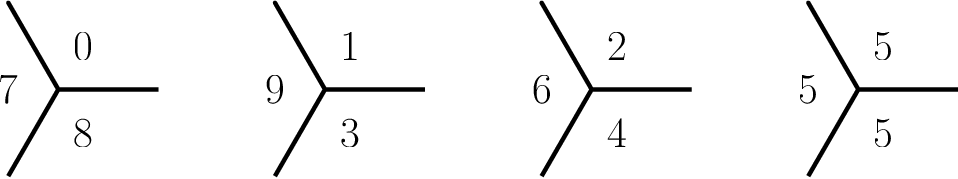}
    \caption{How the cosets change with rotations about the origin. To read the diagram, locate the coset you wish to follow under rotation, e.g. lower-right 4 in the third diagram. Then rotation by $2\pi/3$ is counterclockwise one notch, taking 4 to 2. The first two figures consist of even curvatures, and the last two of odd curvatures.}\label{fig:cosetrotations}
\end{figure}

The fact that coset 5 is preserved by rotations about the origin does not make it ``distinguished''. Indeed, rotating about $1$, $\om$, or $1+\om$ will preserve cosets $4$, $6$, and $2$ respectively.

Using these moves, we can make 12 cosets of $\GammaE$, which cover exactly the edges in Figure~\ref{fig:cosettangencies} consisting of nodes of opposite colours. The other 3 cosets require an inversion.

\begin{corollary}\label{cor:GammaEcosetsgiveedges}
    Twelve of the fifteen cosets of $\GammaE$ have representatives
    \[\text{Id}, T_1, T_{\om}, T_1T_{\om}, R, RT_1, RT_{\om}, RT_1T_{\om}, R^2, R^2T_1, R^2T_{\om}, R^2T_1T_{\om},\]
    which represent translation by $0$, $1$, $\om$, or $1+\om$, followed by a rotation of $0$, $2\pi/3$, or $4\pi/3$.

    The images of $\widehat{\RR}$ under these cosets produce the 12 edges of Figure~\ref{fig:cosettangencies} between nodes of opposite colour.

    Let $U=\sm{0}{-1}{1}{0}$. The remaining three cosets of $\GammaE$ have representatives $U$, $RU$, and $R^2U$, and they correspond to the three edges between pairs of orange nodes in Figure~\ref{fig:cosettangencies} (i.e. packings with only even curvatures). 
\end{corollary}
\begin{proof}
    By analyzing Figures~\ref{fig:cosettranslations} and~\ref{fig:cosetrotations}, the first claim follows. For the second, by Proposition~\ref{prop:stxzbasictransformations}, we see that $U$ sends the coset represented by $(s, t, x, z)\pmod{2}$ to $(t, s, -x, z-x)\equiv (t, s, x, z+x)\pmod{2}$. In particular, this preserves coset 0, and sends coset 4 to coset 1, giving the $0-1$ edge. Rotations by $R$ and $R^2$ produce the $7-9$ and $8-3$ edges, completing the claim.
\end{proof}

Intriguingly, the inversion $U$ can be replaced by a dilation!

\begin{proposition}\label{prop:dilate}
    The set of circles in coset 0 or 1 are obtained from the Eisenpint circles $\pint$ by a dilation with factor $\frac{1}{2}$ about the origin.
\end{proposition}
\begin{proof}
    By Proposition~\ref{prop:eisensteincircleconverse} and the subsequent discussion, the circles in $\PSL(2, \ZZ[\om])(\widehat{\RR})$  are precisely given by the reduced coordinates of integers $(s, t, x, z)$ satisfying $x^2-xz+z^2-1=3st$ and $x+z\equiv 2\pmod{3}$. Those circles that lie in a given coset correspond to the solutions $(s, t, x, z)\pmod{2}$ with reference to Table~\ref{table:F4circles}. In particular, for a circle $\cir$ corresponding to reduced coordinates $(s, t, x, z)$, we have
    \[(s, t, x, z)\pmod{2}\equiv\begin{cases}
        (0, 0, 0, 1) & \text{if and only if $\cir$ is in coset 0;}\\
        (0, 1, 0, 1) & \text{if and only if $\cir$ is in coset 1;}\\
        (1, 0, 0, 1) & \text{if and only if $\cir$ is in coset 4.}
    \end{cases} \]
    A dilation by a factor of $\frac{1}{2}$ about the origin maps
    \[(s, t, x, z) \rightarrow (2s, t/2, x, z).\]
    Therefore, if $\cir$ is in coset 0 or 4, then $t\equiv 0\pmod{2}$, hence the output circle will again be an integral solution with $x+z\equiv 2\pmod{3}$, and therefore correspond to a circle in $\PSL(2, \ZZ[\om])(\widehat{\RR})$. Its reduced coordinates modulo 2 is of the form $(0, \ast, 0, 1)$, i.e. in coset 0 or 1.

    Similarly, starting with a circle inside coset 0 or 1 and applying a dilation by $2$ will end up with a circle in coset 0 or 4. These maps are inverse to each other, proving that they are bijections.
\end{proof}

Given two tangent circles in the full Schmidt arrangement, by Corollary~\ref{cor:tangentonlypicture}, there exist cosets $i,j$ where the set of circles where $i-j$ is an edge in Figure~\ref{fig:cosettangencies}, and this pair of circles occur in the union of circles in cosets i and j. By Corollary~\ref{cor:GammaEcosetsgiveedges} and Proposition~\ref{prop:dilate}, this packing is the image of $\pint$ by a combination of a dilation (about origin, factor $1,\frac{1}{2}$), translation (by $0, 1, \om, 1+\om$), and rotation (about origin, angle of $0,\frac{2\pi}{3}, \frac{4\pi}{3}$). In particular, Theorem~\ref{thm:pintschmditiscanonical} follows.

\section{Symmetries of the Eisenpint Schmidt arrangement}\label{sec:schmidtsymmetries}

In order to demonstrate that the Eisenpint Schmidt arrangement produces scaled primitive Eisenstein circle packings, we must introduce equivalence classes of circles in the arrangement. Recall that $\widepint$ denotes the set of Eisenpint circles with both orientations. Equivalently, by Proposition~\ref{prop:eisensteincircleconverse}, we can consider integral coordinates $(s, t, x, z)$ where $-3st+x^2-xz+z^2=1$.

The coset of this circle is identified with $(s, t, x, z)\pmod{2}$, as in Table~\ref{table:F4circles}. Since negation acts trivially modulo 2, the coset of a circle and its opposite orientation are identical.

\subsection{Fundamental domain}

Translation by $2\ZZ+2\om\ZZ$ descends to the identity map on $\widehat{\FF_2[\om]}-$circles, hence it preserves cosets.

\begin{proposition}\label{prop:enschmidttranslate}
    The set $\widepint$ is invariant under translation by $2\ZZ+2\om\ZZ$.
\end{proposition}

Next, we consider a rotation.

\begin{proposition}\label{prop:enschmidtrotatebypi}
    The set $\widepint$ is invariant under a rotation by $\pi$ about the origin.
\end{proposition}
\begin{proof}
    This corresponds to the map $(s, t, x, z)\rightarrow (s, t, -x, -z)$, which clearly preserves the coordinates modulo 2.
\end{proof}

Note that Proposition~\ref{prop:enschmidtrotatebypi} applies to any coset, not just cosets 0 and 4. This will not be true of our final symmetry.

\begin{proposition}\label{prop:enschmidtreflect}
    The set $\widepint$ is invariant under reflection across the imaginary axis.
\end{proposition}
\begin{proof}
    This corresponds to the map $(s, t, x, z)\rightarrow (s, t, -x, z-x)$. The coordinates are preserved modulo 2 if and only if $x\equiv 0\pmod{2}$, which occurs in cosets 0 and 4 (as well as 1 and 5, but not the rest).
\end{proof}

Putting together Propositions~\ref{prop:enschmidttranslate}-\ref{prop:enschmidtreflect} yields an equivalence class of circles, and a fundamental domain for their equivalence classes, depicted in Figure~\ref{fig:pintschmidtfdom}.

\begin{figure}[htb]
    \includegraphics[scale=0.85]{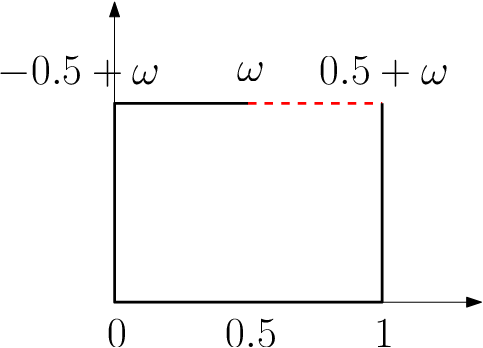}
    \caption{Fundamental domain for the (centres of) circles in $\widepint$.}\label{fig:pintschmidtfdom}
\end{figure}

\begin{definition}
    Call two circles in $\widepint$ \emph{equivalent} if they are related by a sequence of translations by $2\ZZ+2\om\ZZ$, rotations about the origin by $\pi$, and reflections across the imaginary axis.
\end{definition}

\begin{proposition}\label{prop:schmidtequivclasses}
    There is one equivalence class of curvature zero circles, which has representative $\widehat{\RR}$. Each non-zero curvature circle has a unique representative whose centre lies in the rectangle with corners $(0, 1, 0.5+\om, -0.5+\om)$, i.e. of width 1 and height $\sqrt{3}/2$. We include the boundary of this rectangle, except for the line segment $\{a+\om:0<a\leq 0.5\}$.
\end{proposition}
\begin{proof}
    The claim for curvature zero is immediate, so consider the non-zero curvature case. Let $C=a+bi$ denote the centre of our circle. By Proposition~\ref{prop:enschmidttranslate}, $C$ can be translated to lie in the parallelogram with vertices $(0, 2, 2+2\om, 2\om)$. If $b\geq \sqrt{3}/2$, we can use Proposition~\ref{prop:enschmidtrotatebypi} to rotate by $\pi$, and then apply a translation by $2+2\om$ to land inside the parallelogram with vertices $(0, 2, 2+\om, \om)$. From here,
    \begin{itemize}
        \item If $0\leq a\leq 1$, do nothing, as we are already in the fundamental domain.
        \item If $1<a<2$, then use Proposition~\ref{prop:enschmidtreflect} to reflect across the imaginary axis, then apply a translation by 2 to land in the fundamental domain.
        \item If $2\leq a$, apply a translation by $-2$ to land in the fundamental domain.
    \end{itemize}
    This process is depicted in Figure~\ref{fig:fundamentaldomainreduction}. In particular, this proves we can always reduce to this domain including the boundary. For the final reduction, if the centre is $a+ \om$ with $0<a\leq 0.5$, then apply a rotation by $\pi$ and a translation by $2\om$ to obtain centre $-a+\om$. This completes the existence.

    For uniqueness, note that a translation by $2m+2n\om$ (with $m, n\in\ZZ$) corresponds to $(a, b)\rightarrow (a+2m+n, b+\sqrt{3}n)$. Rotation by $\pi$ about the origin is $(a, b)\rightarrow (-a, -b)$, and reflection in the imaginary axis is $(a, b)\rightarrow (-a, b)$. By chaining these operations, it is clear that the general equivalence takes the form
    \[(a, b)\rightarrow(\pm a+u, \pm b+v\sqrt{3}),\]
    where $u, v\in\ZZ$ and $u\equiv v\pmod{2}$. If $0\leq a\leq 1$ and $0\leq b\leq\sqrt{3}/2$, first assume that $b\neq\sqrt{3}/2$. Then, it is clear that $\pm b+\sqrt{3}n$ can only lie within $[0, \sqrt{3}/2]$ if $\pm = +$ and $n=0$. For $a$, we have $\pm a+m$ where $m$ is necessarily even, and this will only lie inside $[0, 1]$ when it equals $a$. (Note that if $a=1$, you need not have $\pm=+$ and $m=0$, as $-1+2=1$, but this will still equal $a$.)

    The final case is $b=\sqrt{3}/2$. In this case, the only other way for $\pm b+\sqrt{3}n$ to lie within $[0, \sqrt{3}/2]$ is with $-b+\sqrt{3}=b$. This is paired with $\pm a+m$, where $m$ is odd. If $\pm=+$ we get no new solutions, but we have one final equivalence of $1-a$, which pairs up $(a, \sqrt{3}/2)$ with $(1-a, \sqrt{3}/2)$, and has been accounted for.
\end{proof}

\begin{figure}[htb]
	\centering
	\begin{subfigure}{.5\textwidth}
		\centering
		\includegraphics[width=.9\linewidth]{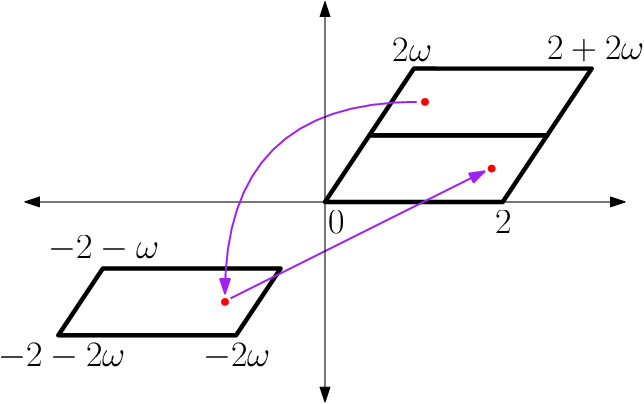}
		\caption{Halving the first parallelogram: rotation by $\pi$\\ about the origin, then translation by $2+2\om$.}
	\end{subfigure}%
	\begin{subfigure}{.5\textwidth}
		\centering
		\includegraphics[width=.9\linewidth]{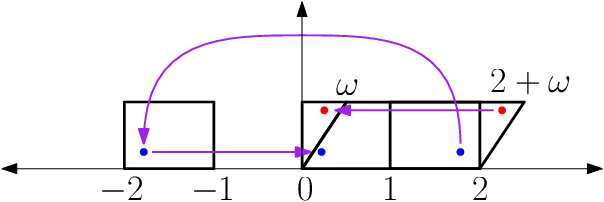}
		\caption{Final reduction step: either nothing, reflect across the imaginary axis then translate by 2, or translate by $-2$..}
	\end{subfigure}
    \caption{Reduction of the circles to their fundamental domain.}\label{fig:fundamentaldomainreduction}
\end{figure}

\subsection{Bijection with quadratic forms}

In Section~\ref{sec:quadraticforms}, we bijected circles in Eisenstein circle packings with equivalence classes of first-odd binary quadratic forms. By bijecting these quadratic forms with equivalence classes of circles in $\widepint$, we will demonstrate that the Eisenpint tangency packings are exactly the (scaled) primitive Eisenstein circle packings. Furthermore, this places the Eisenstein circle packing into the plane in a canonical way, up to equivalence.

\begin{definition}
    Let $\Nm:\QQ(\sqrt{-3})\rightarrow\QQ$ denote the norm map, and let $M=\sm{\alpha}{\gamma}{\beta}{\delta}\in\GammaE$. Associate to $M$ the quadratic form
    \[f_M(x, y):=\Nm(\delta x-\beta y).\]
\end{definition}

\begin{proposition}\label{prop:schmidtcircletoquadratic}
    Assume that $M(\widehat{\RR})$ has reduced curvature $s$. Then the quadratic form $f_M$ is a primitive positive semidefinite first-odd binary quadratic form of discriminant $-3s^2$. Furthermore, this associates equivalence classes of circles in the Eisenpint Schmidt arrangement with $\Gamma_0^{\PGL}(2)-$equivalence classes of primitive positive semidefinite first-odd binary quadratic forms.
\end{proposition}
\begin{proof}
    The quadratic form $f_M$ has discriminant $(\delta\overline{\beta}-\overline{\delta}\beta)^2$, which by Propositions~\ref{prop:mobiuscurv} and~\ref{prop:eisensteincircleformulae} is equal to $-3s^2$. As $\delta,\beta\in\ZZ[\om]$, the form $f_M$ is necessarily integral and positive semidefinite. Primitivity follows from the fact that $\delta,\beta$ must be coprime in $\ZZ[\om]$, and first-oddness is equivalent to $f_M(1, 0)=\Nm(\delta)$ being odd, which is true as $\delta\equiv 1\pmod{2\ZZ[\om]}$.

    Next, consider another matrix $M'\in\GammaE$ which represents the same circle. The stabilizer of $\widehat{\RR}$ in $\GammaE$ is
    \[\GammaE\cap\PSL(2, \RR) = \left\{\lm{a}{b}{c}{d}\in\PSL(2, \ZZ):a-1\equiv b\equiv d-1\equiv 0\pmod{2}\right\}.\]
    Thus, $M'=MN$ for some integral matrix $N=\genmtx$ where $a$ is odd, $b$ is even, and $d$ is odd. The bottom row of $M'$ is $(a\beta+c\delta, b\beta+d\delta)$, hence
    \begin{align*}
        f_{M'}(x, y) & = \Nm((b\beta+d\delta)x-(a\beta+c\delta)y) = \Nm((dx-cy)\delta-(-bx+ay)\beta) \\
        & = f_M(dx-cy, -bx+ay) = \lm{d}{-c}{-b}{a}\circ f_M(x, y).
    \end{align*}
    As $\sm{d}{-c}{-b}{a}\in\Gamma_0^{\PGL}(2)$, this demonstrates that we land inside the same $\Gamma_0^{\PGL}(2)-$equivalence class of quadratic forms, as desired.

    At the moment, we have associated each circle in the Eisenpint Schmidt arrangement with a well-defined $\Gamma_0^{\PGL}(2)-$equivalence class of quadratic forms. The next step is to demonstrate that equivalent circles give the same equivalence classes.

    To this end, begin with an $M\in\GammaE$, and first consider a translation. This replaces $M$ with $M'=\sm{1}{2x}{0}{1}M$, where $x\in\ZZ[\om]$. Since $(\beta, \delta)$ remains unchanged, $f_{M'}=f_M$.

    Next, consider a rotation by $\pi$ about the origin. This is equivalent to conjugating $M$ by $\sm{1}{0}{0}{-1}$ (conjugation is required to retain determinant 1), and replaces $(\beta, \delta)$ with $(-\beta, \delta)$, i.e. $f_M\rightarrow \sm{1}{0}{0}{-1}\circ f_M$. As $\sm{1}{0}{0}{-1}\in \Gamma_0^{\PGL}(2)$, we are again done.

    Finally, we consider the reflection across the imaginary axis. Replace $M=\sm{\alpha}{\gamma}{\beta}{\delta}$ with $M'=\sm{\overline{\alpha}}{-\overline{\gamma}}{-\overline{\beta}}{\overline{\delta}}$, where $M'\in\GammaE$. By Proposition~\ref{prop:mobiuscurv}, it preserves the curvature and co-curvature, and sends the curvature centre $w$ to $-\overline{w}$, i.e. represents the reflection over the imaginary axis. It is associated to the quadratic form
    \[f_{M'}(x, y)=\Nm(\overline{\delta}x+\overline{\beta}y)=\Nm(\delta x+\beta y)=\sm{1}{0}{0}{-1}\circ f_M(x, y),\]
    once again preserving the equivalence class.
\end{proof}

\begin{remark}
    Note that we are identifying $(s, t, x, z)$ with $-(s, t, x, z)$, so both orientations of a circle are sent to the same quadratic form. If we wanted to separate these, we would also need to retain the information of the sign of $s$. This could be accomplished by allowing both positive and negative definite forms.
\end{remark}

In order to demonstrate that this association is a bijection, we need to go backwards. Let
\[Q(x, y)=\Nm(x+y\om)=x^2+xy+y^2\]
denote the unique primitive positive definite reduced quadratic form of discriminant $-3$.

%f(x, y)=x^2-xy+2y^2 has automorph of [1, -1;1/2, 1/2]. So, in Q, the automorphism group is non-trivial. Bleh.
%In fact for x^2+xy+y^2, even here we have [5/7, -3/7;3/7, 8/7].

\begin{proposition}\label{prop:qfblift}
    Let $f(x, y)$ be a primitive integral quadratic form of discriminant $-3s^2$, for some nonzero $s\in\ZZ$. Then there exists a primitive matrix $N=\genmtx\in\Mat(2, \ZZ)$ of determinant $s$ such that $f(x, y)=N\circ Q(x, y)=Q(ax+by, cx+dy)$. This matrix is uniquely identified up to left multiplication by one an integral automorph of $Q(x, y)$.
\end{proposition}
\begin{proof}
    See Chapter 7 of \cite{Coxprimesx2ny2} for the background on quadratic forms and orders in quadratic fields required.

    There exists bijections between primitive positive definite integral quadratic forms of discriminant $-3s^2$ and invertible fractional ideals of $\mathcal{O}_{-3s^2}$, the unique quadratic order of discriminant $-3s^2$ living inside $\ZZ[\om]$. If the quadratic form is $q=[A, B, C]$, it is sent to the quadratic ideal
    \[\mathfrak{a}_q=A\ZZ+\frac{-B+s\sqrt{-3}}{2}\ZZ = A\ZZ+\left(\frac{-B-s}{2}+s\om\right)\ZZ.\]
    For the other direction, let $\mathfrak{a}=\alpha_1\ZZ+\alpha_2\ZZ$ be an invertible fractional ideal of $\mathcal{O}_{-3s^2}$, where the basis is oriented to satisfy $(\alpha_2\overline{\alpha_1}-\alpha_1\overline{\alpha_2})/\sqrt{-3}>0$. This is mapped to
    \[q_{\mathfrak{a}}=\frac{\Nm(\alpha_1x-\alpha_2y)}{\Nm(\mathfrak{a})}.\]

    Now, start with $f(x, y)$ as prescribed by the proposition, which is associated to the invertible fractional $\mathcal{O}_{-3s^2}-$ideal $\mathfrak{a}_f$. Let $\mathfrak{a}'=\mathfrak{a}_f\ZZ[\om]=\alpha_1\ZZ+\alpha_2\ZZ$ be the lift of this fractional ideal to a fractional $\mathcal{O}_{-3}-$ideal of the same norm, where $(\alpha_1, \alpha_2)$ is oriented as above. Since $\mathfrak{a}_f\subseteq\mathfrak{a'}$, it follows that there exists integers $a,b,c,d$ so that an oriented basis of $\mathfrak{a}_f$ is $(a\alpha_1+b\alpha_2, c\alpha_1+d\alpha_2)$. Therefore,
    \begin{align*}
        f(x, y) & = \frac{\Nm((a\alpha_1+b\alpha_2)x-(c\alpha_1+d\alpha_2)y)}{\Nm(\mathfrak{a}_f)}\\
        & = \frac{\Nm((ax-cy)\alpha_1-(-bx+dy)\alpha_2)}{\Nm(\mathfrak{a}')}\\
        & = q_{\mathfrak{a}'}(ax-cy, bx-dy) = \lm{a}{-c}{-b}{d}\circ q_{\mathfrak{a}'}(x, y).
    \end{align*}

    Since $q_{\mathfrak{a}'}(x, y)$ is an integral quadratic form of discriminant $-3$, it is $\SL(2, \ZZ)-$equivalent to $Q(x, y)$, say $N_1\circ Q=q_{\mathfrak{a}'}$. Let $N=N_1\sm{a}{-c}{-b}{d}\in\Mat(2, \ZZ)$, and then $f=N\circ Q$. Since $f$ has discriminant $-3s^2$ and is primitive, it follows that $\det(N)=\pm s$ and it is primitive. The sign being positive follows from both bases being oriented correctly.

    The final part is to demonstrate the uniqueness of this lift, up to an automorph. Let $N_1$ and $N_2$ be two possible matrices $N$, and we have $(N_2N_1^{-1})\circ Q = Q$. Let $N_2N_1^{-1}=L\in\SL(2, \QQ)$, which is thus an automorphism of $Q$, albeit not necessarily with integer entries. Scale this matrix to be primitive and integral, i.e. write
    \[L=\frac{1}{k}\lm{e}{f}{g}{h},\quad e,f,g,h\in\ZZ,\quad k\in\ZZ^+,\quad \gcd(e,f,g,h)=1.\]
    If $k=1$, then $L\in\SL(2, \ZZ)$ is integral, and the claim is proven. Thus, assume $k>1$, and we seek a contradiction. The final contradiction will be derived from the assumption that $f$ is primitive. Otherwise, this claim would be false!
    
    Let $p$ be any prime divisor of $k$. We claim that $e$ is coprime to $p$. Expanding $L\circ Q=Q$ and equating the $x^2$ and $y^2$ coefficients gives the equations
    \[e^2+eg+g^2\quad =\quad f^2+fh+h^2\quad =\quad k^2.\]
    Since $L^{-1}=\frac{1}{k}\sm{h}{-f}{-g}{e}$ also preserves $Q$, we analogously derive
    \[g^2-gh+h^2\quad =\quad e^2-ef+f^2\quad =\quad k^2.\]
    If $p\mid e$, then $p\mid k^2-e^2-eg=g^2$, whence $p\mid g$. Thus, $p\mid k^2-g^2+gh=h^2$, so $p\mid h$. Finally, $p\mid k^2-h^2-fh=f^2$, so $p\mid f$, contradicting $p\mid \gcd(e, f, g, h)=1$. Thus, $p\nmid e$.

    Write $N_1=\sm{a}{b}{c}{d}$, whence
    \[f(x, y)=N_1\circ Q = (a^2+ac+c^2)x^2+(2ab+ad+bc+2cd)xy+(b^2+bd+d^2)y^2.\]
    Observe that 
    \[N_2=LN_1=\frac{1}{k}\lm{ea+fc}{eb+fd}{ga+hc}{gb+hd}\in\Mat(2, \ZZ).\]
    As $e$ is coprime to $p$, we derive $a\equiv -fc/e\pmod{p}$, whence
    \[a^2+ac+c^2\equiv \frac{c^2}{e^2}(f^2-ef+e^2)\equiv 0\pmod{p}.\]
    Similarly, $b\equiv -fd/e$ gives
    \[b^2+bd+d^2\equiv \frac{d^2}{e^2}(f^2-ef+e^2)\equiv 0\pmod{p}.\]
    Finally,
    \[2ab+ad+bc+2cd\equiv \frac{f^2}{e^2}2cd-\frac{f}{e}cd-\frac{f}{e}cd+2cd\equiv \frac{2cd}{e^2}(f^2-ef+e^2)\equiv 0\pmod{p}.\]
    Thus, all coefficients of $f(x, y)$ are multiples of $p$, contradicting its primitivity, and completing the proof.
\end{proof}

Proposition~\ref{prop:qfblift} enables us to construct the reverse map. Let $M=\sm{\alpha}{\gamma}{\beta}{\delta}\in\GammaE$, and write $\beta=e+f\om$ and $\delta=g+h\om$, for $e,f,g,h\in\ZZ$. Then,
\begin{align}\label{eqn:fMtoQ}
    f_M(x, y) & =\Nm(\delta x-\beta y) = \Nm((g+h\om)x-(e+f\om)y)\\
    & =\Nm((gx-ey)+(hx-fy)\om)=\lm{g}{-e}{h}{-f}\circ Q(x, y).
\end{align}
With this motivation, start with $f(x, y)$, a first-odd primitive integral quadratic form of discriminant $-3s^2$, for some $s\in\ZZ$. Proposition~\ref{prop:qfblift} gives a primitive integral matrix $N=\genmtx$ of determinant $s$ for which $f(x, y)=\genmtx\circ Q(x, y)$. Motivated by Equation~\eqref{eqn:fMtoQ}, the natural choice would be to construct $\beta=-b-d\om$ and $\delta=a+c\om$. However, we want to construct $M=\sm{\alpha}{\gamma}{\beta}{\delta}$ to live in $\GammaE$, which requires $a$ to be odd and $c$ to be even. Since $Q$ has non-trivial automorphs of $U=\sm{0}{1}{-1}{-1}$ and $U^2=\sm{-1}{-1}{1}{0}$, we have three possibilities for $N$ to choose from:
\[N = \lm{a}{b}{c}{d},\quad UN = \lm{c}{d}{-a-c}{-b-d},\quad U^2N = \lm{-a-c}{-b-d}{a}{b}.\]
The first-oddness assumption on $f$ ensures that $f(1, 0)\equiv a^2+ac+c^2\equiv 1\pmod{2}$, which is equivalent to $a,c$ not being both even. By examining the first columns of $N, UN, U^2N$, we see that in each case ((odd, odd), (odd, even), (even, odd)), exactly one of these three gives the modulo 2 condition required.

Without loss of generality, assume that it was $N$ (as we can replace $N$ by $UN$ or $U^2N$, as required). Take $\beta=-b-d\om$ and $\delta=a+c\om$, and since (by construction) $f(x, y)=\Nm(\delta x-\beta y)$, the primitivity of $f$ implies that $\beta$ and $\delta$ are coprime in $\ZZ[\om]$. Thus, there exists $\alpha,\gamma\in\ZZ[\om]$ for which $M_f:=\sm{\alpha}{\gamma}{\beta}{\delta}$ has determinant 1. As $\delta\equiv 1\pmod{2}$, we can perform a row replacement of the first row to assume that $\gamma\equiv 0\pmod{2}$. It follows that $\alpha\equiv 1\pmod{2}$, and thus $M_f\in\GammaE$.

\begin{definition}\label{def:firstoddtocircleclass}
    Given $f(x, y)$, a first-odd primitive integral quadratic form of discriminant $-3s^2$, associate it to the equivalence class of circles given by $M_f(\widehat{\RR})$, with $M_f$ as above.
\end{definition}

We must check that this definition is well-founded. By Proposition~\ref{prop:qfblift}, there are exactly three possibilities (due to the automorphs) for $N$, and exactly one of them is valid to satisfy the congruence requirements (as seen above). The only ambiguity comes from the final row replacement for $M$. This corresponds to left multiplication by $\sm{1}{2x}{0}{1}$ for some $x\in\ZZ[\om]$, which is a translation by $2\ZZ[\om]$, and thus does not change the equivalence class of $M_f(\widehat{\RR})$.

\begin{proposition}
    The associations in Proposition~\ref{prop:schmidtcircletoquadratic} and Definition~\ref{def:firstoddtocircleclass} are inverse bijections between the equivalence classes of (unoriented) circles in the Eisenpint Schmidt arrangement, and $\Gamma_0^{\PGL}(2)-$equivalence classes of primitive positive semidefinite first-odd binary quadratic forms.
\end{proposition}
\begin{proof}
    Start with $M=\sm{\alpha}{\gamma}{\beta}{\delta}\in\GammaE$, which is associated to $\Nm(\delta x-\beta y)$. When constructing the inverse, the values of $\beta$ and $\delta$ were uniquely identified, and the top row of $M$ was identified up to translation by $2\ZZ[\om]$, which produces equivalent circles. Therefore we must end back in the starting equivalence class of circles.

    The other direction is immediate as well: starting with $f(x, y)$, we write $f(x, y)=\Nm(\delta x-\beta y)$, and the resulting matrix gets associated back to this quadratic form.
\end{proof}

\subsection{Maximal tangency graphs}

We now settle Theorem~\ref{thm:Qrt3packingsareEisenstein}, which claims that the maximal tangency packings in the Eisenpint Schmidt arrangement (up to equivalence) biject with primitive Eisenstein packings. We approach this with a series of lemmas.

\begin{lemma}\label{lem:mincurv}
    The minimal non-zero absolute curvature of a circle in $\widepint$ is $\sqrt{3}$.
\end{lemma}
\begin{proof}
    Consider the circle $M(\widehat{\RR})$, where $M=\sm{a+b\om}{c+d\om}{e+f\om}{g+h\om}\in\PSL(2, \ZZ[\om])$. By Proposition~\ref{prop:eisensteincircleformulae}, the curvature is $(eh-fg)\sqrt{3}$. As $e,f,g,h\in\ZZ$, the smallest non-zero absolute value this can be is $\sqrt{3}$, which is realized.
\end{proof}

The previous and following lemma also apply to the full Eisenstein Schmidt arrangement, though this is not necessary for us.

\begin{lemma}\label{lem:notangenttriangles}
    There do not exist three circles in $\widepint$ which are pairwise externally tangent at three distinct points.
\end{lemma}
\begin{proof}
    Assume otherwise, where $\cir_1,\cir_2,\cir_3$ are pairwise externally tangent. By applying a $\PSL(2, \ZZ[\om])-$ transformation, we can assume that $\cir_1$ is the real axis, with interior being the lower half plane. By further applying a $\PSL(2, \ZZ)-$transformation, we can assume that the tangency point with $\cir_2$ is $\infty$, hence $\cir_2$ is a horizontal line in the upper half place. Thus, $\cir_2$ has an equation of the form $y=\sqrt{3}n$, for some $n\in\ZZ^+$. 
    
    As the three intersection points are distinct, $\cir_3$ must be a circle with positive curvature. Thus, the radius of $\cir_3$ must be exactly $\sqrt{3}n/2$, hence curvature $\frac{2}{\sqrt{3}n}\leq \frac{2}{\sqrt{3}}<\sqrt{3}$. This contradicts Lemma~\ref{lem:mincurv}.
\end{proof}

Lemma~\ref{lem:notangenttriangles} is powerful in demonstrating that certain tangencies are immediate. First, another useful lemma.

\begin{lemma}\label{lem:4wheelseparatingcircle}
    Let $\cir_1, \cir_2, \cir_3, \cir_4$ be oriented circles with disjoint interiors, where $\cir_i$ is externally tangent to $\cir_{i+1}$ at point $P_i$ for $1\leq i\leq 4$. Let $\cir'$ be a circle which is tangent to $\cir_1$ at $P_1$, whose interior contains $\cir_2$ and exterior contains $\cir_1$, and who does not intersect $\cir_3$ or $\cir_4$ in more than one point. Then, $\cir'$ must be tangent to $\cir_3$ and $\cir_4$ at $P_3$.
\end{lemma}
\begin{proof}
    This setup is depicted in Figure~\ref{fig:onewheel}.

    \begin{figure}[htb]
        \includegraphics[scale=0.8]{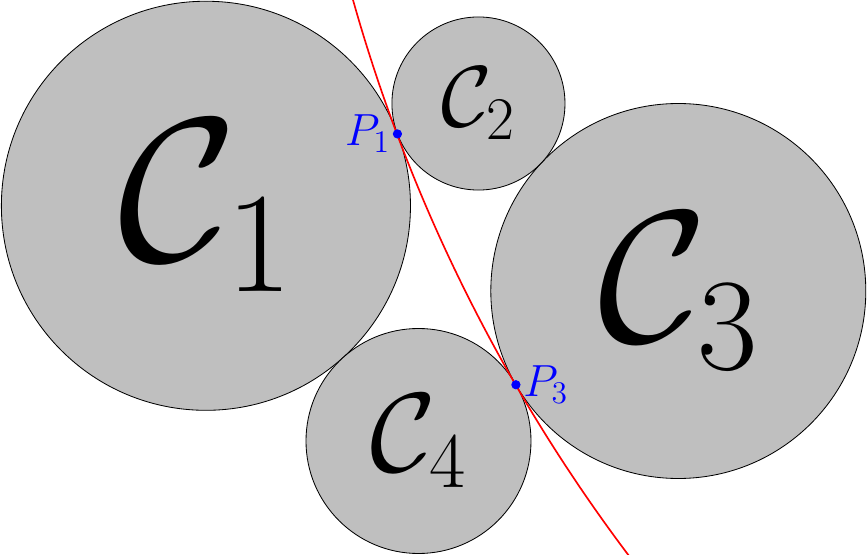}
        \caption{Four consecutively tangent circles, with the unique possible $\cir'$ from Lemma~\ref{lem:4wheelseparatingcircle} in red.}\label{fig:onewheel}
    \end{figure}

    Consider $\cir'=\cir_2$, and continuously decrease the curvature, while keeping the circle tangent to $\cir_1$ (and $\cir_2$) at $P_1$. Immediately, it will intersect $\cir_3$ in two places, until eventually becoming tangent to $\cir_3$. After this it must intersect $\cir_4$ in two places, until eventually becoming $\cir_1$.
    
    In particular, the only possibility for $\cir'$ to be at most tangential to all four circles is if it is tangent to $\cir_3$ and $\cir_4$ simultaneously, necessarily at $P_3$.
\end{proof}

We can now demonstrate that 4-wheels must all be a part of the same maximal tangency packing. 

\begin{lemma}\label{lem:4wheelisimmediatetangency}
    Let $(\cir_1, \cir_2, \cir_3, \cir_4)$ be a 4-wheel with each circle $\cir_i$ being an Eisenpint circle. Then, the maximal immediate tangency packing inside $\widepint$ corresponding to $\cir_1$ contains the entire 4-wheel.
\end{lemma}
\begin{proof}
    It suffices to show that $\cir_2$ is in the maximal immediate tangency packing. Assume otherwise: thus, there exists an Eisenpint circle $\cir'$, which is tangent to $\cir_1$ at $P_1$ (the tangency point of $\cir_1$ and $\cir_2$), which contains $\cir_2$ and is in the exterior of $\cir_1$. As $\cir'$ and all the $\cir_i$'s are in the Eisenpint arrangement, they have at most tangential intersections, and therefore satisfy the conditions of Lemma~\ref{lem:4wheelseparatingcircle}. In particular, $\cir_1, \cir', \cir_4$ is a triple of Eisenpint circles which are mutually tangent at distinct places, contradicting Lemma~\ref{lem:notangenttriangles}.
\end{proof}

Not only do we have 4-wheels, but also all of their swaps.

\begin{lemma}\label{lem:4wheelandallswapsinsideEnEisenstein}
    Let $(\cir_1, \cir_2, \cir_3, \cir_4)$ be a 4-wheel inside of $\widepint$. For $1\leq i\leq 4$, let $\cir_i'$ be the circle obtained by swapping $\cir_i$, as in Definition~\ref{def:circleswap}. Then, each $\cir_i'\in\widepint$, and they are all a part of the maximal immediate tangency packing corresponding to $\cir_1$.
\end{lemma}
\begin{proof}
    Let $\cir_i$ correspond to reduced coordinates $(s_i, t_i, x_i, z_i)\in\ZZ^4$. As this is a linear change of basis from the general $(u, v, p, q)$ coordinates representing oriented circles, we see that Proposition~\ref{prop:circleswapmatrices} still applies. In particular,
    \[\cir_1'=-\cir_1+2(\cir_2+\cir_4),\quad \cir_2'=-\cir_2+2(\cir_1+\cir_3),\quad \cir_3'=-\cir_3+2(\cir_2+\cir_4),\quad \cir_4'=-\cir_4+2(\cir_1+\cir_3).\]
    Thus, we see that $\cir_i'\in\ZZ^4$ and $\cir_i'\equiv \cir_i\pmod{2}$ for all $i$. By Proposition~\ref{prop:eisensteincircleconverse}, $\cir_i'$ is in the full Schmidt arrangement. Since $\cir_i'\equiv \cir_i\pmod{2}$ and $\cir_i$ is an Eisenpint circle, so $\cir_i'$ is as well. The final claim follows from Lemma~\ref{lem:4wheelisimmediatetangency}, as each $\cir_i'$ is in a 4-wheel containing three of the $\cir_j$'s.
\end{proof}

Note that the Eisenpint Schmidt arrangement contains the 4-wheel with reduced coordinates
\[\Dcir_1=(0, 0, 0, 1),\quad \Dcir_2 = (0, 2, 0, -1), \quad \Dcir_3 = (1, 2, 2, -1),\quad \Dcir_4=(1, 0, 0, -1),\]
giving circles with respective equations
\[\Dcir_1: y=0,\quad \Dcir_2:y=\sqrt{3},\quad \Dcir_3: (x-1)^2+(y-2/\sqrt{3})^2=1/3,\quad \Dcir_4: x^2+(y-1/\sqrt{3})^2=1/3,\]
and correspond to
\begin{equation}\label{eqn:4wheelbase}
    \Dcir_1 = -\lm{1}{0}{0}{1}(\widehat{\RR}),\quad \Dcir_2 = \lm{1}{2\om}{0}{1}(\widehat{\RR}),\quad \Dcir_3 = -\lm{-1+2\om}{2\om}{\om}{1}(\widehat{\RR}),\quad \Dcir_4 = \lm{1}{0}{-\om}{1}(\widehat{\RR}).
\end{equation}
The circle $\Dcir_1$ is immediately tangent to $\Dcir_2$ at $\infty$, and to $\Dcir_4$ at $0$. These four base circles form the start of the Eisenstein strip packing, depicted in Figure~\ref{fig:strip} and \ref{fig:stripbase}. They are useful in the following claim, as well as in Section~\ref{sec:strongapproximation}.

\begin{lemma}\label{lem:everytangencyis4wheel}
    Let $\cir,\cir'\in\widepint$ be immediately tangent. Then, there exists a 4-wheel in $\widepint$ containing $\cir$ and $\cir'$.
\end{lemma}
\begin{proof}
    Write $\cir=\pm M_1(\widehat{\RR})$ for $M_1\in \GammaE$. Applying $M_1^{-1}$ and possibly $\pm$, we can assume that $\cir=\widehat{\RR}$. The intersection point with $\cir'$ is thus in $\widehat{\QQ}$. We claim that we can further move this intersection point to $0$ or $\infty$.

    If it is not $\infty$, it is $p/q$ for $p,q\in\ZZ$ coprime. If $p$ is even, then $q$ is odd, and there exists a matrix $M_2=\sm{q}{-p}{\ast}{\ast}\in\Gamma_1^T(2)\subseteq\GammaE$, which preserves $\widehat{\RR}$ and sends the intersection point to $0$. If $p$ is odd, then there exists a matrix $M_2=\sm{\ast}{\ast}{q}{-p}\in\Gamma_1^T(2)\subseteq\GammaE$, which preserves $\widehat{\RR}$ and sends the intersection point to $\infty$.

    Now, the image of $\cir$ is $\widehat{\RR}=\Dcir_1$, and the image of $\cir'$ is immediately tangent to $\widehat{\RR}$ at either $0$ or $\infty$. In the first case, the image must be $\Dcir_4$, and in the second, it must be $\Dcir_2$. In any case, the image is part of the 4-wheel $(\Dcir_1, \Dcir_2, \Dcir_3, \Dcir_4)$, so working backwards, we see that our original tangency $\cir-\cir'$ is a part of a 4-wheel.
\end{proof}

We have all the geometric setup required to prove we get Eisenstein circle packings.

\begin{proposition}
    Every Eisenpint tangency packing is an Eisenstein circle packing.
\end{proposition}
\begin{proof}
    Pick a circle $\cir\in\widepint$. By Lemma~\ref{lem:everytangencyis4wheel}, it is a part of a 4-wheel $W$ in $\widepint$. By Lemma~\ref{lem:4wheelandallswapsinsideEnEisenstein}, $W$ and all of its circle swaps are inside the immediate tangency packing generated by $\cir$. Iterating, it implies that the Eisenstein circle packing generated by $W$ is contained inside the Eisenpint tangency packing generated from $\cir$.

    To demonstrate that it is everything, consider any immediately tangent circle $\cir'$ to $\cir$. By Lemma~\ref{lem:everytangencyis4wheel}, this tangency is described by a 4-wheel in $\widepint$. Repeating the above argument, we embed $\cir$ and $\cir'$ inside another Eisenstein circle packing. It suffices to demonstrate that these Eisenstein circle packings are equal.

    For this, it suffices to demonstrate for $\cir=\Dcir_1$, the real line whose interior is the upper half plane, as any example can be M\"obius transformed to this case. Then, there is a unique packing that $\Dcir_1$ is a part of up to $\RR-$translation: the strip packing (Figure~\ref{fig:strip}), scaled by $\frac{1}{\sqrt{3}}$. Consider the curvature 1 circles from the primitive strip packing: we see that we cannot fit another in between two consecutive, without them intersecting. Thus, the two packings must overlay the curvature 1 circles, and thus be identical.
\end{proof}

\begin{figure}[htb]
    \includegraphics[scale=1.2]{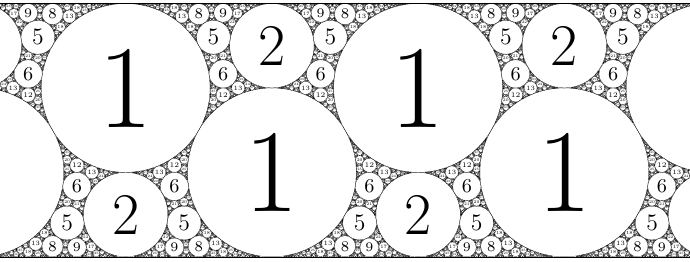}
    \caption{A slice of the Eisenstein strip circle packing, i.e. generated from root quadruple $(0, 0, 1, 1)$. Circles of curvatures up to 300 are shown.}\label{fig:strip}
\end{figure}

The final step is to demonstrate that the Eisenpint tangency packings are primitive, and show that the quadratic form equivalence class associated to a circle is the same as the one associated to the Eisenstein packing it generates.

To begin, let $M=\sm{*}{*}{e+f\om}{g+h\om}\in\GammaE$. Take the 4-wheel in Equation~\eqref{eqn:4wheelbase}, and multiply by $M$ on the left to get a 4-wheel containing $-M(\widehat{\RR})$. From Proposition~\ref{prop:eisensteincircleformulae}, we compute the four curvatures to be $(a, b, c, d)$, where
\begin{align*}
    a & = fg-eh,\\
    b & = 2e^2+2f^2+2ef+eh-fg,\\
    c & = 2e^2+2f^2+g^2+h^2+2ef+2eg+3eh-fg+2fh+gh,\\
    d & = g^2+h^2+eh-fg+gh.
\end{align*}

\begin{lemma}
    Every Eisenpint tangency packing is primitive.
\end{lemma}
\begin{proof}
    Take the above equations for $a,b,c,d$ in terms of $e,f,g,h$. First, we claim that their $\gcd$ is odd. Since $M\in\GammaE$, $g$ is odd and $h$ is even. If $f$ is odd, then $a=fg-eh\equiv 1\pmod{2}$, and if $f$ is even, then $d=g^2+h^2+eh-fg+gh\equiv 1\pmod{2}$. In both cases either $a$ or $d$ is odd, so the $\gcd$ is odd. Also note that $a\equiv b\pmod{2}$ and $c\equiv d\pmod{2}$, so $(a, b, c, d)$ is in standard position.

    Now, let $p$ be an odd prime, and assume that $p\mid a, b, c, d$. It follows that
    \begin{align*}
        p & \mid fg-eh = \frac{1}{\sqrt{-3}}(\beta\overline{\delta}-\overline{\beta}\delta),\\
        p & \mid e^2+ef+f^2 = \Nm(\beta), \\
        p & \mid g^2+gh+h^2 = \Nm(\delta),
    \end{align*}
    where $\beta=e+f\om$ and $\delta=g+h\om$ as usual. Since $\beta$ and $\delta$ must be coprime, this implies that $p$ must split in $\ZZ[\om]$ into $p=\pi\overline{\pi}$, where $\pi\mid \beta$ and $\overline{\pi}\mid\delta$. However, we then have $\pi\mid \beta\overline{\delta}-\overline{\beta}\delta$, whence $\pi\mid \overline{\beta}$ or $\pi\mid\delta$. In either case we find a common factor of $\beta$ and $\delta$, contradiction.

    Therefore we have no common factors, so the quadruple $(a, b, c, d)$ is primitive.
\end{proof}

Now, we know that the image is a primitive 4-wheel, which is in standard position. Applying the map $\phi$ from Proposition~\ref{prop:eisensteinpackingbijectionqf}, this is sent to the ($\Gamma_0^{\PGL}(2)-$equivalence class of the) quadratic form
\[\left[a+d, b+d-c, \frac{a+b}{2}\right] = [g^2+gh+h^2, -2eg-eh-fg-2fh, e^2+ef+f^2].\]
On the other hand, we compute that
\begin{align*}
    f_M(x, y) & =\Nm(\delta x-\beta y)\\
    & =\Nm((gx-ey)+(hx-fy)\om)\\
    & =(gx-ey)^2+(gx-ey)(hx-fy)+(hx-fy)^2\\
    & = (g^2+gh+h^2)x^2 - (2eg-eh-fg-2fh)xy + (e^2+ef+f^2)y^2,
\end{align*}
which is the same quadratic form as above. As both of these forms were defined up to $\Gamma_0^{\PGL}(2)-$equivalence, the following result holds.

\begin{theorem}
    Let $\cir=\pm M(\widehat{\RR})$ be an Eisenpint circle, where $M\in\GammaE$. Then, the Eisenpint tangency packing generated by $\cir$ is a primitive Eisenstein circle packing scaled by $1/\sqrt{3}$. 
    
    Next, let $q$ be an Eisenstein quadruple inside this Eisenstein circle packing, with $\cir$ representing the first curvature. Then, the $\Gamma_0^{\PGL}(2)-$equivalence class of $\phi(q)(x, y)$ is equal to the $\Gamma_0^{\PGL}(2)-$equivalence class of $f_M(x, y)$.

    Finally, this association is bijective on equivalence classes: any other circle which gives rise the the same equivalence class of quadratic forms must be equivalent (up to orientation) to $\cir$.
\end{theorem}

A simplified version of this is Theorem~\ref{thm:Qrt3packingsareEisenstein}, which is now proven.

\section{Strong approximation and density one local-global}\label{sec:strongapproximation}

The goal is to apply the results of \cite{FSZ19}, which give strong approximation and a density one local-global theorem in a very general setting. To accomplish this, we need to describe an Eisenstein circle packing as a union of orbits of circles acted upon by a Kleinian group possessing certain properties. We start by constructing the appropriate group.

Recall the ``base quadruple'' $(\Dcir_1, \Dcir_2, \Dcir_3, \Dcir_4)$ given in Equation~\eqref{eqn:4wheelbase}, which is also depicted in Figure~\ref{fig:stripbase}. The four basic circle swaps applied to this base quadruple are inversions in the following circles:
\begin{enumerate}
    \item $\mathcal{T}_1$, circle with radius $1$ centred at $\sqrt{3}i$, i.e. $z \mapsto  \frac{\sqrt{-3}\,\bar{z} - 2}{\bar{z} + \sqrt{-3}} $;
    \item $\mathcal{T}_2$, circle with radius $1$ centred at $1$, i.e. $z \mapsto \frac{ \overline{z} }{\overline{z}-1}$;
    \item $\mathcal{T}_3$, line $\Re(z) = 0$, i.e. $z \mapsto -\overline{z}$;
    \item $\mathcal{T}_4$, line $\Re(z) = 1$, i.e. $z \mapsto -\overline{z} + 2$.
\end{enumerate}

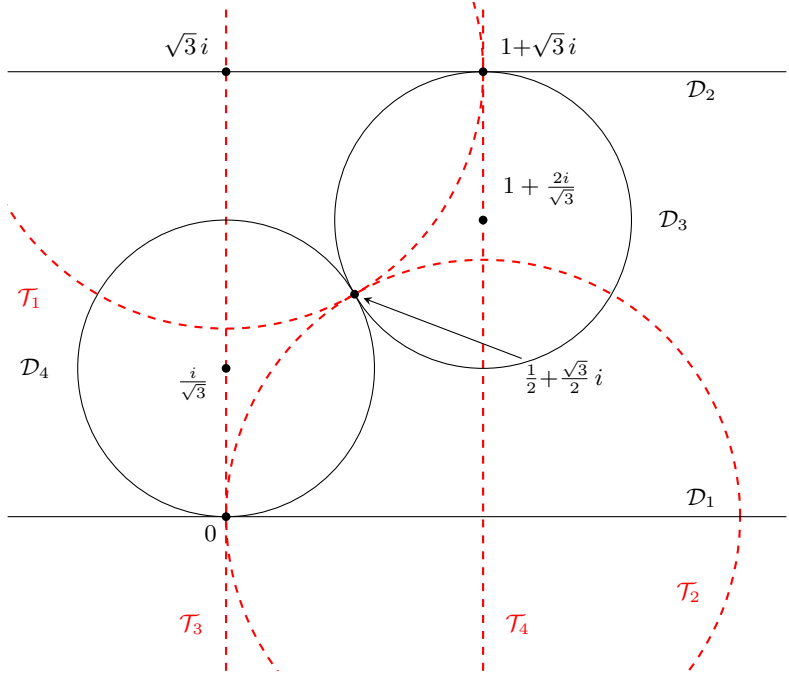
\begin{figure}
\newcommand{\sq}{1.73205}    % sqrt(3)
\newcommand{\isq}{0.57735}   % 1/sqrt(3)
\newcommand{\tisq}{1.15470}  % 2/sqrt(3)
\newcommand{\hsq}{0.86603}   % sqrt(3)/2

\begin{tikzpicture}[scale=3.4, every node/.style={font=\small}]

  \clip (-0.85, -0.6) rectangle (2.18, 2.0);

  % === T circles (red dashed, drawn first so solid elements sit on top) ===

  % T_1: circle at (0, sqrt(3)), r = 1
  \draw[dashed, thick, red] (0, \sq) circle[radius=1];
  \node[red, left] at ({-0.68}, 0.85) {$\mathcal{T}_1$};

  % T_2: circle at (1, 0), r = 1
  \draw[dashed, thick, red] (1, 0) circle[radius=1];
  \node[red, below] at (1.8, -0.22) {$\mathcal{T}_2$};

  % T_3: line x = 0
  \draw[dashed, thick, red] (0, -0.6) -- (0, 2.0);
  \node[red, left] at (-0.05, -0.42) {$\mathcal{T}_3$};

  % T_4: line x = 1
  \draw[dashed, thick, red] (1, -0.6) -- (1, 2.0);
  \node[red, right] at (1.05, -0.42) {$\mathcal{T}_4$};

  % === D circles (solid black, on top) ===

  % D_1: real axis
  \draw (-0.85, 0) -- (2.18, 0);
  \node[above] at (1.85, 0) {$\Dcir_1$};

  % D_2: line y = sqrt(3)
  \draw (-0.85, \sq) -- (2.18, \sq);
  \node[below] at (1.85, \sq) {$\Dcir_2$};

  % D_3
  \draw (1, \tisq) circle[radius=\isq];
  \node[right] at ({1 + \isq + 0.07}, \tisq) {$\Dcir_3$};

  % D_4
  \draw (0, \isq) circle[radius=\isq];
  \node[left] at ({-\isq - 0.07}, \isq) {$\Dcir_4$};

  % === Labeled points ===

  % Origin
  \fill (0, 0) circle[radius=0.017];
  \node[below left] at (0, 0) {$0$};

  % sqrt(3) i
  \fill (0, \sq) circle[radius=0.017];
  \node[above left] at ({-0.03}, {\sq + 0.02}) {$\sqrt{3}\,i$};

  % 1 + sqrt(3) i
  \fill (1, \sq) circle[radius=0.017];
  \node[above right] at ({1 + 0.03}, {\sq + 0.02}) {$1{+}\sqrt{3}\,i$};

  % Center of D_1
  \fill (0, \isq) circle[radius=0.017];
  \node[below left] at (-0.03, {\isq + 0.04}) {$\frac{i}{\sqrt{3}}$};

  % Center of D_2
  \fill (1, \tisq) circle[radius=0.017];
  \node[above right] at ({1 + 0.04}, {\tisq + 0.02}) {$1+\frac{2i}{\sqrt{3}}$};

  % Tangent point D_1 ∩ D_2
  \fill (0.5, \hsq) circle[radius=0.017];
  \draw[-stealth, thin, shorten >=1.5pt]
    ({1.15}, {\hsq - 0.25}) -- ({0.52}, {\hsq - 0.01});
  \node[right] at ({1.12}, {\hsq - 0.32}) {$\frac{1}{2}{+}\frac{\sqrt{3}}{2}\,i$};
  
\end{tikzpicture}
  \caption{The base quadruple, consisting of two horizontal lines through $0$ and $\sqrt{-3}$, as well as circles of radius $1/\sqrt{3}$ centred at $i/\sqrt{3}$ and $1 + \frac{2i}{\sqrt{3}}$. In dotted lines are the dual circles, representing the inversions giving the four circle swaps.
  }
  \label{fig:stripbase}
  \end{figure}

\begin{remark}
\label{rmk:crystl}
    Using the $\Dcir_i$ as a cluster and the $T_i$ as a co-cluster, this structure gives the Eisenpint packing as a crystallographic superpacking.  The Gram matrix of inversive products for $\Dcir_1,\Dcir_2,\Dcir_3,\Dcir_4,\mathcal{T}_1,\mathcal{T}_2,\mathcal{T}_3,\mathcal{T}_4$ is
    \[
    \begin{pmatrix}
1 & -1 & -2 & -1 & -\sqrt{3} & 0 & 0 & 0 \\
-1 & 1 & -1 & -2 & 0 & -\sqrt{3} & 0 & 0 \\
-2 & -1 & 1 & -1 & 0 & 0 & -\sqrt{3} & 0 \\
-1 & -2 & -1 & 1 & 0 & 0 & 0 & -\sqrt{3} \\
-\sqrt{3} & 0 & 0 & 0 & 1 & -1 & 0 & -1 \\
0 & -\sqrt{3} & 0 & 0 & -1 & 1 & -1 & 0 \\
0 & 0 & -\sqrt{3} & 0 & 0 & -1 & 1 & -1 \\
0 & 0 & 0 & -\sqrt{3} & -1 & 0 & -1 & 1
\end{pmatrix}.
    \]
\end{remark}

We consider the \emph{Ford circles}, which are those circles in $\widepint$ which are immediately tangent to the real axis. The tangency point is necessarily in $\widehat{\QQ}$, and these circles are in the Eisenpint tangency packing corresponding to $\widehat{\RR}$: call this the \emph{Eisenpint strip packing}.

Given one Ford circle, we can produce an orbit of Ford circles by any matrix action of

\[\GammaE\cap\SL(2, \ZZ) =\Gamma_1^T(2) := \left\{ \begin{pmatrix} a & b \\ c & d \end{pmatrix} \in \SL(2,\ZZ) : a \equiv d \equiv 1 \!\! \pmod{2}, \;\; b \equiv 0 \!\!\pmod{2} \right\}.\]

Acting on $\widehat{\FF_2}$, there are two orbits: $\{0\}$ and $\{1, \infty\}$. Accordingly, it breaks $\widehat{\QQ}$ into two orbits also, being the preimages under reduction. The orbit $\Gamma_1^T(2)\cdot \Dcir_4$ produces the Ford circles tangent to $\widehat{\RR}$ as $p/q$ with $p$ even, and the orbit $\Gamma_1^T(2)\cdot \Dcir_2$ is the Ford circles circles tangent to $\widehat{\RR}$ at $p/q$ with $p$ odd.

As in the introduction, let
\[
V: z \mapsto -z + 1 + \sqrt{-3}, \quad \text{i.e.} \quad V =
\begin{pmatrix}
    -1 & 1 + \sqrt{-3} \\
    0 & 1
\end{pmatrix},
\]
which is a $180-$degree rotation about the tangency point between $\Dcir_3$ and $\Dcir_4$, and has determinant $-1$. It swaps the pairs $(\Dcir_1, \Dcir_2)$ and $(\Dcir_3, \Dcir_4)$, and preserves curvature. As in the introduction, define two groups:
\[
\Egp' := \langle \Gamma_1^T(2), V \rangle, \quad
\Egp'' := \Egp' \cap \SL(2,\ZZ[w]]).
\]
Our first main goal is to demonstrate that $\Egp'$ applied to $\Dcir_1$ (the lower line) and $\Dcir_4$ (the left circle) produce the circles in the Eisenpint strip packing.

\begin{lemma}\label{lem:Egpevenvsodd}
    Circles in the orbit of $\Egp'$ acting on $\Dcir_1$ and $\Dcir_4$ all lie in the Eisenpint Schmidt arrangement. Furthermore, those that are the images of $\Dcir_1$ have even reduced curvature, and those that are the images of $\Dcir_4$ have odd reduced curvature.
\end{lemma}
\begin{proof}
    Consider a word $W$ in $\Egp'$, consisting of elements of $\Gamma_1^T(2)$ and $V$'s. As $V=\sm{-1}{2\om}{0}{1}\equiv \sm{1}{0}{0}{1}\pmod{2}$, we see that $W\equiv \sm{1}{0}{\ast}{1}\pmod{2}$. If there are an even number of $V$'s, then it has determinant 1, hence is in $\Egp$, whence $W(\Dcir_i)$ is in the Eisenpint Schmidt arrangement. If there are an odd number of $V$'s, then $W(\Dcir_i)=WV(\Dcir_j)$ where $(i, j)=(1, 2)$ or $(4, 3)$. Then, $WV\in\Egp$, and again $WV(\Dcir_j)=W(\Dcir_i)$ is in the Eisenpint Schmidt arrangement.

    For the second claim, $V$ does not change a circle's curvature. Thus, it suffices to show that $\Gamma_1^T(2)$ acting on the Eisenpint Schmidt arrangement also preserves parity of reduced curvature.

    Pick any such circle, say it has reduced coordinates $(s, t, x, z)$, and is the image of $M(\widehat{\RR})$, where $M=\sm{A+B\om}{C+D\om}{E+F\om}{G+H\om}\in\GammaE$. From Proposition~\ref{prop:eisensteincircleformulae}, $s=EH-FG$, $t=AD-BC$, and $x=AH-BG+DE-CF$. From Proposition~\ref{prop:GammaEcoset04} and Table~\ref{table:F4circles}, we know that $t\equiv x\equiv 0\pmod{2}$. Given an element $N\in\Gamma_1^T(2)$, if $N\equiv \sm{1}{0}{0}{1}\pmod{2}$, then $NM\equiv M\pmod{2}$, hence the parity of $s$ is unchanged. Otherwise,
    \[NM\equiv \lm{1}{0}{1}{1}M\equiv \lm{A+B\om}{C+D\om}{(E+A)+(F+B)\om}{(G+C)+(H+D)\om}\pmod{2},\]
    so the new reduced curvature is
    \begin{align*}
        s' & \equiv (E+A)(H+D)-(F+B)(G+C)\\
        & \equiv (EH-FG)+(AD-BC)+(AH-BG+ED-FC)\equiv s+t+x\equiv s\pmod{2},
    \end{align*}
    also unchanged.
\end{proof}

We can now prove that $\Egp'$ and $\Egp''$ produce the strip packing.

\begin{proposition}\label{prop:orbitsE}
    The group $\Egp'$, applied to $\Dcir_1$ and $\Dcir_4$, will give the complete set of circles of the Eisenpint strip packing. The group $\Egp''$, applied to $\Dcir_1$, $\Dcir_2$, $\Dcir_3$ and $\Dcir_4$, will give the complete set of circles of the Eisenpint strip packing.
\end{proposition}

\begin{proof}
First, we claim that if $W\in\Egp'$, then $W(\Dcir_i)$ is in the Eisenpint strip packing for all $i=1,2,3,4$. Consider the word $W$ as comprised of \emph{chunks}, where $V$ is one chunk, and any element of $\Gamma_1^T(2)$ is one chunk. We induct on the number of chunks in $W$.

First, $\Dcir_1$ through $\Dcir_4$ are all consecutively immediately tangent and lie in the Eisenpint strip packing, giving the result for the empty word. In general, if the word $W$ ends in $V$, then write $W=W'V$, and $W(\Dcir_i)=W'(\Dcir_j)$ for $\{i, j\}=\{1, 2\}$ or $\{3, 4\}$, hence we have already found this circle. Otherwise, the word $W$ ends in a non-trivial element of $\Gamma_1^T(2)$: $W=W'M$ with $M\in\Gamma_1^T(2)$ being the final chunk. If $i=1$, then $M(\Dcir_1)=M(\widehat{\RR})=\widehat{\RR}$, so $W(\Dcir_1)=W'(\Dcir_1)$, and this circle was already found to be in the Eisenpint strip packing. We see that this is immediately tangent to $W(\Dcir_2)$ and $W(\Dcir_4)$, whence they are in the Eisenpint strip packing as well. Finally, each of these circles is immediately tangent to $W(\Dcir_3)$, finishing the induction.  This shows $\Egp'(\Dcir_1)\cup\Egp'(\Dcir_4)$ is a subset of the Eisenpint strip packing.

Next, let $R_0=\{\Dcir_1\}$, and for all $n\geq 1$, let $R_n$ denote all the Eisenpint strip packing circles which are in $R_{n-1}$ or immediately tangent to a circle in $R_{n-1}$ of \emph{even} reduced curvature. We claim by induction that $R_n$ is in the union $\Egp'(\Dcir_1)\cup\Egp'(\Dcir_4)$ for all $n$. The claim is immediate for $n=0$, and for $n=1$ this is the claim that all Ford circles are in $\Egp'(\Dcir_1)\cup\Egp'(\Dcir_4)$, which was demonstrated at the start of this section. In general, take a circle $\cir$ in $R_{n+1}$ that is not in $R_n$. Thus, it is tangent to a circle $\cir'\in R_n$ of even reduced curvature. By Lemma~\ref{lem:Egpevenvsodd}, $\cir'=M(\Dcir_1)=M(\widehat{\RR})$ for $M\in\GammaE$. Consider all the circles in $M(R_1)$: this is $M$ acting on $\widehat{\RR}$ and all Ford circles, i.e. $\widehat{\RR}$ and all immediately tangent circles. In particular, the result will be $\cir'$ and all immediately tangent circles in the Eisenpint arrangement, which therefore includes $\cir$. Thus, the induction holds.

Finally, we show that all Eisenpint strip packing circles are found in $\cup_{n=0}^{\infty} R_n$. Pick any Eisenpint strip packing circle $\cir$, and it suffices to show that we can walk via tangencies across \emph{only even curvature} circles to reach $\widehat{\RR}$.  Pick any 4-wheel containing our starting circle $\cir$, and consider reducing this 4-wheel, necessarily ending up at a cyclic permutation of either $(\Dcir_1, \Dcir_2, \Dcir_3, \Dcir_4)$ or $(\Dcir_4, \Dcir_3, \Dcir_2, \Dcir_1)$. Take the union of the circles we see along the way: of all the $4$-wheels in the reduction process. We claim this union contains a tangency path from $\cir$ to $\Dcir_1$, consisting only of even curvatures (after the starting circle). 

They key observation is that in a given 4-wheel with primitive integral curvatures, every circle is tangent to exactly 1 odd and 1 even curvature circle: this is the content of Lemma~\ref{lem:twoeventwoodd}. Now, start a tangency path at $\cir$.  We consider step-by-step the 4-wheel reduction process. When we get to a point where we are swapping out $\cir$, then to the tangency path we add the even curvature circle $\cir'$ that is tangent to  $\cir$ in the current 4-wheel.  This is neither $\cir$ nor the circle replacing it in the swap.   Continue reducing the 4-wheel, and when we swap out $\cir'$, we again extend the tangency path by adding the even curvature circle $\cir''$ that is tangent to $\cir'$ in the current 4-wheel.   In this way the last circle in our tangency path is always in the current $4$-wheel.  Repeating this procedure, at the very end, we have a tangency path from $\cir$ to one of $\Dcir_1, \Dcir_2, \Dcir_3, \Dcir_4$, where all curvatures (save possibly the first) are even. If we are at $\Dcir_1$ we are done, and otherwise we finally go from $\Dcir_2$ to $\Dcir_1$, and are done. This completes the proof that $\Egp'(\Dcir_1)\cup\Egp'(\Dcir_4)$ is the Eisenpint strip packing.

The claim for $\Egp''$ follows quickly from this. For any word $W$ with an odd number of $V$'s, we have $W(\Dcir_1)=WV(\Dcir_2)$ and $W(\Dcir_4)=WV(\Dcir_3)$, which implies the result.
\end{proof}

\begin{proposition}
    \label{prop:moietyorbit}
    The orbits $\Egp'' \cdot \Dcir_1$ and $\Egp'' \cdot \Dcir_3$ form one moiety of the strip Eisenstein packing, and $\Egp'' \cdot \Dcir_2$ and $\Egp''\cdot \Dcir_4$ form the other.
\end{proposition}
\begin{proof}
    If a circle is in one moiety of the packing, then any immediately tangent circles must be in the other. With this in mind, call the two moieties of our packing ``blue'' and ``red'', and say $\Dcir_1$ is blue. It follows that $\Dcir_2,\Dcir_4$ are red, hence $\Dcir_3$ is blue.

    We claim that for all $W\in\Egp''$, $W(\Dcir_1)$ and $W(\Dcir_3)$ are blue, whereas $W(\Dcir_2)$ and $W(\Dcir_4)$ are red. As in Proposition~\ref{prop:orbitsE}, write $W$ as a product of chunks, where $V$ is one chunk and any element of $\Gamma_1^T(2)$ is one chunk. We induct on the number of chunks.

    The base case is where $W$ is empty, and was covered above. For the induction step, note that it suffices to prove only one of the four claims ($W(\Dcir_1)$ is blue, etc.), as the others will all follow immediately from the immediate tangencies swapping moieties.

    If $W$ ends in a non-trivial word in $\Gamma_1^T(2)$, then write $W=W'M$ with $M\in\Gamma_1^T(2)$ the final chunk. Then, $W(\Dcir_1)=W'(M(\Dcir_1))=W'(\Dcir_1)$, where $W'$ has one less chunk than $W$. By induction, $W'(\Dcir_1)$ is blue, which implies the result.

    Otherwise, $W$ ends in a $V$, and since $W\in\Egp''$, we can write $W=W'VMV$ for some $W'\in\Egp''$ and $M\in\Gamma_1^T(2)$. Now, we see that
    \[W(\Dcir_2)=W'VMV(\Dcir_2)=W'VM(\Dcir_1)=W'V(\Dcir_1)=W'(\Dcir_2).\]
    As $W'$ has 3 less chunks than $W$, by induction, $W'(\Dcir_2)=W(\Dcir_2)$ is red, which again implies the result.
.
\end{proof}

In an analogous fashion to Proposition~\ref{prop:gamma0pgl2generators}, the following elements generate $\Gamma_1^T(2)$: 
\[
T: z \mapsto z + 2, \quad
S: z \mapsto z/(z+1).
\]
The fundamental domain of $\Gamma_1^T(2)$ in the upper half plane is bounded by $\Re(z) = 1$, $\Re(z) = -1$, the circle centred at $1$ of radius $1$, and the circle centred at $-1$ of radius $1$, having two cusps, as in Figure~\ref{fig:fundreg}.

\begin{figure}
\begin{tikzpicture}[scale=2.2, every node/.style={font=\small}]

  % --- Clip region for shading ---
  % The fundamental domain: between x=-1 and x=1, above the two unit circles,
  % up to some height (representing the cusp at infinity).
  \pgfmathsetmacro{\htop}{3.2}

  % Shade the fundamental domain
  % Path: go up x=-1 from where it meets |z+1|=1 (which is at y=0, but the
  % circle only contributes in the upper half), up to top; across to x=1;
  % down to where x=1 meets |z-1|=1; then arc along |z-1|=1 to origin;
  % then arc along |z+1|=1 back to x=-1.
  \begin{scope}
    \fill[colblue]
      (-1, 0) -- (-1, \htop) -- (1, \htop) -- (1, 0)
      arc[start angle=0, end angle=90, radius=1]     % along |z-1|=1 from (1,0) up to (0,1)
      -- (0, 1)                                        % top of right arc (redundant but clear)
      arc[start angle=90, end angle=180, radius=1]    % along |z+1|=1 from (0,1) to (-1,0)
      -- cycle;
    % Now cut out the interiors of the two circles by filling white
    \fill[white]
      (1, 0) arc[start angle=0, end angle=180, radius=1] -- (1,0) -- cycle;
  \end{scope}
  % Redo: cleaner approach
  \begin{scope}
    \clip (-1, 0) rectangle (1, \htop);
    \fill[colblue]
      (-1, 0) -- (-1, \htop) -- (1, \htop) -- (1, 0) -- cycle;
    % Cut out circle |z-1|=1
    \fill[white] (1, 0) circle[radius=1];
    % Cut out circle |z+1|=1
    \fill[white] (-1, 0) circle[radius=1];
  \end{scope}

  % --- Boundary curves ---

  % Vertical lines (only upper half plane portion shown)
  \draw[thick] (-1, 0) -- (-1, \htop);
  \draw[thick] (1, 0) -- (1, \htop);

  % Circle arcs in upper half plane
  % |z+1|=1: center (-1,0), from angle 0 (at origin) to angle 90 (at (-1,1))
  % But we want the arc from (-1,0) side... 
  % Center (-1,0): angle 0 -> point (0,0), angle 90 -> point (-1,1)
  \draw[thick] (0,0) arc[start angle=0, end angle=90, x radius=1, y radius=1];
  % Actually let's just draw the upper semicircles and clip
  \begin{scope}
    \clip (-1, 0) rectangle (1, 1.05);
    \draw[thick] (-1, 0) circle[radius=1];
    \draw[thick] (1, 0) circle[radius=1];
  \end{scope}

  % --- Dashed continuation of circles below to show they're circles ---
  % (optional, skip for cleanliness)

  % --- Axes for context (light) ---
  \draw[gray, thin, ->] (-1.8, 0) -- (1.8, 0) node[right, black] {$\operatorname{Re}$};
  \draw[gray, thin, ->] (0, -0.15) -- (0, \htop + 0.15);

  % --- Tick marks and labels on real axis ---
  \foreach \x/\lab in {-1/{-1}, 1/{1}} {
    \draw (\x, 0.04) -- (\x, -0.04);
    \node[below] at (\x, -0.06) {$\lab$};
  }

  % --- Label the cusp at 0 ---
  \fill (0, 0) circle[radius=0.03];
  \node[below] at (0, -0.06) {$0$};

  % --- Label the circles ---
  \node[gray] at (-1.35, 0.6) {$|z{+}1|{=}1$};
  \node[gray] at (1.35, 0.6) {$|z{-}1|{=}1$};

  % --- Label boundary lines ---
  \node[left] at (-1, 2.2) {$\operatorname{Re}(z)={-}1$};
  \node[right] at (1, 2.2) {$\operatorname{Re}(z)=1$};

  % --- Corner labels ---
  \fill (-1, 1) circle[radius=0.03];
  \node[left] at (-1.06, 1) {$-1{+}i$};
  \fill (1, 1) circle[radius=0.03];
  \node[right] at (1.06, 1) {$1{+}i$};

\end{tikzpicture}
  \caption{The fundamental domain of $\Gamma_1^T(2)$.}
  \label{fig:fundreg}
\end{figure}
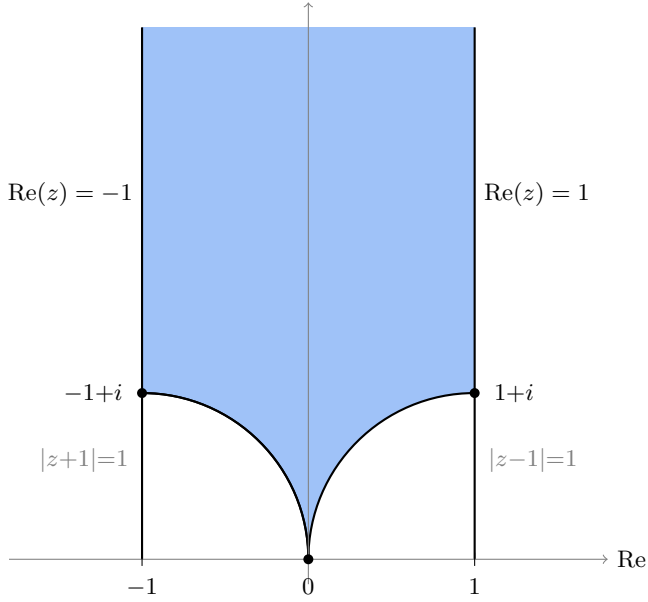

Let $K$ be a number field. 
An algebraic group $G$ has \emph{strong approximation} if $G(K)$ is dense in $G(\mathbb{A}_K^{\infty})$, where $\mathbb{A}_K^\infty$ denotes the finite adeles. A subgroup $\Gamma \subseteq G(\mathcal{O}_K)$, which is Zariski dense in $G$, is said to have strong approximation if the closure in $G(\mathbb{A}_K^\infty)$ is an open subgroup. More directly, $\Gamma$ has strong approximation if $\Gamma$ and $G(\mathcal{O}_K)$ have the same reduction modulo $\mathfrak{p}$ for all but finitely many prime ideals $\mathfrak{p} \subseteq \mathcal{O}_K$, and at those finitely many \emph{bad} primes $\mathfrak{p}$,
\[
\Gamma \bmod{\mathfrak{p}^k} / \Gamma \bmod{\mathfrak{p}^{k-1}}
= 
G(\mathcal{O}_K) \bmod{\mathfrak{p}^k} / G(\mathcal{O}_K) \bmod{\mathfrak{p}^{k-1}}
\]
for sufficiently large $k$, say $k > m_\mathfrak{p}$.  The \emph{bad modulus} is $\prod_{\mathfrak{p} \text{ bad}} \mathfrak{p}^{m_\mathfrak{p}}$.  

In \cite[Theorem 8.1]{FSZ19}, a slight variation on this definition of strong approximation is given, where the bad modulus is a rational integer. This integer controls the congruence obstructions that can occur in certain types of packings arising from Kleinian groups.

\begin{theorem}
\label{thm:hypotheses}
    The group $\Egp''$ satisfies the hypotheses of \cite[Theorem 1.3, 1.6]{FSZ19}, namely:
    \begin{enumerate}
    \item \label{it:PSL} it is a Kleinian group contained in $\PSL(2,\ZZ[\omega])$;
    \item \label{it:congruence} it contains a congruence subgroup of $\PSL(2,\ZZ)$;
    \item \label{it:fg} it is finitely generated;
     \item \label{it:Hausdorff} the limit set has Hausdorff dimension greater than $1$;
    \item \label{it:Zariski} it is Zariski dense (in the sense of \cite{FSZ19});
    \item \label{it:covolume} it has infinite covolume;
    \item \label{it:geofinite} it is geometrically finite.
    \end{enumerate}
\end{theorem}

\begin{proof}
Parts \eqref{it:PSL} through \eqref{it:fg} are immediate from the definition. Next, \cite[Corollary 1.8]{BishopJones} implies that the limit set of a finitely-generated Kleinian group is either totally disconnected, a circle, or of Hausdorff dimension $>1$. The limit set of $\Egp'$ (hence equivalently $\Egp''$) includes $\widehat{\RR}$ and other circles, whence it is of Hausdorff dimension $>1$.  This verifies \eqref{it:Hausdorff}.

Now we consider Zariski density.  The sense in \cite{FSZ19} is to consider $\PGL(2, \CC)$ as a \emph{real} algebraic group; that is, use Weil restriction.  In this setup, we can move between Kleinian groups and subgroups of the orthogonal group $\operatorname{O}_{3,1} < \operatorname{GL}_4$ via the spin homomorphism $\SL(2,\CC) \rightarrow \operatorname{SO}_{3,1}^+(\RR)$.  In particular, consider the Zariski closure $G$ of the image of $\Egp''$.  The possibilities for this group are classified as in the proofs of \cite[Theorem 9.4]{StangeApollonianStructureBianchi} and \cite[Lemma 1.6(ii)]{FuchsStrongApproximation}.  In particular, since $G$ is not finite and satisfies \eqref{it:Hausdorff}, it is either $\operatorname{O}_{3,1}$ or $\operatorname{SO}_{3,1}$, exactly as in \cite[Theorem 9.4]{StangeApollonianStructureBianchi}.  Pulling back through the spin map, we find that $\Egp''$ is Zariski dense, i.e. \eqref{it:Zariski}.

For infinite covolume and geometrical finiteness, we describe a fundamental domain.  Using the generators and fundamental domain already listed for $\Gamma_1^T(2)$ above, together with $V$, we obtain a fundamental domain with the following geodesic walls (given by their boundaries in $\widehat{\CC}$):
\begin{equation*}
\begin{alignedat}{3}
A &: |z - 1| = 1, \qquad &
C &: \Re(z) = 1, \qquad &
E &: \Im(z) = \sqrt{3}/2, \\
B &: |z + 1| = 1, \qquad &
D &: \Re(z) = -1, \qquad &
F &: \Im(z) = -\sqrt{3}/2.
\end{alignedat}
\end{equation*}
It is therefore both of infinite covolume and geometrically finite.  This verifies \eqref{it:covolume} and \eqref{it:geofinite}.
\end{proof}

These include all the hypotheses necessary to apply \cite[Theorem 1.3, 8.1]{FSZ19}, obtaining the following.

\begin{corollary}
The group $\Egp''$ satisfies strong approximation, and has a geometric and combinatorial spectral gap.
\end{corollary}

Furthermore by \cite[Theorem 8.1]{FSZ19} and the remark following, the only bad primes for strong approximation are potentially $p=2,3$. Adopt the notation of this theorem.

\begin{proposition}\label{prop:SAmod3}
    The prime $3$ is a good prime for strong approximation for $\Egp''$.
\end{proposition}
\begin{proof}
First, we bound the exponent of $3$ in the bad modulus to be at most $1$.  Observe that the conditions defining $\Gamma_1^T(2)$ are vacuous modulo powers of $3$.  We can take $\Gamma$ to be a two-power principal congruence subgroup in \cite[Theorem 8.1]{FSZ19}, so that $m_3' = 1$ (in the notation of that theorem).

To apply the theorem, we must also compute the theorem statement's $\iota_3$, i.e. the smallest non-negative integer such that
\[
3^{\iota_3}\mathfrak{sl}(2,\ZZ_3[\omega]) \subseteq \operatorname{Span}_{\ZZ_3} \left( \Egp'' \cdot \mathfrak{sl}(2,\ZZ_3) \right),
\]
where the action in $\Egp'' \cdot \mathfrak{sl}(2,\ZZ_3)$ is the usual conjugation action.  Observe that $\mathfrak{sl}(2,\ZZ)$ is generated by 
\begin{equation}\label{eqn:HRL}
H = \begin{pmatrix} 1 & 0 \\ 0 & - 1 \end{pmatrix}, \quad
R = \begin{pmatrix} 0 & 1 \\ 0 & 0 \end{pmatrix}, \quad
L = \begin{pmatrix} 0 & 0 \\ 1 & 0 \end{pmatrix}.
\end{equation}

Since $\det(V) = -1$, an element of $\Egp''$ is given by
\begin{equation}
\label{eqn:W}
W := VSV = \begin{pmatrix}
    -\sqrt{-3} & 2\sqrt{-3} - 2 \\
    -1 & \sqrt{-3}+2\end{pmatrix}
\end{equation}

Amongst the elements of $\Egp'' \cdot \mathfrak{sl}(2,\ZZ_3)$ are, with $a = \sqrt{-3}$, the images in $\ZZ_3[\sqrt{-3}]$ of these matrices over $\ZZ[\sqrt{-3}]$:
\begin{equation}
\label{eqn:3sl}
\begin{split}
    WHW^{-1} &= \begin{pmatrix} -4a+5 & -4a-12 \\ -2a-4 & 4a -5 \end{pmatrix}, \\
     WRW^{-1} &= \begin{pmatrix} -a & -3 \\ -1 & a \end{pmatrix}, \\
     % WLW^{-1} &\equiv \begin{pmatrix} 4a & 16 \\ 3 & -4a \end{pmatrix} \pmod{3}\\
       SWHW^{-1}S^{-1} &= \begin{pmatrix} 17 & -4a-12 \\ -6a+18 & -17 \end{pmatrix}.\end{split}
\end{equation}

Thus,
\[
\begin{pmatrix} 0 & -4a \\ -2a & 0 \end{pmatrix}, \;
\begin{pmatrix} a & 0 \\ 0 & -a \end{pmatrix}, \;
\begin{pmatrix} 0 & -4a \\ -6a & 0 \end{pmatrix} \in \operatorname{Span}_{\ZZ_3} \left( \Egp'' \cdot\mathfrak{sl}(2,\ZZ_3) \right).
\]
Therefore
\[
\mathfrak{sl}(2,\ZZ_3[\omega]) = \left\langle H, R , L, 
\begin{pmatrix} 0 & 0 \\ a & 0 \end{pmatrix}, \; 
\begin{pmatrix} a & 0 \\ 0 & -a \end{pmatrix}, \;
\begin{pmatrix} 0 & a \\ 0 & 0 \end{pmatrix} \right\rangle \subseteq \operatorname{Span}_{\ZZ_3} \left( \Egp'' \cdot \mathfrak{sl}(2,\ZZ_3) \right),
\]
implying that $\iota_3 = 0$ and $m_3 \leq 1 + 0 = 1$.

To finish, we check surjectivity modulo $3$, i.e. that the reduction of $\Egp''$ modulo 3 gives the reduction of $\SL(2, \ZZ[\om])$ modulo 3, whence $3$ is a good prime.

Observe that $\ZZ[\omega]/(3) \cong \FF_3[x]/(x^2)$, and $(\FF_3[x]/(x^2))^* = \{ a + xb : a \neq 0 \}$.  Similarly,
\[
\SL(2, \FF_3[x]/(x^2)) = \{ A + xB : A \in \SL(2, \FF_3), B \in \Mat(2, \FF_3) \}.
\]
Let $G$ be a subgroup of $\SL(2, \FF_3[x]/(x^2))$ containing $\SL(2, \FF_3)$.  Then consider the subgroup $G' \leq G$ given as follows, where we define the set $D_I$ in terms of $G'$:
\[
G' = \{ I + xB  \in G \} =: I + xD_I.
\]
Then $D_I$ is an $\FF_3$-vector space, since $(I + xA)(I + xB) = I + x(A + B)$.  Observe that for any $M \in \SL(2, \FF_3)$, we have
\[
M D_I M^{-1} \subseteq D_I,
\]
since $M G' M^{-1} \in G'$.
Therefore it is a vector space preserved by conjugation by $\SL(2, \FF_3)$.

Now specialize to the case that $G$ is the image of $\Egp''$ under reduction modulo $3$.
Reducing $W$, $SW$, $WT$, and $WS$, we have
\[
\begin{pmatrix} 2 & 2 \\ 0 & 1 \end{pmatrix}, \;
\begin{pmatrix} 2 & 2 \\ 2 & 0 \end{pmatrix}, \; 
\begin{pmatrix} 2 & 0 \\ 0 & 1 \end{pmatrix}, \; 
\begin{pmatrix} 1 & 2 \\ 1 & 1 \end{pmatrix} \in D_I.
\]
These four matrices are sufficient to generate $\Mat(2, \FF_3)$ by subspace closure.  Therefore, $D_I = \Mat(2, \FF_3)$.  Together with the fact that $G$ contains $\SL(2, \FF_3)$, this implies that $\Egp''$ reduces modulo $3$ to all of $\SL(2, \FF_3[x]/(x^2))$, so $3$ is a good prime.
\end{proof}

\begin{proposition}\label{prop:SAmod2power}
    Strong approximation for $\Egp''$ modulo $2$ has bad modulus $16$.
\end{proposition}

\begin{proof}
Using the methodology of \cite{FSZ19}, it suffices to compute $\Egp'' \cdot \mathfrak{sl}(2,\ZZ_2)$ and $\Gamma_1^T(2)/(2)$.  

Observe that for $k_2' \ge 1$ (in the notation of \cite{FSZ19}), we have
\[
\Gamma_1^T(2)(2^1) / \Gamma_1^T(2)(2^{k_2'})
\cong 
\SL(2,\ZZ)(2^1) / \SL(2,\ZZ)(2^{k_2'}),
\]
simply because the conditions specifying $\Gamma_1^T(2)$ are given modulo $2$ (with no further conditions modulo $4$). Therefore $m_2' = 1$ in the notation of \cite{FSZ19}.  

Let us now compute $\Egp'' \cdot \mathfrak{sl}(2,\ZZ_2)$, with the goal to determine the smallest $\iota_2\geq 0$ for which
\[
2^{\iota_2}\mathfrak{sl}(2,\ZZ_2[\omega]) \subseteq \operatorname{Span}_{\ZZ_2} \left( \Egp'' \cdot \mathfrak{sl}(2,\ZZ_2) \right).
\]
Recall that $\mathfrak{sl}(2,\ZZ)$ is generated by $H,R,L$, as in Proposition~\ref{prop:SAmod3}, Equation~\eqref{eqn:HRL}.

Using $W$ as in \eqref{eqn:W}, 
%\[ WRONG
%W := VSV = \begin{pmatrix}
%    -2a-1 & 4a - 4 \\
%    -2 & 2a+3
%\end{pmatrix}
%\]
amongst the elements of $\Egp'' \cdot \mathfrak{sl}(2,\ZZ_2)$ are those displayed in \eqref{eqn:3sl}, now considered in $\ZZ_2[\om]$.
Thus,
\[
\begin{pmatrix} 0 & 8a \\ 0 & 0 \end{pmatrix}, \;
\begin{pmatrix} a & 0 \\ 0 & -a \end{pmatrix}, \;
\begin{pmatrix} 0 & 0 \\ 4a & 0 \end{pmatrix} \in \operatorname{Span}_{\ZZ_2} \left\langle \Egp'' \cdot \mathfrak{sl}(2,\ZZ_2) \right\rangle.
\]
Therefore
\[
2^4 \mathfrak{sl}(2,\ZZ_2[\omega]) = \left\langle 16H, 16R, 16L, 
\begin{pmatrix} 0 & 0 \\ 8a & 0 \end{pmatrix}, \; 
\begin{pmatrix} 8a & 0 \\ 0 & -8a \end{pmatrix}, \;
\begin{pmatrix} 0 & 8a \\ 0 & 0 \end{pmatrix} \right\rangle \subseteq \operatorname{Span}_{\ZZ_2} \left( \Egp'' \cdot \mathfrak{sl}(2,\ZZ_2) \right).
\]
Thus $\iota_2 \leq 4$, and $m_2 \leq 5$. Therefore we have at most bad modulus $32$.

Now we ask about the fibre above $I$ when lifting from mod $2^k$ to mod $2^{k+1}$.  Let $M \in \Mat(2, \ZZ_2[\omega])$. We have (evaluating the characteristic polynomial of $-M$ at $2^{-k}$), that
\[
\det(I + 2^k M) = 1 + 2^k \Tr(M) + 2^{2k} \det(M).
\]
Therefore the lifts having determinant $1$ modulo $2^{k+1}$, at least when $k \ge 1$, are given by the condition $\Tr(M) \equiv 0 \pmod{2}$.  That is 1/4 of the total $4^4$ matrices, i.e. $2^6 = 64$ determinant $1$ lifts.

Now we turn instead to our subgroup $\Egp''$.  This is generated by determinant one words in $V$ and $G := \Gamma_1^T(2)$.  These can all be generated by products from $V G V$ and $G$.

We compute $VgV$ for some $g \in G$:
\[
\begin{pmatrix}
    -1 & 2\omega \\ 0 & 1 
\end{pmatrix}
\begin{pmatrix}
    A & B \\ C & D 
\end{pmatrix}
\begin{pmatrix}
    -1 & 2\omega \\ 0 & 1 
\end{pmatrix}
=
\begin{pmatrix}
   A - 2\omega C & -B + 2\omega(D-A) + 4 \omega^2 C \\
   -C & D + 2\omega C
\end{pmatrix},
\]
where $2\omega = 1+a$, 
which demonstrates that the upper-right entry lies in $2\ZZ[2\omega] = 2\ZZ[\sqrt{-3}]$, and the entries overall lie in $\ZZ[2\omega]$.   Observe that all of $\Gamma_1^T(2)$ is of this same form, therefore all products from $\Gamma_1^T(2)$ and $V \Gamma_1^T(2) V$ are of this form.  This shows that the lifting process from $2$ to $4$ is indeed obstructed, and we have at least bad modulus 4.

To close the gap between modulus 4 and modulus 32, we ran a computation in Pari/GP to check whether all 64 lifts of the identity exist at each level.  We found that the bad modulus is 16.  This computation can be verified with the public code in \cite{GHEisenstein}, in particular, the file ``strongapprox.gp''.
\end{proof}

Combining these results, we have:

\begin{theorem}\label{thm:SA}
    The group $\Egp''$ has bad modulus $16$ for strong approximation.  Consequently, any congruence conditions that hold for the curvatures of a primitive Eisenstein packing or a moiety can be described modulo $16$.
\end{theorem}

This allows us to prove density one.

\begin{proof}[Proof of Theorem~\ref{thm:densityone}]
    We reference \cite[p.1121]{FSZ19}. By Theorem~\ref{thm:hypotheses}, $\Egp''$ is an infinite-covolume, geometrically finite, Zariski-dense, familial Kleinian group in $\PSL(2,\ZZ[\omega])$.  The circles $\Dcir_i$, $i=1,2,3,4$ are images of $\widehat{\RR}$ under elements of $\PSL(2,\QQ[\omega])$.  The orbits $\Egp'' \cdot \Dcir_2$ and $\Egp'' \cdot \Dcir_4$ satisfy the hypotheses of \cite[Theorem 1.6]{FSZ19}, since $\Dcir_2$ and $\Dcir_4$ are tangent to the real line.  Although $\Dcir_1$ and $\Dcir_3$ are not, $\Egp'' \cdot \Dcir_1 = V \Egp'' \cdot \Dcir_2$ and $\Egp'' \cdot \Dcir_3 = V \Egp'' \cdot \Dcir_4$, so that any Eisenstein packing is the union of four orbits of the form $M \Egp'' \Dcir_i$, $i \in \{2,4\}$, $M \in \PSL(2,\QQ[\omega])$.  Hence it is a union of four packings in the Fuchs-Stange-Zhang sense, and any moiety is the union of two by Proposition~\ref{prop:moietyorbit}. 
    We now apply \cite[Corollary 1.4 and Theorem 1.6]{FSZ19}.
\end{proof}

\section{Reciprocity obstructions}\label{sec:reciprocity}

We adapt the strategy of \cite{HKRS23}, which has also been subsequently used by \cite{WhiteheadEtAlReciprocity}. For each circle $\cir$, using the associated quadratic form, we define a symbol $\chi_2(\cir)=\pm 1$, which governs the reciprocity behaviour of the curvature with respect to the curvatures of tangent circles. We are then able to prove that this symbol is either always preserved or always swaps across a coprime tangency. Since the graph of coprime tangencies is connected and bipartite, this implies that we have a well-defined quadratic invariant on either the whole packing, or each of the two bipartite halves. In most of these cases, we derive reciprocity obstructions.

\subsection{The symbol $\chi_2$}

For now, forget the definition of $\chi_2$ from Definition~\ref{def:chi2fromtangentcircle}. We will recover this expression later in Proposition~\ref{prop:rhofromtangentcircle}.

Consider a circle $\cir$ of curvature $n$ in a primitive Eisenstein circle packing, and complete it to an Eisenstein quadruple in standard position $q=(n, b, c, d)$. This is associated to the primitive first-odd binary quadratic form $\phi(q)(x, y)=Ax^2+Bxy+Cy^2$ (Proposition~\ref{prop:eisensteinpackingbijectionqf}).

\begin{definition}
    If $n\not\equiv 4\pmod{8}$, define $\rho=\rho(\cir)$ to be any positive integer that is equivalent modulo $n$ to a value of $\phi(q)(x, y)$, where $x$ is odd, $y$ is even, $\gcd(x, y)=1$, and $\gcd(\rho, n)=1$ (which must exist). If $n\equiv 4\pmod{8}$, repeat this definition, except we only allow integers that are equivalent modulo $2n$ to a value of $\phi(q)(x, y)$ under the same conditions. 
\end{definition}

This is analogous to the definition of $\rho$ from \cite[Proposition 4.1]{HKRS23}. If $n=0$, we necessarily have $\rho=1$. Otherwise, $\rho$ captures the modulo $n$ quadratic residuosity of $\phi(q)$.

\begin{proposition}\label{prop:rhowelldefined}
    Taken modulo $n\neq 0$, all the possible $\rho$'s are in the same multiplicative coset of the squares. Additionally, if $n\equiv 4\pmod{8}$, $\rho$ is uniquely determined modulo 8.
\end{proposition}
\begin{proof}
    By Proposition~\ref{prop:quadraticformtangentcircles}, the ambiguity in completing $\cir$ into an Eisenstein quadruple in standard position corresponds to the $\Gamma_0^{\PGL}(2)-$orbit of $\phi(q)$. The values of these forms on coprime odd $x$ and even $y$ are all the same, hence this set is well-defined.

    By taking $\rho$ to be one such value, we can make a $\Gamma_0^{\PGL}(2)$ change of coordinates to assume that $\phi(q)=[\rho, B, C]$. In that case, we find that (note that if $n$ is even, $B$ is even)
    \[\phi(q)(x, y)=\rho x^2+Bxy+Cy^2=\rho\left(x+\frac{B}{2\rho}y\right)^2+\frac{4\rho C-B^2}{4\rho}y^2=\rho\left(x+\frac{B}{2\rho}y\right)^2+\frac{3n^2}{4\rho}y^2.\]
    Since $y$ is even and $\rho$ is coprime to $n$, this is equivalent to $\rho(x+(B/2\rho)y)^2\pmod{n}$, giving the first result.

    If $n\equiv 4\pmod{8}$, then $x+(B/2\rho)y$ is necessarily odd and $8\mid \frac{3n^2}{4\rho}y^2$, hence $\phi(q)(x, y)\equiv \rho\pmod{8}$, as claimed.
\end{proof}

With this setup in hand, we can define $\chi_2$ on individual circles.

\begin{definition}
    Let $\cir$ be a circle of curvature $n\geq 0$ in a primitive Eisenstein circle packing of type $(3, t)$. Define $\chi_2(\cir)\in\{\pm 1\}$ as
    \[\chi_2(\cir):=\begin{cases}
        \kron{2\rho}{n} & \text{if $n$ is odd;}\\
        \kron{\rho+n}{n/2} & \text{if $n$ is even and $t=1$;}\\
        -\kron{-\rho+n}{n/2} & \text{if $n$ is even and $t=3$;}
    \end{cases}\]
    If $n<0$, we retain this definition, except we insist that $\rho>|n|$ if $n$ is even and $t=1$.
\end{definition}

\begin{remark}
    The case of $n<0$ corresponds to at most one circle in each packing. The only issue it causes in the above definitions is $\kron{x}{-1}=\text{sign}(x)$, hence we need to worry about changing the sign of the numerator. In any case, going forward, for simplicity we will generally assume that all curvatures are positive, though the cases for $n\leq 0$ can be checked with minimal extra effort.
\end{remark}

First, we verify that $\chi_2$ is well-defined.

\begin{proposition}
    The value of $\chi_2(\cir)$ is well defined.
\end{proposition}
\begin{proof}
    If $n\not\equiv 2\pmod{4}$ is positive, then the Kronecker symbol $\kron{\ast}{n}$ is well defined modulo $n$. Furthermore, if $a$ is coprime to $n$, then $\kron{a^2b}{n}=\kron{b}{n}$. By Proposition~\ref{prop:rhowelldefined}, this demonstrates that $\chi_2$ is well-defined if $n$ is odd.

    If $n\equiv 0,2,6\pmod{8}$, then $n/2\not\equiv 2\pmod{4}$, and again, $\chi_2$ is well-defined. Finally, assume that $n\equiv 4\pmod{8}$. It suffices to show that
    \[\kron{\rho+n}{n/2}=\kron{\rho+n}{2}\kron{\rho+n}{n/4},\quad \kron{-\rho+n}{n/2}=\kron{-\rho+n}{2}\kron{-\rho+n}{n/4}\]
    are well-defined. By Proposition~\ref{prop:rhowelldefined}, $\rho$ is uniquely determined modulo 8, hence $\kron{\pm\rho+n}{2}$ is determined. We also have that $n/4$ is odd, so $\kron{\pm\rho+n}{n/4}$ is also determined, as desired.
\end{proof}

When considering a circle in an Eisenstein packing, there is no need to compute the corresponding quadratic form to determine $\rho$ or $\chi_2$! Indeed, we can use Proposition~\ref{prop:curvaturesoftangentcircles} to relate it to tangent circles.

\begin{proposition}\label{prop:rhofromtangentcircle}
    Let $\cir$ be a circle of curvature $n$ in a primitive Eisenstein circle packing, and let it be tangent to a circle of coprime curvature $b$. If $n$ and $b$ are odd, we can take $\rho=\frac{b+n}{2}$. Otherwise, we can take $\rho=b+n$. In particular,
    \[\chi_2(\cir):=\begin{cases}
        \kron{b}{n} & \text{if $n$ and $b$ are odd;}\\
        \kron{2b}{n} & \text{if $n$ is odd and $b$ is even;}\\
        \kron{b}{n/2} & \text{if $n$ is even and $t=1$;}\\
        -\kron{-b}{n/2} & \text{if $n$ is even and $t=3$;}
    \end{cases}\]
\end{proposition}
\begin{proof}
    Let $\cir$ be associated to the quadratic form $Q$. If $n$ and $b$ have opposite parity, by Proposition~\ref{prop:curvaturesoftangentcircles}, there exists coprime $x,y$ with $y$ even for which $b=Q(x, y)-n$. Thus $Q(x, y)=b+n$ is coprime to $n$, and is a valid choice for $\rho$.

    If $n$ and $b$ have the same parity, they are necessarily odd, as they are coprime by assumption. In this case, by Proposition~\ref{prop:curvaturesoftangentcircles}, $b=2Q(x, y)-n$ for coprime $x,y$ with $y$ odd. Since we are working modulo the odd number $n$, the parities of $x,y$ are irrelevant, and we can find a coprime pair $(x', y')\equiv (x, y)\pmod{n}$ where $x'$ is odd and $y'$ is even. Then, $b\equiv 2Q(x', y')\pmod{n}$, completing the claim.

    The resulting expressions for $\chi_2(\cir)$ follow from using this value of $\rho$ in the definition of $\chi_2$ and basic properties of the Kronecker symbol.
\end{proof}

This proposition gives the definition of $\chi_2$ found in the introduction, Definition~\ref{def:chi2fromtangentcircle}.

\subsection{Propagation of $\chi_2$.}
Now, we study how the value of $\chi_2$ propagates through circle tangencies, culminating in the proof that it is constant in type $(3, 1)$, and alternating in type $(3, 3)$. The first step is to confirm our description of the moieties of the packing, which follows Definition~\ref{def:moiety}.

\begin{proposition}
    Let $W=(\cir_1, \cir_2, \cir_3, \cir_4)$ be a 4-wheel, set $E_1$ to be all circles that appear in the first or third position in the orbit $W\cdot \Egp$, and set $E_2$ to be all circles that appear in the second or fourth position. Then, $E_1$ and $E_2$ form the two moieties of the packing.
\end{proposition}

\begin{proof}
    It suffices to show that no circles in $E_i$ are tangent. The Cayley graph for $\Egp$ on the generators $S_1, S_2, S_3, S_4$ is in Figure~\ref{fig:cayley}, and as elucidated upon in Section~\ref{sec:eispackingalgebra}, every circle corresponds to a unique vertex. In particular, as we generate the packing beginning at $W$, say we are at the 4-wheel $(\cir_1', \cir_2', \cir_3', \cir_4')$, and form a new circle with the move $S_i$. Then, the new circle is tangent to only $\cir_{i-1}'$ and $\cir_{i+1}'$, both of which are in the opposite moiety. Running this process forever generates the whole circle packing, proving the result.
\end{proof}

Next is the proof that $\chi_2$ alternates across tangencies in type $(3, 3)$, and is constant in type $(3, 1)$. In both cases, it is therefore well-defined on each moiety of the packing.

\begin{proof}[Proof of Proposition~\ref{prop:chi2invariance}]
    Consider tangent circles $\cir_1$ and $\cir_2$ of coprime curvatures $a$ and $b$ respectively. We will demonstrate that $\chi_2(\cir_1)=\chi_2(\cir_2)$ if the type is $(3, 1)$ and $\chi_2(\cir_1)=-\chi_2(\cir_2)$ otherwise. By combining this with Proposition~\ref{prop:coprimetangencypath}, the result will follow.

    First, assume the type is $(3, 1)$, and assume that $a,b$ are odd. Thus, $a\equiv b\equiv 1\pmod{4}$, and from Proposition~\ref{prop:rhofromtangentcircle},
    \[\chi_2(\cir_1)=\kron{b}{a}=\kron{a}{b}=\chi_2(\cir_2),\]
    where the middle step is quadratic reciprocity.

    Next, assume the type is $(3, 1)$, take $a$ to be odd, and $b$ to be even. By Proposition~\ref{prop:rhofromtangentcircle},
    \[\chi_2(\cir_2)=\kron{a}{b/2}=\kron{b/2}{a}=\kron{2b}{a}=\chi_2(\cir_1),\]
    as desired.

    It remains to consider type $(3, 3)$. If $a,b$ are both odd, then they are both $3\pmod{4}$, hence
    \[\chi_2(\cir_1)=\kron{b}{a}=-\kron{a}{b}=-\chi_2(\cir_2).\]
    If $a$ is odd and $b$ is even, let the odd part of $b$ be $b_o$, so that by quadratic reciprocity,
    \[\kron{a}{b/2}\kron{b/2}{a}=(-1)^{(a-1)(b_o-1)/4}=(-1)^{(b_o-1)/2}.\]
    By Proposition~\ref{prop:rhofromtangentcircle} we have
    \[\chi_2(\cir_2)=-\kron{-a}{b/2}=-(-1)^{(b_o-1)/2}\kron{a}{b/2}=-\kron{b/2}{a}=-\kron{2b}{a}=-\chi_2(\cir_1),\]
    completing the proof.
\end{proof}

The final piece is the existence of coprime tangency paths between any pairs of circles. In the Apollonian case (\cite[Corollary 4.7]{HKRS23}), given a pair of tangent circles whose curvatures share a non-trivial factor, we considered the family of circles tangent to both. Their curvatures were values of a quadratic form, and it was shown that one of them will be coprime to our two original circle curvatures.

In the Eisenstein case, we need a modification, as there are \emph{no} triples of mutually tangent circles. Instead, we must follow a procedure similar to \cite{WhiteheadEtAlReciprocity}, where we pass to at least two auxiliary circles to make a coprime path. To begin, we require a standard lemma about quadratics.

\begin{lemma}\label{lem:quadraticoddprime}
    Let $f(x)=Ax^2+Bx+C\in\ZZ[x]$. If $p$ is an odd prime with $p\nmid \gcd(A, B, C)$, then there exists a residue $r$ such that if $x\equiv r\pmod{p}$, then $p\nmid f(x)$.
\end{lemma}
\begin{proof}
    As divisibility by $p$ is preserved across the $x\pmod{p}$ residue class, it suffices to prove that at least one integer $x$ satisfies $f(x)$ is coprime to $p$. If not, then $p\mid f(0)=C$, $p\mid f(1)=A+B+C$, and $p\mid f(-1)=A-B+C$, whence $p\mid A,B,C$, contradiction.
\end{proof}

Next, we make the coprime path, but ignore the prime 2, zero curvature circles, and equal curvature circles.

\begin{lemma}\label{lem:onetangencyreplacement}
    Let $\cir$ and $\cir'$ be tangent circles with distinct and non-zero curvatures in a primitive Eisenstein circle packing. Then, there exist circles $\cir_1,\cir_2$ in the packing such that
    \begin{itemize}
        \item $(\cir, \cir_1, \cir_2, \cir')$ is a 4-wheel, hence $\cir-\cir_1-\cir_2-\cir'$ is a path of tangent circles;
        \item The curvatures of the four circles are distinct and non-zero;
        \item In each of the three pairs of consecutive tangent circles $(\cir, \cir_1)$, $(\cir_1, \cir_2)$, and $(\cir_2, \cir')$, their curvatures share no odd prime factors.
    \end{itemize}
\end{lemma}
\begin{proof}
    Call a pair of integers $(x, y)$ \emph{odd-coprime} if $\gcd(x, y)$ is a power of two.

    Complete the tangent circles $\cir,\cir'$ into any 4-wheel in the packing, say $(\cir, \cir_1, \cir_2, \cir')$, with curvatures $(a,b,c,d)$.  Applying words of the form $(S_3S_2)^n$ will alternately swap out $\cir_2$ and $\cir_3$, and run through an infinite family of 4-wheels containing $\cir,\cir'$ as the first and last entries. In this way, we replace $(a,b_0 := b,c_0:=c,d)$ with $(a,b_n,c_n,d)$. We claim that for some value of $n$, the pairs $(a, b_n)$, $(b_n, c_n)$, and $(c_n, d)$ are all odd-coprime.

    Compute
    \[
        S_3 S_2 = \begin{pmatrix}
            1 & 2  & 0  & 0 \\
            0 & 3  & 2  & 0 \\
            0 & -2 & -1 & 0 \\
            0 & 4  & 2  & 1
    \end{pmatrix},
    \]
    whence the dynamical system obtained is
    \[\begin{pmatrix}b_n \\ c_n\end{pmatrix}=
    \begin{pmatrix} 3 & -2 \\ 2 & -1 \end{pmatrix}
        \begin{pmatrix}
        b_{n-1} \\ c_{n-1}
    \end{pmatrix} +
        \begin{pmatrix}
        2a + 4d \\ 2d
    \end{pmatrix}.
    \]
    Writing
    \[
  \begin{pmatrix} 3 & -2 \\ 2 & -1 \end{pmatrix} = A
= I + M, \quad M =
 \begin{pmatrix} 2 & -2 \\ 2 & -2 \end{pmatrix},\quad B=\begin{pmatrix}
     2a+4d\\2d
 \end{pmatrix},
    \]
    We obtain (for $n\geq 0$)
    \[\begin{pmatrix}b_n \\ c_n\end{pmatrix}=
    (I+M)^n\begin{pmatrix}b_0 \\ c_0\end{pmatrix}
    + 
    \sum_{i=0}^{n-1} (I+M)^iB
    \]
    Noticing that $M^2=0$, this becomes
\begin{align*}
    \begin{pmatrix}b_n \\ c_n\end{pmatrix}
    &=(I+nM)\begin{pmatrix}b_0 \\ c_0\end{pmatrix}
    + \sum_{i=0}^{n-1} \left(I+iM\right)         \begin{pmatrix}
        2a + 4d \\ 2d
    \end{pmatrix} \\
     &=(I+nM)\begin{pmatrix}b_0 \\ c_0\end{pmatrix}
    + n\begin{pmatrix}
        2a + 4d \\ 2d
    \end{pmatrix} + \frac{n(n-1)}{2}  M       \begin{pmatrix}
        2a + 4d \\ 2d
    \end{pmatrix} \\
    &= n^2 \begin{pmatrix} 2a+2d \\ 2a+2d \end{pmatrix}
    + n \begin{pmatrix} 2d + 2b_0 - 2 c_0 \\ -2a + 2b_0 - 2 c_0\end{pmatrix}
    +      \begin{pmatrix}
        b_0 \\ c_0
    \end{pmatrix}.
    \end{align*}

    Let $g = \gcd(a,d)$. Recall the Eisenstein equation, which gives $a^2+b_n^2+c_n^2+d^2=2(a+c_n)(b_n+d)$. Let $p$ be an odd prime, and if $p\mid g$, we see that $b_n^2+c_n^2\equiv 2b_nc_n\pmod{p}$, hence $(b_n-c_n)^2\equiv 0\pmod{p}$, i.e. $p\mid b_n-c_n$. However, if $p\mid b_n$ or $c_n$ for any $n$, then $p\mid b_n,c_n$, and so we have a non-primitive Eisenstein quadruple, contradiction. Thus, $b_n$ and $c_n$ never share prime factors with $g$.

    We now define three finite sets of primes:
    \begin{itemize}
        \item $P_1$ is the set of prime divisors of $a$ that do not divide $g$;
        \item $P_2$ is the set of prime divisors of $d$ that do not divide $g$;
        \item $P_3$ is the set of prime divisors of $a-d$.
    \end{itemize}
    The above computation shows that $P_1$ and $P_2$ are disjoint, and if $p\in P_3\cap(P_1\cup P_2)$, then $p\mid g$, contradiction. Thus, the three sets are disjoint.

    The equation for $b_n$ is quadratic, and $p$ dividing all of its coefficients is equivalent to
    \[p\mid b_0,\quad a\equiv -d\pmod{p},\quad c_0\equiv d\pmod{p}.\]
    From the Eisenstein equation, this implies that
    \[3a^2\equiv a^2+b_0^2+c_0^2+d^2\equiv 2(a+c_0)(b_0+d)\equiv 0\pmod{p},\]
    so $p=3$ or $p\mid a$. In particular, if $p\neq 3$ or $(a, b_0, c_0, d)\not\equiv (a, 0, -a, -a)\pmod{3}$, then $p$ does not divide the $\gcd$ of the coefficients, and Lemma~\ref{lem:quadraticoddprime} applies.

    Analogously, for the quadratic equation giving $c_n$, we have that $p$ does not divide the $\gcd$ of the coefficients as long as $p\neq 3$ or $(a, b_0, c_0, d_0)\not\equiv (a, a, 0, -a)\pmod{3}$.

    Now, for our three sets of primes, do the following:
    \begin{itemize}
        \item For $p\in P_1\cup P_3$, if $p\neq 3$, then Lemma~\ref{lem:quadraticoddprime} applies with $b_n$. If $p=3$, then since $p\mid a$ or $p\mid a-d$, we must have $(a, b_0, c_0, d_0)\not\equiv (a, 0, -a, -a)\pmod{3}$, as otherwise the quadruple would not be primitive at 3. In particular, Lemma~\ref{lem:quadraticoddprime} always applies, and guarantees a residue $r_p$ so that if $n\equiv r_p\pmod{p}$, then $b_n\not\equiv 0\pmod{p}$.
        \item For $p\in P_2$, apply the analogous result with $c_n$, obtaining a residue $r_p$ so that if $n\equiv r_p\pmod{p}$, then $c_n\not\equiv 0\pmod{p}$.
    \end{itemize}

    As our three sets of primes are disjoint, by the Chinese remainder theorem, we can combine the above residues into one class $r\pmod{N}$, where $N$ is the product of primes in $P_1\cup P_2\cup P_3$. We claim that any $n\equiv r\pmod{N}$ will satisfy the odd-coprimality on pairs.

    For $(a, b_n)$, let $p$ be an odd prime divisor of $a$. If $p\mid g$, then we already know that $p\nmid b_n$, hence $p\nmid\gcd(a, b_n)$. Otherwise, $p\in P_1$, and by our construction, $p\nmid b_n$, so again, $p\nmid \gcd(a, b_n)$, proving that this is a power of 2.

    The argument for $(c_n, d)$ is analogous, so it remains to check $(b_n, c_n)$. Let $p\mid b_n, c_n$, and by the Eisenstein equation, we have
    \[a^2+d^2\equiv a^2+b_n^2+c_n^2+d^2\equiv 2(a+c_n)(b_n+d)\equiv 2ad\pmod{p},\]
    whence $p\mid (a-d)^2$, so $p\mid a-d$. Therefore $p\in P_3$, and by our construction, $p\nmid b_n$, so $\gcd(b_n, c_n)$ is a power of 2, as claimed.

    To finish, we just need to note that the $n^2$ coefficient in $b_n$ and $c_n$ is $2a+2d$, which is non-zero. (Indeed, otherwise, $a=-d\neq 0$, but there is a unique smallest absolute curvature in a bounded packing.) Thus, for some large $n\equiv r\pmod{N}$, the four curvatures $a,b_n,c_n,d$ must be distinct and non-zero, the final requirement.
\end{proof}

The reason we assumed that our starting circles had non-zero curvature was that otherwise, $P_1$ or $P_2$ would be infinite. Similarly, if they had equal curvature, then $P_3$ is infinite. First, we rectify the zero curvature issue.

\begin{lemma}\label{lem:coprimetangencypath0curv}
    Let $\cir$ and $\cir'$ be tangent circles in a primitive Eisenstein circle packing, where $\cir$ has curvature zero. Then, there is a tangency path between $\cir$ and $\cir'$ where consecutive circles have coprime curvatures.
\end{lemma}
\begin{proof}
    Since we have a curvature zero circle, we must be in the Eisenstein strip packing, Figure~\ref{fig:strip}. If $\cir'$ is also of curvature $0$, then we can traverse across the two curvature 1 circles in the root 4-wheel of curvatures $(0, 0, 1, 1)$. Otherwise, assume that $(\cir, \cir_1,\cir_2,\cir_3)$ is a 4-wheel with curvatures $(0, b, c, d)$. By the Eisenstein equation, $b^2+c^2+d^2\equiv 2c(b+d)$. So, if $p$ is a prime that divides either $(b, c)$ or $(c, d)$, then it divides $b,c,d$, whence $(0, b, c, d)$ is not primitive, contradiction. Therefore it follows that $\gcd(b, c)=\gcd(c, d)=1$.

    Thus, we can pick any 4-wheel $(\cir, \cir_1,\cir_2,\cir')$, and reduce it, eventually obtaining a quadruple of circles with curvatures being a permutation of $(0, 0, 1, 1)$. However, along the way, start with $\cir'$. Any time we are swapping this circle out, move to the tangent circle $\cir_2$, which must have coprime curvature. Repeat this process, obtaining a coprime tangency path from $\cir_2$ to one of the other circles in the base quadruple. We can then move directly to $\cir$ (if curvature 1), or through the two curvature 1 circles to $\cir$. In any case, we have found our tangency-connected path.
\end{proof}

Finally, we can turn any tangency-connected path into a coprime tangency-connected path.

\begin{proposition}\label{prop:coprimetangencypath}
    Let $\cir$ and $\cir'$ be circles in a primitive Eisenstein circle packing. Then, there is a tangency path between $\cir$ and $\cir'$ where consecutive circles have coprime curvatures.
\end{proposition}
\begin{proof}
    The tangency graph is connected, so to begin, pick any path. Consider any consecutive tangency $\Dcir-\Dcir'$ of curvatures $d, d'$ respectively, where $\gcd(d, d')>1$. We claim that we can replace this with a path $\Dcir-\Dcir_1-\Dcir_2-\cdots-\Dcir_n-\Dcir'$, where every pair of consecutive circles have coprime curvatures. If this is true, then repeating this replacement for all such pairs produces the final path.

    Assume that $d\neq d'$. If either are 0, we are done by Lemma~\ref{lem:coprimetangencypath0curv}. Otherwise, by Lemma~\ref{lem:onetangencyreplacement}, we can find circles $\Dcir_1$ and $\Dcir_2$ of curvatures $x, y$ which make a 4-wheel of curvatures $(d, x, y, d')$, where consecutive pairs are distinct, non-zero, and share no odd prime divisors. If $d$ and $d'$ are even, then by Lemma~\ref{lem:twoeventwoodd}, $x$ and $y$ are necessarily odd, whence $\gcd(d, x)=\gcd(x, y)=\gcd(y, d')=1$.

    If $d$ and $d'$ are odd, then $x$ and $y$ are even, so $\gcd(d, x)=\gcd(y, d')=1$, but $\gcd(x, y)$ is an even power of 2. In this case, we can just repeat the argument for circles $\Dcir_1$ and $\Dcir_2$, ending up with a 6 circle path between $\Dcir$ and $\Dcir'$.

    Similarly, if one of $d, d'$ is even and the other is odd, we can assume that $d$ is even and $d'$ is odd. Therefore $x$ is even and $y$ is odd, so $\gcd(x, y)=\gcd(y, d')=1$. This time $\gcd(d, x)$ is an even power of 2, so repeat the argument on $\Dcir$ and $\Dcir_1$ to get a valid 6 circle coprime path between $\Dcir$ and $\Dcir'$.

    The final case is if $d=d'\neq 0$. Pick a 4-wheel $(\Dcir, \Dcir_1,\Dcir_2,\Dcir')$, and by alternately applying swaps $S_2$ and $S_3$ some number of times, we can reduce to the case where each pair of consecutive circles has distinct curvature. Repeat the above arguments for each pair to derive the final result.
\end{proof}

\subsection{The obstructions}

We settle Theorem~\ref{thm:reciprocity} through a series of lemmas. Throughout this section, let $\cir$ denote a circle of curvature $c$ in a primitive Eisenstein packing of type $(3, t)$. Pick a tangent circle with coprime curvature $b$, which can always be taken to be odd. Recall the definition of $\chi_2(\cir)$ involving $(b, c, t)$, Definition~\ref{def:chi2fromtangentcircle}.

\begin{lemma}\label{lem:teq1chi2eqm1}
    If $t=1$ and $\chi_2(\cir)=-1$, then $c\neq (2n+1)^2, 2n^2, 6n^2$ for any $n\in\ZZ$, despite these families being entirely admissible.
\end{lemma}
\begin{proof}
    If $c=(2n+1)^2$, then
    \[-1=\chi_2(\cir)=\kron{b}{(2n+1)^2}=1,\]
    contradiction.

    If $c=2n^2$, then
    \[-1=\chi_2(\cir)=\kron{b}{2n^2/2}=\kron{b}{n^2}=1,\]
    another contradiction.

    If $c=6n^2$, then as $b$ is coprime to $c$, it is coprime to 6, hence $b\equiv 1, 5\pmod{6}$. However, by Proposition~\ref{prop:tangentcurvaturemod6sum}, $b+c\not\equiv 5\pmod{6}$, hence $b\equiv 1\pmod{6}$. Thus,
    \[-1=\chi_2(\cir)=\kron{b}{6n^2/2}=\kron{b}{3}=1,\]
    a final contradiction.
\end{proof}

\begin{lemma}\label{lem:teq3chi2eq1}
    If $t=3$ and $\chi_2(\cir)=1$, then $c\neq 2n^2, 3(2n+1)^2$ for any $n\in\ZZ$, despite these families being entirely admissible.
\end{lemma}
\begin{proof}
    If $c=2n^2$, then
    \[1=\chi_2(\cir)=-\kron{-b}{2n^2/2}=-\kron{-b}{n^2}=-1,\]
    a contradiction.

    If $c=3(2n+1)^2$, then $c\equiv 3\pmod{6}$. Since $b$ is odd and coprime to $3(2n+1)^2$, $b\equiv 1, 5\pmod{6}$. By Proposition~\ref{prop:tangentcurvaturemod6sum}, $b+c\not\equiv 4\pmod{6}$, hence $b\equiv 5\pmod{6}$. Therefore,
    \[1=\chi_2(\cir)=\kron{b}{3(2n+1)^2}=\kron{b}{3}=-1,\]
    another contradiction.
\end{proof}

\begin{lemma}\label{lem:teq3chi2eqm1}
    If $t=3$ and $\chi_2(\cir)=-1$, then $c\neq 6n^2$ for any $n\in\ZZ$, despite this family being entirely admissible.
\end{lemma}
\begin{proof}
    If $c=6n^2$, as with the above two lemmas, $b\equiv 1, 5\pmod{6}$, hence $b\equiv 1\pmod{6}$ as $b+c\not\equiv 5\pmod{6}$. Therefore
    \[-1=\chi_2(\cir)=-\kron{-b}{6n^2/2}=-\kron{-b}{3}=1,\]
    contradiction.
\end{proof}

It is clear that Theorem~\ref{thm:reciprocity} follows from Lemmas~\ref{lem:teq1chi2eqm1}, \ref{lem:teq3chi2eq1}, \ref{lem:teq3chi2eqm1}.

\section{Computational results on sporadic curvatures}\label{sec:computation}

\tablestretch{
\begin{table}[htb]
	\caption{Sporadic curvature data for small quadruples.}\label{table:conjecturedata}
	\begin{tabular}{|l|c|c|c|c|c|c|c|} 
		\hline
		Root quadruple      & Moiety &  Type        & Obstructions & $N$              & $|\mathcal{S}_E(N)|$ &  $\max(\mathcal{S}_E(N))$ & $\approx\frac{N}{\max(\mathcal{S}_E(N))}$\\ \hline
		$(0, 0, 1, 1)$      & O/E  & $(3, 1, 1)$  &              & $10^{8}$         & 66                   & 12490                     & 8006.41\\
    	  $(-2, 4, 9, 5)$     & O    &              &              & $10^{8}$         & 3188                 & 1079526                   & 92.63\\
		                    & E    &              &              & $10^{8}$         & 2938                 & 993678                    & 100.64\\ 
		$(-6, 14, 25, 13)$  & O    &              &              & $4\cdot 10^{8}$  & 22719                & 9351538                   & 42.77\\
		                      & E    &              &              & $4\cdot 10^{8}$  & 22499                & 8171254                   & 48.95\\
		$(-7, 9, 52, 44)$   & O    &              &              & $4\cdot 10^{8}$  & 48655                & 23129156                  & 17.29\\ 
		                      & E    &              &              & $4\cdot 10^{8}$  & 49169                & 20342406                  & 19.66\\ \hline
		
        $(-3, 17, 4, 20)$   & O    & $(3, 1, -1)$ & $(2n+1)^2,2n^2,6n^2$ & $10^{9}$ & 14704                & 9814206                   & 101.89\\
		                    & E    &              &              & $10^{9}$         & 14256                & 4650662                   & 215.02\\
		$(-3, 5, 14, 10)$   & O    &              &              & $10^{9}$         & 2354                 & 426408                    & 2345.17\\
		                    & E    &              &              & $10^{9}$         & 2139                 & 447314                    & 2235.57\\
		$(-4, 28, 33, 5)$   & O    &              &              & $10^{9}$         & 45277                & 31740702                  & 31.51\\
		                    & E    &              &              & $10^{9}$         & 45184                & 19615334                  & 50.98\\
		$(-7, 13, 30, 20)$  & O    &              &              & $10^{9}$         & 33841                & 13236632                  & 75.55\\
		                    & E    &              &              & $10^{9}$         & 33157                & 11734136                  & 85.22\\ \hline
		
        $(-1, 3, 4, 2)$     & O    & $(3, 3, 1)$  & $2n^2, 3(2n+1)^2$ & $10^{8}$    & 265                  & 43406                     & 2303.83\\
		$(-2, 8, 11, 3)$    & O    &              &              & $4\cdot 10^{8}$  & 3349                 & 1412726                   & 283.14\\
		$(-4, 10, 15, 11)$  & E    &              &              & $10^{9}$         & 27095                & 10182338                  & 98.21\\
		$(-4, 12, 19, 7)$   & O    &              &              & $4\cdot 10^{8}$  & 29239                & 11266738                  & 35.50\\
		$(-5, 7, 30, 24)$   & E    &              &              & $4\cdot 10^{8}$  & 13855                & 4687696                   & 85.33\\
		$(-5, 11, 20, 12)$  & O    &              &              & $4\cdot 10^{8}$  & 11241                & 3903298                   & 102.48\\
		$(-6, 14, 23, 15)$  & E    &              &              & $4\cdot 10^{8}$  & 21928                & 7490626                   & 53.40\\
		$(-8, 14, 35, 27)$  & O    &              &              & $6\cdot 10^{9}$  & 280363               & 163242802                 & 36.76\\ \hline
		
        $(-1, 3, 4, 2)$     & E    & $(3, 3, -1)$ & $6n^2$       & $10^{8}$         & 267                  & 45846                     & 2181.22\\
		$(-2, 8, 11, 3)$    & E    &              &              & $4\cdot 10^{8}$  & 3556                 & 875106                    & 457.09\\
		$(-4, 10, 15, 11)$  & O    &              &              & $10^{9}$         & 27037                & 17277798                  & 57.88\\
		$(-4, 12, 19, 7)$   & E    &              &              & $4\cdot 10^{8}$  & 29021                & 12264182                  & 32.62\\
		$(-5, 7, 30, 24)$   & O    &              &              & $4\cdot 10^{8}$  & 13894                & 4464868                   & 89.59\\
		$(-5, 11, 20, 12)$  & E    &              &              & $4\cdot 10^{8}$  & 11349                & 3153514                   & 126.84\\
		$(-6, 14, 23, 15)$  & O    &              &              & $4\cdot 10^{8}$  & 22379                & 7337122                   & 54.52\\
		$(-8, 14, 35, 27)$  & E    &              &              & $6\cdot 10^{9}$  & 284218               & 315974568                 & 18.99\\ \hline
	\end{tabular}
\end{table}
}

There are four extended types possible. For every packing of extended type $(3, 1, \pm 1)$, we get two moieties: one for the odd index circles (1 and 3, denoted O), and one for the even index circles (2 and 4, denoted E). For each packing of type $(3, 3)$, one moiety has extended type $(3, 3, 1)$, and the other moiety has extended type $(3, 3, -1)$.

Overall, for the 48 root quadruples with curvatures at most 60, we have 96 sets of moieties, which divide into
\begin{itemize}
    \item 26 of extended type $(3, 1, 1)$ (with two being equal: in the packing $(0, 0, 1, 1)$, the even index and odd index moieties clearly contain the same sets of curvatures);
    \item 20 of extended type $(3, 1, -1)$;
    \item 25 of extended type $(3, 3, 1)$;
    \item 25 of extended type $(3, 3, -1)$.
\end{itemize}

For each moiety, we computed the sporadic set $\mathcal{S}_E(N)$ for $N=10^8$, and observed the final sporadic curvature found. If it was larger than $N/10$, we increased $N$ and recomputed. In each such case, we were able to succeed for some $N\leq 3\cdot 10^{10}$. A summary of these computations for 4 small packings of each extended type is found in Table~\ref{table:conjecturedata}. While this does not prove that $\mathcal{S}_E$ is finite, it is strong evidence towards this conjecture, and further suggests that we have found the exact sporadic set in the described cases.

\printbibliography

\end{document}